\documentclass[a4paper,10pt]{article}
\usepackage[top=2cm, bottom=2cm, left=3cm, right=2cm]{geometry}
\usepackage[latin1]{inputenc}
\usepackage{amsmath,amsfonts,stmaryrd}
\usepackage[none]{hyphenat}
\usepackage{amsthm}
\newtheorem{theorem}{Theorem}
\newtheorem{definition}{Definition}
\newtheorem{lemma}{Lemma}
\newtheorem{corollary}{Corollary}
\numberwithin{equation}{section} 

\begin{document}
\title{The Cauchy problem for a combustion model in porous media}
\author{J.~C. da Mota\footnotemark[2] \footnotemark[4]
\and M.~M. Santos\footnotemark[3]
\and R.~A. Santos\footnotemark[2] \footnotemark[5]
}
\date{\vspace{-5ex}}
\maketitle

\newcommand{\R}{\mathbb{R}}

\newtheorem{remark}{Remark}

\renewcommand{\thefootnote}{\fnsymbol{footnote}}

\footnotetext[2]{Departamento de Matem\'atica, IME-UFG (Instituto de Matem\'atica e Estat\'\i stica--
Universidade Federal de Goi\'as). \  
Cx. Postal 131, Campus II, 74001-970 Goi\^ania, GO, Brazil.\\
jesus@ufg.br, rasantos@ufg.br}
\footnotetext[3]{Departamento de Matem\'atica, IMECC-UNICAMP (Instituto de Matem\'atica, Estat\'\i stica
e Computa\c c\~ao Cient\'\i fica--Universidade Estadual de Campinas). \ 
Rua S\'ergio Buarque de Holanda, 651, Cidade Universit\'aria Zeferino Vaz, 
13083-859 Campinas, SP, Brazil.\\
msantos@ime.unicamp.br}
\footnotetext[4]{J.C. da Mota thanks FAPEG--Funda\c c\~ao de Amparo \`a Pesquisa do Estado de 
Goi\'as, Brazil, for partial support of this work, under grant \# 05/2012-Universal.}
\footnotetext[5]{R.A. Santos thanks CAPES--Coordena\c c\~ao de Aperfei\c coamento de Pessoal de 
N\'\i vel Superior, Brazil, for financial support through a scholarship during his doctorate at
IMECC-UNICAMP, Brazil, under project \# 23038.002308/2010-48-AUXPE 747/2010.}

\renewcommand{\thefootnote}{\arabic{footnote}}

\pagestyle{myheadings}
\thispagestyle{plain}
\markboth{J.~C da Mota, M.~M. Santos and R.~A. Santos}{Combustion model}

\begin{abstract} %
We prove the existence of a global solution to the Cauchy problem for a nonlinear 
reaction-diffusion system coupled with a system of ordinary differential equations. 
The system models the propagation of a combustion front in a porous medium with two layers, as
derived by J. C. da Mota and S. Schecter in \textit{Combustion fronts in a porous medium with 
two layers}, Journal of Dynamics and Differential Equations, \textbf{18}(3) (2006). For the particular case, when the fuel concentrations in both layers are known functions, the Cauchy problem was solved by J. C. da Mota and  M. M. Santos in \textit{An application of the monotone iterative method to a combustion problem in porous media}, Nonlinear Analysis: Real World
Application, \textbf{12} (2010). For the full system, in which the fuel concentrations are also unknown functions, we construct an iterative scheme that contains a sequence which converges to a
solution of the system, locally in time, under the conditions that the initial data are H\"older
continuous, bounded and nonnegative functions. We also show the existence of a global solution, if the initial date are additionally in the Lebesgue space $L^p$, for some $p\in (1,\infty)$. 
Our proof of the local existence relies on a careful analysis on the construction of the
fundamental solution for parabolic equations obtained by the parametrix method. 
In particular, we show the continuous dependence of the fundamental solution for parabolic
equations with respect to the coefficients of the equations. To obtain the global existence, 
we employ the {\lq\lq}method of auxiliary functions{\rq\rq} as used by O. A. Oleinik and S. N.
Kruzhkov in \textit{Quasi-linear second-order parabolic equations with many independent variables}, Russian Mathematical Surveys, \textbf{16}(5) (1961). Furthermore, for a broad class of 
reaction-diffusion systems we show that the non negative quadrant is a positively invariant region, and, as a consequence, that classical solutions of similar systems, with the reactions functions being non decreasing in one unknown and semi-lipscthitz continuous in the other, are bounded by lower and upper solutions for any positive time if so they are at time zero.  
\end{abstract}

\pagestyle{myheadings}
\thispagestyle{plain}
\markboth{ }{ }

\section{Introduction}

We are mainly concerned with a specific system of the type
\begin{equation}
\label{eq1-general}
(u_i)_t-\alpha_i(y_i){(u_i)}_{xx}+\beta_i(y_i){(u_i)}_x=
f_i(y_i,u_1,u_2), \quad\ x\in\mathbb{R}, \quad t>0
\end{equation}
for the unknowns $u_i,\, y_i$, with $i=1,2$, where $y_i$ satisfies an ordinary diferential equation
which can be solved depending on $u_i$, and $\alpha_i(y_i),\beta_i(y_i)$ are given functions of 
$y_i$, and $f_i(y_i,u_1,u_2)$ is a function (also given) of $y_i$, $u_1$ and $u_2$. For fixed 
$y_i$, the equations \eqref{eq1-general} are a system of parabolic equations for $u_1,u_2$ coupled 
by the function $f_i$. For the full system, in the unknowns $u_1,u_2,y_1,y_2$, since $y_i$ can be
expressed depending on $u_i$, our system can be writen in the unknowns $u_1, u_2$ only, but with 
coefficients depending in a peculiar way on $u_1, u_2$. In fact, the system we shall consider can 
be written in the form
\begin{equation}
\label{eq1-general-2}
\begin{array}{l}  (u_i)_t - a(x,\int_0^tf(u_i)d\tau)\ {(u_i)}_{xx}
 + b(x,\int_0^tf(u_i)d\tau)\ {(u_i)}_x\\
\ \ \ \ \ \ \ \ \ \ \ \ \ \ \ \ \ \ \ \ \ \ \
= F_i(x,u_1,u_2,\int_0^tf(u_i)d\tau)\end{array}
\end{equation}
for given functions $a$, $b$, $f$, and $F_i$. 

Specifically, the functions $\alpha_i$, $\beta_i$ and 
$f_i$ in \eqref{eq1-general} are given by
\begin{equation}
\label{functions}
\begin{array}{c}
\alpha_i(y_i)=\dfrac{\lambda_i}{a_i+b_iy_i}, \quad
\beta_i(y_i)=\dfrac{c_i}{a_i+b_iy_i} \quad \mbox{ and}\\
f_i(y_i,u_1,u_2)=\dfrac{b_iA_iu_i+d_i}{a_i+b_iy_i}y_if(u_i)+(-1)^iq\dfrac{u_1-u_2}{a_i+b_iy_i}
\end{array}
\end{equation}
where $f(u_i)$ is the {\lq\lq}Arrhenius function{\rq\rq} 
\begin{equation}
\label{f}
f(u_i)=\mbox{e}^{-\frac{E}{u_i}}, 
\end{equation}
being $E$ is a positive constant,
\footnote{We notice that the function $f(s)=\mbox{e}^{-\frac{E}{s}}$, $s\not=0$, can be extended by zero
continuously from $s>0$ to $s=0$. 
In fact, $\lim_{s\to0+}\frac{d^k f}{ds^k}=0$ for any $k=0,1,2,\cdots$. Despite the discontinuity
when $s\to0-$ ($\lim_{s\to0-}f(s)=\infty$), this will not cause problems in our analysis because
essentially we will deal only with non negative functions $u_i$, $i=1,2$, cf. theorems \ref{local}
and \ref{global}.} and $\lambda_i,a_i,b_i,c_i,d_i,A_i$, $i=1,2$, and $q$ are positive constants.
  
The unknown $y_i$ satisfies the ordinary differential equation
\begin{equation}
\label{yi ode}
(y_i)_t=-A_iy_if(u_i).
\end{equation}
Joint with equations \eqref{eq1-general} we add the initial data
\begin{equation}
\label{in cond ui}
u_i\big|_{t=0}=u_{i,0}
\end{equation}
and
\begin{equation}
\label{in cond yi}
y_i\big|_{t=0}=y_{i,0},
\end{equation}
for given functions $u_{i,0}, y_{i,0}$.
Solving \eqref{yi ode} for $y_i$ we find 
\begin{equation}
\label{yi}
y_i=y_{i,0}(x)\mbox{e}^{-A_i \int_0^tf(u_i)d\tau}.
\end{equation} 
Substituting \eqref{yi} in \eqref{eq1-general} we obtain \eqref{eq1-general-2}, with
\begin{equation}
\label{a}
a(x,\int_0^tf(u_i)d\tau)=\alpha_i(y_{i,0}(x)\mbox{e}^{-A_i \int_0^tf(u_i))d\tau}), 
\end{equation}
\begin{equation}
\label{b}
b(x,\int_0^tf(u_i)d\tau)=\beta_i(y_{i,0}(x)\mbox{e}^{-A_i \int_0^tf(u_i)d\tau})
\end{equation}
and
\begin{equation}
\label{Fi}
F_i(x,u_1,u_2,\int_0^tf(u_i)d\tau)=f_i(y_{i,0}(x)\mbox{e}^{-A_i \int_0^tf(u_i)d\tau},u_1,u_2).
\end{equation}

\bigskip

The system formed by the equations \eqref{eq1-general} and \eqref{yi ode}, with the constitutive
functions \eqref{functions} and \eqref{f}, models the propagation of a combustion front in a porous
medium with two layers \cite{ms}. The unknowns $u_1$ and $y_1$ stands for the temperature and the
fuel concentration, respectively, in one layer, and $u_2$ and $y_2$ stands for the same in the
other layer, and the constants $\lambda_i, a_i$, etc. are parameters related to the medium. 
We refer to \cite{ms} for a detailed derivation of this model.

\medskip

In this paper we solve the Cauchy problem \eqref{eq1-general}, 
\eqref{functions}--\eqref{in cond yi} (or, equivalently, \eqref{eq1-general-2} joint with the initial conditions \eqref{in cond ui}, for
given functions $y_{i,0}$ and $a,b,F_i$ in \eqref{a}-\eqref{Fi}, being $\alpha_i, \beta_i, f_i$ and
$f$ given in \eqref{functions} and \eqref{f}). Furthermore, for a broad class of 
reaction-diffusion systems (see \eqref{general rd s} and \eqref{general rd s wo delta}) we show that the non negative quadrant is a positively invariant region, and, as a consequence, that classical solutions of similar systems, with the reactions functions being non decreasing in one unknown and semi-lipscthitz continuous in the other 
(see \eqref{semi lip}), are bounded by lower and upper solutions for any positive time if so they are at time zero.

Setting some notations, we say that a function is of class $C^{2,1}$ in a set 
$S\subset\R^d\times [0,\infty)$ if it has continuous derivatives up to second order with respect to 
$x$ and up to first order with respect to $t$ for all $(x,t)\in S$, and denote this class by 
$C^{2,1}(S)$ (or simply by $C^{2,1}$), and of class $C^{\alpha,\frac{\alpha}{2}}$ in $S$, for some 
$\alpha\in (0,1]$, if it is bounded and H\"older continuous in $S$ with exponent $\alpha$ with
respect to $x$ (Lipschitz continuous if $\alpha=1$) and with exponent $\frac{\alpha}{2}$ with
respect to $t$, and denote this class by $C^{\alpha,\frac{\alpha}{2}}(S)$ (or simply by 
$C^{\alpha,\frac{\alpha}{2}}$), i.e. a function $u(x,t)$ is said to be in 
$C^{\alpha,\frac{\alpha}{2}}(S)$, for some $\alpha\in (0,1]$, if there is a constant $C>0$ such
that $|u(x,t)|\le C$ for all $(x,t)\in S$ and 
$|u(x_1,t_1)-u(x_2,t_2)|\le C(|x_1-x_2|^\alpha+|t_1-t_2|^{\frac{\alpha}{2}})$ for all 
$(x_1,t_1),(x_2,t_2)\in S$. The space $C^{\alpha,\frac{\alpha}{2}}(S)$ is endowed with the norm
$\|u\|_{\alpha,\frac{\alpha}{2}}\equiv\|u\|_{C^{\alpha,\frac{\alpha}{2}}(S)}:=
\sup_{(x,t)\in S}|u(x,t)| + \sup_{\{(x,t)\not=(y,s), \, (x,t),(y,s)\in S\}}\frac{|u(x,t)-u(y,s)|}{|x-y|^\alpha+|t-s|^\frac{\alpha}{2}}$. \ For the space of lipschitzian bounded functions $u$, defined
in a set $S$ in $\R^d$ or $\R^d\times [0,\infty)$, we use the norm\break
$\|u\|_1:=\sup_S|u|+\sup_{\{x\not=y, \, x,y\in S\}}\frac{|u(x)-u(y)|}{|x-y|}$. 

Throughout the paper $i,j=1,2$ with $j\not=i$.

We denote by $\varphi$ the {\lq\lq}upper solution{\rq\rq} 
$\varphi(t)=(M+\beta)\mbox{e}^{\alpha t}-\beta$ \label{upper} for the Cauchy
problem \eqref{eq1-general}, \eqref{in cond ui} with given $y_i$, satisfying 
$0\le y_i\le \|y_{i,0}\|_\infty$, where\footnote{If $g$ is a bounded function defined in $\mathbb{R}$, $\|g\|_\infty:=\sup_{x\in\mathbb{R}}|g(x)|$.}\break\hfill $M=\max_{i=1,2}\|u_{i,0}\|_\infty$, 
$\alpha=\max_{i=1,2}\{\frac{A_ib_i\|y_{i,0}\|_\infty}{a_i}\}$ and 
$\beta=\max_{i=1,2}\{\frac{d_i}{A_ib_i}\}$, and,
for $0<T\le\infty$, we denote by $\langle 0,\varphi\rangle_T$ the sector 
(set) of vector functions $u=(u_1,u_2) : \mathbb{R}\times [0,T)\to \mathbb{R}^2$ such that 
$0\le u_i(x,t)\le \varphi(t)$ (for $i=1,2$ and) for all 
$(x,t)\in\mathbb{R}\times [0,T).$\footnote{If $T<\infty$ and the function 
$u_i$ is defined and continuous in $\mathbb{R}\times [0,T]$, obviously we can extend the inequality $0\le u_i(x,t)\le \varphi(t)$ to $t=T$.}  
It is easy to check that the pair of (vector) functions $\hat{u}:=(0,0)$ and 
$\tilde{u}:=(\varphi,\varphi)$ is an ordered pair (ordered in the sense that $\hat{u}_i\le\tilde{u}_i$) of lower and upper solutions to the system \eqref{eq1-general} \cite[Lemma 2]{js}. 
See Section \ref{local proof}, p. \pageref{pair}, for details.

Our main results assuring the existence of a local and a global solution to the Cauchy problem 
\eqref{eq1-general}, \eqref{functions}--\eqref{in cond yi} (or, equivalently, \eqref{eq1-general-2}--\eqref{in cond ui}, \eqref{a}--\eqref{Fi}) are the following theorems:

\smallskip

\begin{theorem}\rm{(Local solution)}.
\label{local}
Let $u_{i,0}$ and $y_{i,0}$ be nonnegative, lipschitz continuous and bounded functions in $\R$. Then there is a positive number $T$ such that the Cauchy problem 
\eqref{eq1-general-2}--\eqref{in cond ui}, \eqref{a}--\eqref{Fi} has a solution $u=(u_1,u_2)$ in the class $C^{2,1}(\R\times(0,T])\cap C^{1,\frac{1}{2}}(\R\times[0,T])$ satisfying 
$0\le u_i(x,t)\le \varphi(t)$ for all $(x,t)\in\mathbb{R}\times [0,T]$.
Besides, if additionally $u_{i,0}\in L^p(\R)$ for some $p\in (1,\infty)$ then 
$u\in L^\infty( (0,T) ; L^p(\R) )$, with a possible smaller $T$.
\end{theorem}

\smallskip

\begin{theorem}\rm{(Global solution)}.
\label{global}
Assume that the hypotheses of Theorem \ref{local} are in force, including $u_{i,0}\in L^p(\R)$ for
some $p\in(1,\infty)$, and, in addition, that $y_{i,0}\in C^2(\R)$ and $(y_{i,0})^\prime$ is
bounded. Then the Cauchy problem \eqref{eq1-general-2}--\eqref{in cond ui}, 
\eqref{a}--\eqref{Fi} has a solution $u=(u_1,u_2)$ in \ $C^{2,1}( \R\times(0,\infty) )
\cap  C_{\mbox{\tiny{loc}}}^{1,\frac{1}{2}}(\R\times[0,\infty))\cap L_{\!\!\!\mbox{ \tiny{loc} }}^\infty( (0,\infty) ; L^p(\R))$\footnote{Here the term {\lq\lq}{$loc$}{\rq\rq} stands for {\lq\lq}locally{\rq\rq} in time, i.e. a function $u\in C_{\mbox{\tiny{loc}}}^{1,\frac{1}{2}}(\R\times[0,\infty))\cap L_{\!\!\!\mbox{ \tiny{loc} }}^\infty( (0,\infty) ; L^p(\R) )$ if $u\big|\mathbb{R}\times [0,T]\in C^{1,\frac{1}{2}}(\R\times[0,T])\cap L^\infty((0,T) ; L^p(\R) )$, for any $T>0$.}   satisfying $0\le u_i(x,t)\le \varphi(t)$, for all $(x,t)\in\mathbb{R}\times [0,\infty)$.
\end{theorem}

\bigskip

Furthermore, considering general parabolic operators
\begin{equation}
\label{general rd op}
\mathcal{L}_i= 
\partial_t-\sum_{k,l=1}^da_{i,kl}(x,t)\partial_{x_kx_l}+\sum_{k=1}^db_{i,k}(x,t)\partial_{x_k}, 
\end{equation}
where $x=(x_1,\cdots,x_d)\in\R^d, \, 0<t<T\le\infty$, $d\ge1$, and the operator 
$\mathcal{L}_i$ is uniformly parabolic, i.e. for some constant $\lambda>0$, 
$\sum_{k,l=1}^da_{i,kl}(x,t)\xi_k\xi_l\ge\lambda |\xi|^2$ for all $\xi=(\xi_1,\cdots,\xi_d)\in\R^d$ and all $(x,t)\in\Omega_T:=\R^d\times (0,T)$, using the arguments on invariant regions given in 
\cite{ccs,smoller}
, which basics is the proof of the maximum principle for the heat equation, we state and prove Theorem \ref{invariant} below, and as a consequence, Theorem \ref{comparison}. 
In these theorems we take (vector) functions $u=(u_1,u_2)$ in the class 
$C^{1,2}(\Omega_T)\cap C(\R^d\times [0,T))$\footnote{$C(\R^d\times [0,T))$ denotes the space of continuous vector functions in $\R^d\times [0,T)$.} satisfying the condition 
\begin{equation}
\label{growth}
\liminf_{|x|\to\infty,\, t\to0+}u_i(x,t)\ge0
\end{equation}
(cf. {\em condition K} in \cite{smoller}). 
 
\smallskip

\begin{theorem}
\label{invariant}
Let $\delta$ be a positive number and $c_i(x,t)$ a bounded function in $\Omega_T$.
If $f_i(x,t,u_1,u_2)$ 
is a function such that, for some positive number $\varepsilon_0$, it satisfies $f_i\ge0$ when $-\varepsilon_0<u_i<0$ and $u_j>-\varepsilon_0$, for each $(x,t)\in\Omega_T$ (where $j\not=i$, $i,j=1,2$) then the quadrant 
$Q=\{(u_1,u_2)\,;\, u_1\ge0, u_2\ge0\}$ is a positively invariant region to the system 
\begin{equation}
\label{general rd s}
\mathcal{L}_i(u_i)+c_iu_i=f_i(x,t,u_1,u_2) + \delta 
\end{equation}
for any classical solution $u=(u_1,u_2)$ satisfying the condition \eqref{growth}.
More precisely, under the above hypotheses, if 
$u=(u_1,u_2)\in C^{1,2}(\Omega_T)\cap C(\R^d\times [0,T))$ satisfies \eqref{growth} 
and the inequality
$((\mathcal{L}_i+c_i)u_i)(x,t)\ge f_i(x,t,u_1(x,t),u_2(x,t))+\delta$, for all 
$(x,t)\in\Omega_T$, and $u(x,0)\in Q$ for all $x\in\R^d$, then, $u(x,t)\in Q$ for all 
$(x,t)\in\R^d\times [0,T)$.
\end{theorem}

And as a corollary we obtain

\begin{theorem}
\label{comparison}
Let $\delta$ be a positive number. 
Suppose 
that for each fixed $(x,t)\in\Omega_T$, $f_i(x,t,u_1,u_2)$ is a non decreasing function with respect to $u_j$ (where $j\not= i$, $i,j=1,2$) and, 
for some positive number $\varepsilon_0$, it satisfies the {\lq\lq}semi-lipschitz{\rq\rq} condition
\begin{equation}
\label{semi lip}
\begin{array}{c}
\ f_1(x,t,s+u_1,u_2)-f_1(x,t,u_1,u_2)\ge c_1(x,t)s,\\
f_2(x,t,u_1,s+u_2)-f_2(x,t,u_1,u_2)\ge c_2(x,t)s
\end{array}
\end{equation} 
for all $s\in (-\varepsilon_0,0)$ and all $((x,t),(u_1,u_2))\in\Omega_T\times\R^2$, where $c_i(x,t)$ is some bounded function in $\Omega_T$,\footnote{The conditions \eqref{semi lip} are used in 
\cite[e.g. (2.2)/\S 8.2]{pao}.} and else
\begin{equation}
\label{cond 2}
\begin{array}{c} 
\ f_1(x,t,u_1,s+u_2)-f_1(x,t,u_1,u_2)\ge-\delta',\\
f_2(x,t,s+u_1,u_2)-f_1(x,t,u_1,u_2)\ge-\delta'
\end{array}
\end{equation}
for all $s\in (-\varepsilon_0,0)$ and all $((x,t),(u_1,u_2))\in\Omega_T\times\R^2$,
where $\delta'$ is some positive number less than $\delta$.

{\em 1.} If $\hat{u}=(\hat{u}_1,\hat{u}_2)
\in C^{1,2}(\Omega_T)\cap C(\R^d\times [0,T))$ is a lower solution to the system 
\begin{equation}
\label{general rd s wo ci and delta}
\mathcal{L}_i(u_i)=f_i(x,t,u_1,u_2) 
\end{equation}   
i.e. $(\mathcal{L}_i\hat{u}_i)(x,t)\le f_i(x,t,\hat{u}_1(x,t),\hat{u}_2(x,t))$, 
for all $(x,t)\in\Omega_T$, and 
$u=(u_1,u_2)\in C^{1,2}(\Omega_T)\cap C(\R^d\times [0,T))$ is an upper solution to the system 
\begin{equation}
\label{general rd s 2}
\mathcal{L}_i(u_i)=f_i(x,t,u_1,u_2)+\delta  
\end{equation}
i.e. $\mathcal{L}_i(u_i)(x,t) \ge f_i(x,t,u_1(x,t),u_2(x,t))+\delta$ 
for all $(x,t)\in\Omega_T$, and such that $u-\hat{u}$ satisfies the condition 
\eqref{growth}, and $u_i(x,0)\ge \hat{u}_i(x,0)$ for all $x\in\R^d$, then 
$u_i(x,t)\ge \hat{u}_i(x,t)$ for all $(x,t)\in\Omega_T$.

{\em 2.} Analogously, if $\tilde{u}=(\tilde{u}_1,\tilde{u}_2)\in C^{1,2}(\Omega_T)\cap C(\R^d\times [0,T))$ is an upper solution to the system \eqref{general rd s wo ci and delta}, i.e. 
$(\mathcal{L}_i\tilde{u}_i)(x,t)\ge f_i(x,t,\tilde{u}_1(x,t),\tilde{u}_2(x,t))$, for all 
$(x,t)\in\Omega_T$ and  
$u=(u_1,u_2)\in C^{1,2}(\Omega_T)\cap C(\R^d\times [0,T))$ is a lower solution to the system 
\begin{equation}
\label{general rd s 3}
\mathcal{L}_i(u_i)=f_i(x,t,u_1,u_2)-\delta 
\end{equation}
i.e. $\mathcal{L}_i(u_i)(x,t) \le f_i(x,t,u_1(x,t),u_2(x,t))-\delta$ for all 
$(x,t)\in\Omega_T$, and such that $\tilde{u}-u$ satisfies the condition 
\eqref{growth}, and $\tilde{u}_i(x,0)\ge u_i(x,0)$ for all $x\in\R^d$, then 
$\tilde{u}_i(x,t)\ge{u}_i(x,t)$ for all $(x,t)\in\Omega_T$. 
\end{theorem}

\medskip

Next we give the main ideas to prove theorems \ref{local} and \ref{global}.
 
From now on, we refer to problem \eqref{eq1-general}, \eqref{functions}--\eqref{in cond yi}, or, equivalently, \eqref{eq1-general-2}--\eqref{in cond ui}, \eqref{a}--\eqref{Fi}, simply as
problem \eqref{eq1-general}--\eqref{in cond yi}, or, \eqref{eq1-general-2}--\eqref{in cond ui}.

We prove Theorem \ref{local} by taking the limit of a subsequence given by the iterative scheme
\begin{equation}
\label{scheme}
\left\{\begin{array}{l}
(u_i^n)_t-\alpha_i(y_i^{n-1}){(u_i^n)}_{xx}+\beta_i(y_i^{n-1}){(u_i^n)}_x=
\tilde{f}_i(y_i^{n-1},u_1^{n-1},u_2^{n-1})\\
(y_i^{n-1})_t=-A_iy_i^{n-1}\tilde{f}(u_i^{n-1})\\
(u_i^n,y_i^{n-1})\big|_{t=0}=(u_{i,0},y_{i,0}),
\end{array}\right.
\end{equation}
$n=1,2,\cdots$, \ starting from an initial function $(u_1^0,u_2^0)$ in 
$C^{1,\frac{1}{2}}(\R\times[0,T])$ 
for some sufficiently small time $T>0$, where $\tilde{f}$ is the function that coincides with the Arrhenius function $f(s)=\mbox{e}^{-\frac{E}{s}}$ for $s>0$ and it is equal to zero for $s\le0$, and, $\tilde{f}_i$ 
is the function $f_i$ in \eqref{functions} except for the Arrhenius function $f$ which is replaced by 
$\tilde{f}$.  
More precisely, we show that there is a positive time $T$, depending on the initial data $u_{i,0},y_{i,0}$ 
and on the parameters in the equations (i.e. on $\lambda_i, a_i$, etc.), such that the operator \ 
${\mathcal A}(u_1,u_2)=(w_1,w_2)$, \ where $(w_1,w_2)$ solves
\begin{equation}
\label{eq4}
\left\{\begin{array}{l}(w_i)_t-\alpha_i(y_i){(w_i)}_{xx}+\beta_i(y_i){(w_i)}_x=
\tilde{f}_i(y_i,u_1,u_2)\\
(y_i)_t=-A_iy_i\tilde{f}(u_i)\\
(w_i,y_i)\big|_{t=0}=(u_{i,0},y_{i,0}),
\end{array}\right.
\end{equation}
is well defined in some ball $\Sigma:=\{u=(u_1,u_2)\in C^{1,\frac{1}{2}}(\R\times [0,T]);
\|u_i\|_{C^{1,\frac{1}{2}}(\R\times [0,T])}\leq R, \ i=1,2\}$, $R>0$, i.e. there exist positive number $R,T$ such that ${\mathcal A}(u)\in \Sigma$ for all $u\in\Sigma$.
See Lemma \ref{lem2.1}. In particular, the sequence $\{u^n\}= \{(u_1^n,u_2^n)\}$ given by 
${\mathcal A}(u^n)={\mathcal A}(u^{n-1})$, starting from any $u^0\in\Sigma$, 
is bounded in the norm $\|\cdot\|_{1,1/2}$. Therefore, by Arzel\`a-Ascoli's theorem,
there exists a function $u=(u_1,u_2)\in\Sigma$  
and a subsequence of $\{u^n\}$, which we still denote by $\{u^n\}$, that converges to $u$, uniformly on bounded sets in $\R\times [0,T]$. \ To show that the limit $u$ is a solution
of \eqref{eq1-general-2} and \eqref{in cond ui}, we use the integral representation
\begin{equation}\label{aprox repr}
\begin{array}{rl}
u_i^{n}(x,t) = &
\int_\R\Gamma_{i,n}(x,t,\xi,0)u_{i,0}(\xi)d\xi \\
& + \int_0^t\int_\R
\Gamma_{i,n}(x,t,\xi,\tau)\tilde{f}_i(y_i^{n-1},u_1^{n-1},u_2^{n-1})(\xi,\tau)d\xi d\tau,
\end{array}
\end{equation}
for the solution $u_i^{n}$ of the parabolic equation
\begin{equation}
\label{parabolic n}
(u_i^n)_t-\alpha_i(y_i^{n-1}){(u_i^n)}_{xx}+\beta_i(y_i^{n-1}){(u_i^n)}_x=
\tilde{f}_i(y_i^{n-1},u_1^{n-1},u_2^{n-1})
\end{equation}
occurring in \eqref{scheme}, where $\Gamma_{i,n}$ denotes the fundamental solution of the associated  homogeneous equation $\mathcal{L}_{i,n}w_i=0$, for 
$\mathcal{L}_{i,n}:=\partial_t-\alpha_i(y_i^{n-1})\partial_{xx}+\beta_i(y_i^{n-1})\partial_x$.\break\hfill 
Now suppose that the sequence of fundamental solutions $\{\Gamma_{i,n}\}$ converges, in some appropriate sense, to the fundamental solution $\Gamma_i$ of the also parabolic equation
$\mathcal{L}_iw_i=0$, for 
$\mathcal{L}_i:=\partial_t-\alpha_i(y_i)\partial_{xx}+\beta_i(y_i)\partial_x$, when $n$ tends to
infinite, where $y_i=y_{i,0}(x)\mbox{e}^{-A_i \int_0^tf(u_i))ds}$. Then, having that the sequence $\{u_i^n\}$ is bounded in $\R\times [0,T]$, for some positive $T$, and that it converges uniformly to $u_i\in C^{1,1/2}(\R\times [0,T])$ in bounded sets in $\R\times [0,T]$, it follows from \eqref{aprox repr}
that $u_i$ satisfies the integral equation
\begin{equation}\label{repr}
\begin{array}{rl}
u_i(x,t) = &
\int_\R\Gamma_i(x,t,\xi,0)u_{i,0}(\xi)d\xi \\
& + \int_0^t\int_\R
\Gamma_i(x,t,\xi,\tau)\tilde{f}_i(y_i,u_1,u_2)(\xi,\tau)d\xi d\tau,
\end{array}
\end{equation}
for $(x,t)\in\R\times [0,T]$. Thus, by standard arguments, it follows that\break\hfill
$u_i\in C^{2,1}(\R\times(0,T])\cap C^{1,\frac{1}{2}}(\R\times[0,T])$ and it is a solution of
\eqref{eq1-general-2}--\eqref{in cond ui}. \ In Section \ref{fundamental_solution} we show the continuous dependence of fundamental solutions of parabolic equations with respect to the
coefficients of the equations and, as a consequence, the convergence of $\{\Gamma_{i,n}\}$ to 
$\{\Gamma_i\}$, when $n\to\infty$. 
To conclude the last assertion in Theorem \ref{local} we shall show in Section \ref{local proof}, with the help of the {\lq\lq}generalized Young's inequality{\rq\rq} \cite[p. 9]{Fol1}
and the fact that the fundamental solution $\Gamma_{i,n}$ is a {\lq\lq}regular kernel{\rq\rq}, uniformly with respect to $n$ (see Section \ref{local proof}), that the sequence $\{u_i^n\}$
remains in $L^p$ for all $t\in (0,T)$, with $\|u_i^n(\cdot,t)\|_{L^p}$ uniformly bounded with
respect to $t$ and $n$, if the initial data $u_{i,0}\in L^p$ and $T$ is sufficiently small. 
Then the assertion follows by Banach-Alaoglu's theorem. 
To show that the obtained solution $u=(u_1,u_2)$ is in the sector $\langle 0,\varphi\rangle_T$,
we observe that $u=(u_1,u_2)$ is a solution of the Cauchy problem
\begin{equation}
\label{Cauchy}
\left\{\begin{array}{ll}\mathcal{L}_i(w_i)\equiv(w_i)_t-\alpha_i(y_i)(w_i)_{xx}+\beta_i(y_i)(w_i)_x
=\tilde{f}_i(y_i,w_1,w_2), &
x\in\R, \ t>0\\
w_i(x,0)=u_{i,0}(x), & x\in\R
\end{array}
\right.
\end{equation}
in the unknown $w_i$, for $y_i$ given by \eqref{yi}, and show in Section \ref{local solution} 
that the function 
$\tilde{f}_i(x,t,w_1,w_2)\equiv \tilde{f}_i(y_i(x,t),w_1,w_2)$
satisfies all the hypotheses of Theorem \ref{comparison}, or, more precisely, Corollary 
\ref{comparison cor} in Section \ref{invariant and comparison}. Let us just mention here
that the reaction function $\tilde{f}_i$ in \eqref{Cauchy} is increasing with respect
to $w_j$ ($i,j=1,2$, $j\not=i$). Indeed, from \eqref{functions} we have 
$\partial \tilde{f}_i/\partial w_j=q/(a_i+b_iy_i)$ for all $w_j\in\R$.
In Subsection \ref{local proof} we show that the system \eqref{Cauchy} fulfills all the
hypotheses of Corollary \ref{comparison cor}.

\smallskip

To prove Theorem \ref{global}, we let $[0,T*)$, $0<T^*\le\infty$, to be a maximal\break\hfill interval in which there exists a solution $u^*$ to the problem \eqref{eq1-general-2}-\eqref{in cond ui} in the space\break\hfill  
$X_{T^*}:= C^{2,1}( \R\times(0,T^*) )
\cap  C_{\mbox{\tiny{loc}}}^{1,\frac{1}{2}}(\R\times[0,T^*))\cap L_{\!\!\!\mbox{ \tiny{loc} }}^\infty( (0,T^*) ; L^p(\R) )$ \footnote{Similarly as in the statement of Theorem \ref{global}, here the term {\lq\lq}{$loc$}{\rq\rq} stands for {\lq\lq}locally{\rq\rq} in time, i.e. a function $u\in C_{\mbox{\tiny{loc}}}^{1,\frac{1}{2}}(\R\times[0,T^*))\cap L_{\!\!\!\mbox{ \tiny{loc} }}^\infty( (0,T^*) ; L^p(\R) )$ if $u\big|\mathbb{R}\times [0,T]\in C^{1,\frac{1}{2}}(\R\times[0,T])\cap L^\infty((0,T) ; L^p(\R) )$, for any $T\in (0,T^*)$.} \ intercepted with the sector $\langle 0,\varphi\rangle_{T^*}$, i.e. if $T\ge T^*$ and $u$ is a solution of 
\eqref{eq1-general-2}-\eqref{in cond ui} in\break\hfill $X_T\cap \langle 0,\varphi\rangle_{T}$ that coincides with $u^*$ in 
$[0,T^*)$ then $T=T^*$. (The existence of $T^*$ can be assured in the standard way by Zorn's lemma:
we consider the set of pairs $(u, X_T\cap \langle 0,\varphi\rangle_T)$, such that $u$ is a solution of 
\eqref{eq1-general-2}-\eqref{in cond ui} in $X_T\cap \langle 0,\varphi\rangle_T$, $0<T\le\infty$, ordered with the
relation $(u, X_T\cap \langle 0,\varphi\rangle_T)\le (u', X_T'\cap \langle 0,\varphi\rangle_{T'})$ if $T\le T'$ and 
$u'|[0,T]=u$. Any subset $\mathcal{C}$ of this set of pairs that is totally ordered has the upper
bound $(\overline{u}, X_{\overline{T}}\,\cap \langle 0,\varphi\rangle_{\overline{T}})$, where $\overline{T}$
is the supremum of the set of $T$ such that $(u, X_T\cap \langle 0,\varphi\rangle_T)\in\mathcal{C}$ \ 
($\overline{T}=\infty$ if these set of $T$ is unbounded) and $\overline{u}$ is defined by 
$\overline{u}|[0,T]=u$ whatever it is $(u, X_T\cap \langle 0,\varphi\rangle_T)\in\mathcal{C}$. Then, by Zorn's 
lemma the above set of pairs has a maximum element, i.e. there exists a pair 
$(u^*, X_{T^*}\cap \langle 0,\varphi\rangle_{T^*})$ such that if $(u, X_T\cap \langle 0,\varphi\rangle_T)$ is any other pair
such that $(u, X_T\cap \langle 0,\varphi\rangle_T)\ge (u^*, X_{T^*}\cap \langle 0,\varphi\rangle_{T^*})$
then $(u, X_T\cap \langle 0,\varphi\rangle_T)\le (u^*, X_{T^*}\cap \langle 0,\varphi\rangle_{T^*})$ i.e. if
$u$ is a solution of \eqref{eq1-general-2}-\eqref{in cond ui} in $X_T\cap \langle 0,\varphi\rangle_T$ such
that $T\ge T^*$ and $u|[0,T^*)=u^*$ then $T=T^*$ and $u=u^*$.) Then we shall show in Section 
\ref{global solution} that if $T^*<\infty$ then we have a contradiction, by proving that, in this case, the maximal solution $u^*$ has a continuous extension up to the time $T^*$,
and that this extension is lipschitz continuous and it is in $L^p$, as a function of $x\in\mathbb{R}$, for $t=T^*$,  thus $u^*$ can be extended to a larger time, accordingly with  Theorem \ref{local}.
The idea to extend $u^*$ up to the time $T^*$ is, again, to use the integral
representation \eqref{repr}, for $u_i=u_i^*$, $(x,t)\in\mathbb{R}\times [0,T^*)$,
with $\Gamma_i$ being the fundamental solution of the equation $\mathcal{L}_i^*w_i=0$, 
for $\mathcal{L}_i^*:=\partial_t-\alpha_i(y_i^*)\partial_{xx}+\beta_i(y_i^*)\partial_x$, where  
$y_i^*=y_{i,0}(x)\mbox{e}^{-A_i \int_0^tf(u_i^*)ds}$, 
and with $f_i(y_i,u_1,u_2)=f_i(y_i^*,u_1^*,u_2^*)$. 
To accomplish this, we need to prove that the derivatives $\partial_x u_i^*$
are bounded in $\mathbb{R}\times (0,T^*)$ (see Corollary \ref{uix is bounded}) 
and we do that by the {\lq\lq}method of auxiliary functions{\rq\rq}\footnote{This terminology was used by R. Finn in the MathSciNet review \#MR0064286 (16,259b). In this review he also points
out that this method was {\lq\lq}developed by Picard [see, e.g., Courant and Hilbert,
Methoden der mathematischen Physik, Bd II, Springer, Berlin, 1937, pp. 274--276],
Bernstein [Math. Ann. 69, 82--136 (1910); Doklady Akad. Nauk SSSR (N.S.) 18, 385--388
(1938)] and others{\rq\rq}.}, i.e. following \cite{Ole} (or \cite{Ole-v1,Ole-v2}; see \cite[p. 107]{Ole}), we make a substitution $u_i^*=h_i(v_i)$ for an appropriate function $h_i$
(in particular, such that $h_i'$ is positive and bounded) and estimate $|\partial_x
v_i|$ (instead of trying to estimate $|\partial_x u_i^*|$) at a maximum point, by
looking for the equation satisfied by $v_i$. This leads to 
 some technical estimates where we use the explicit forms for the functions
$\alpha(y_i)$, $\beta_i(y_i)$ and $f_i(y_i,u_1,u_2)$ in \eqref{functions} (see
Section \ref{global solution}). Certainly, it would be a very interesting
investigation to extend our main results regarding the system \eqref{eq1-general} (theorems \ref{local} and
\ref{global}) to more general functions $\alpha(y_i)$, $\beta_i(y_i)$ and
$f_i(y_i,u_1,u_2)$ (or functions $a$, $b$ and $F_i$ in \eqref{eq1-general-2}).

The preceding paragraphs give the fundamental and intuitive ideas to prove theorems \ref{local} 
and \ref{global}. In the next sections we give the rigorous and complete proofs of all theorems stated above. In Section \ref{fundamental_solution} we present a brief summary of the construction
of fundamental solutions for parabolic equations by the parametrix method and state some important
known estimates. Also in this section we show the dependence of the fundamental solution on the
coefficients of the equations. In Section \ref{local solution} we prove theorem
\ref{local} and in Section \ref{invariant and comparison} we prove theorems \ref{invariant} and 
\ref{comparison} and state and prove two corollaries which are version of these theorems in the case one has
continuous dependence of the solution of the system with respect to the 
reaction functions, and also make three remarks giving alternative conditions for the hypotheses of theorems \ref{invariant} and \ref{comparison}.
Finally, in Section \ref{global solution} we prove theorem \ref{global}.

\section{The fundamental solution}
\label{fundamental_solution}

In this section, we present a summary on the construction by the parametrix method and main
properties of the fundamental solution for parabolic equations, 
 and show its continuous dependence with respect to the coefficients of the equations.

\subsection{Definition and some properties}
\label{properties}

Consider the equation and the operator $\mathcal{L}$ given by
\begin{equation}
\label{Lu}
\mathcal{L}u\equiv\dfrac{\partial u}{\partial t}-a(x,t)\frac{\partial^2u}{\partial x^2}+
b(x,t)\frac{\partial u}{\partial x}+c(x,t)u=0\,,
\end{equation}
in the set $\Omega_T:=\{(x,t);\, x\in\R,\, 0\leq t\leq T\}$, for some positive
number $T$, with the coefficients $a, b, c$ in the class
$C^{\alpha,\frac{\alpha}{2}}(\Omega_T)$, for some $\alpha\in (0,1]$, with
$\mathcal{L}$ being a uniform parabolic operator in $\Omega_T$, i.e., there are strictly positive constants 
$\lambda_0,\lambda_1$ such that
\begin{equation}
\label{parabolic}
\lambda_0\leq a(x,t)\leq \lambda_1
\end{equation}
for all $(x,t)$ in $\Omega_T$.

\bigskip

\begin{definition} 
A {\em fundamental solution} of the parabolic equation \eqref{Lu} is a\break\hfill
 function $\Gamma(x,t,\xi,\tau)$, defined for all $(x,t)$ and $(\xi,\tau)$ in $\Omega_T$ with $t>\tau$, such that $\mathcal{L}\Gamma = 0$ in $\Omega_T$, as a function of $(x,t)$, for each fixed $(\xi,\tau)\in\Omega_T$, and \break\hfill
$\lim_{t\rightarrow \tau+}\int_{\R}\Gamma(x,t,\xi,\tau)\psi(\xi)d\xi=\psi(x)$, for all 
$x\in\mathbb{R}$ and $\tau\in [0,T)$, for any continuous function $\psi(x)$ such that
$|\psi(x)|\leq c\mbox{e}^{hx^2}$, for all $x\in\R$, for some positive constants $c$ and
$h$ with $h<1/(4\lambda_1T)$.
\end{definition}

\smallskip

Fundamental solutions for parabolic equations was found by E. E. Levi \cite{Lev}, using the {\em parametrix method}. Our presentation in this section follows mostly \cite{Fri} and \cite{Lad}. Accordingly, the fundamental solution to the equation \eqref{Lu} is given by
\begin{equation}
\label{gamma}
\Gamma(x,t,\xi,\tau)=
Z(x,t,\xi,\tau)+\int_\tau^t\int_{\R}Z(x,t,y,\sigma)\phi(y,\sigma,\xi,\tau)dyd\sigma \,,
\end{equation}
where $(x,t),\,(\xi,\tau)\in \Omega_T$, $t>\tau$, the function $Z(x,t,\xi,\tau)$,
as a function $(x,t)$, is the fundamental solution of the heat equation
$\frac{\partial u}{\partial t}-a(\xi,\tau)\frac{\partial^2 u}{\partial x^2}=0$, i.e.
\begin{equation}
\label{z}
Z(x,t,\xi,\tau)=\frac{1}{(4\pi a(\xi,\tau)(t-\tau))^\frac{1}{2}}\mbox{e}^{-\frac{(x-\xi)^2}{4a(\xi,\tau)(t-\tau)}},
\end{equation}
for each fixed $(\xi,\tau)\in\Omega_T$, and
\begin{equation}
\label{phi}
\phi(x,t,\xi,\tau)=\sum^{\infty}_{m=1}(-1)^m(\mathcal{L}Z)_m(x,t,\xi,\tau),
\end{equation}
where $(\mathcal{L}Z)_1=\mathcal{L}Z=(a(\xi,\tau)-a(x,t))\frac{\partial^2 Z}{\partial x^2}+b\frac{\partial Z}{\partial x}+cZ$ and, for $m\ge1$,
\begin{equation}
\label{defLZm}
(\mathcal{L}Z)_{m+1}(x,t,\xi,\tau)=\int_{\tau}^t\int_{\R}[\mathcal{L}Z(x,t,y,\sigma)](\mathcal{L}Z)_m(y,\sigma,\xi,\tau)dyd\sigma.
\end{equation}

Next we give some important estimates, which, in particular, show that the\break\hfill
function $\Gamma$ given by \eqref{gamma} is well defined, i.e. the series in \eqref{phi}
converges and \eqref{gamma} yields a smooth function $\Gamma$, for $t>\tau$. In the
sequel, $(x,t),(\xi,\tau)\in \Omega_T$, $t>\tau$, and, $K$ and $C$ denote any positive constants.

For the function $Z(x,t,\xi,\tau)$, we have the estimate
\begin{equation}
\label{eq0.7}
|D_t^rD_x^s Z(x,t,\xi,\tau)|\leq K(t-\tau)^{-\frac{1+2r+s}{2}}\mbox{e}^{-C\frac{(x-\xi)^2}{t-\tau}},
\end{equation}
for all nonnegative integers $r, s$, where throughout $D_t^r$ or $\partial_t^r$ and 
$D_x^s$ or $\partial_x^s$ stand for the derivatives with
respect to $t$ and $x$ of order $r$ and $s$, respectively. Besides, since $\int_\R Z(z,t,\xi,\tau)dz=1$, we have that
\begin{equation}
\label{eq0.7.2}
\int_\R D_t^rD_z^sZ(z,t,\xi,\tau)dz=0,
\end{equation}
for all $r,s\in\mathbb{Z}_{+}$ such that $2r+s>0$. Finally, $Z$ and its derivatives are H\"older continuous in $\xi$, i.e.
\begin{equation}
\label{eq0.7.3}
|D_t^rD_z^sZ(z,t,\xi,\tau)-D_t^rD_z^sZ(z,{\xi}',t,\tau)|\leq \frac{K{|\xi-{\xi}'|}^\alpha}{{(t-\tau)}^{-\frac{2r+s+1}{2}}}\mbox{e}^{-C\frac{z^2}{t-\tau}},
\end{equation}
where $C=C(\lambda_1)$ and $K=K(\lambda_0,\lambda_1,\|a\|_{\alpha,\frac{\alpha}{2}})$.
For the function $\phi(x,t,\xi,\tau)$, we have the estimates
\begin{equation}
\label{eq0.8}
|\phi(x,t,\xi,\tau)|\leq \frac{K}{(t-\tau)^{\frac{3-\alpha}{2}}}\mbox{e}^{-C\frac{(x-\xi)^2}{t-\tau}},
\end{equation}
where $C=C(\lambda_1)$ and $K=K(\lambda_0,\lambda_1, \|a\|_{\alpha,\frac{\alpha}{2}},\|b\|_\infty,\|c\|_\infty,T)$, and, for any $\gamma\in (0,\alpha)$,
\begin{equation}
\label{eq0.9}
|\phi(x,t,\xi,\tau)-\phi(y,t,\xi,\tau)|\leq \frac{K|x-y|^\gamma}{(t-\tau)^{\frac{3-(\alpha-\gamma)}{2}}}\left(\mbox{e}^{-C\frac{(x-\xi)^2}{t-\tau}}+\mbox{e}^{-C\frac{(y-\xi)^2}{t-\tau}}\right),
\end{equation}
where $C$ and $K$ are as in \eqref{eq0.8}.

Finally, for the function $\Gamma(x,t,\xi,\tau)$ we have the estimate (see also Corollary 
\ref{uniform L} in this paper)
\begin{equation}
\label{eq0.10}
|D_t^rD_x^s\Gamma(x,t,\xi,\tau)|\leq \frac{K}{(t-\tau)^{\frac{1+2r+s}{2}}}\mbox{e}^{-C\frac{(x-\xi)^2}{t-\tau}},
\end{equation}
for all $r,s\in\mathbb{Z}_{+}$ such that  $2r+s\leq 2$, and, again,
$C$ and $K$ are as in \eqref{eq0.8}. 
 Besides, the fundamental solution $\Gamma$ is nonnegative (see \cite{Aro} and \cite{Ole2}).
 
\smallskip

Now consider the Cauchy problem
\begin{equation}
\label{eq0.11}
\left\{\begin{array}{l}
\mathcal{L}u(x,t)=f(x,t),\quad in \quad \R\times(0,T],\ T>0\\
u(x,0)=u_{0}(x),\quad in \quad \R,
\end{array}\right.
\end{equation}
where $\mathcal{L}$ is defined in \eqref{Lu} and $f$ and $u_0$ are given continuous functions,
in $\R\times (0,T]$ and $\R$, respectively, bounded by the exponential growth
\begin{equation}
\label{exp growth}
|f(x,t)|, \ |u_0(x)|\leq c \mbox{e}^{hx^2}
\end{equation}
for positive constants $c$ and $h$ such that $h<4/(\lambda_1T)$, and for all $x\in\mathbb{R}$ and $t\in [0,T]$.
The following theorem gives a representation formula for its solution using the fundamental solution.

\begin{theorem}
\label{teo0.1}
Let $\Gamma$ be the fundamental solution of the equation $\mathcal{L}u=0$, where $\mathcal{L}$ is
the parabolic operator in \eqref{eq0.11}.
If, besides \eqref{exp growth}, the function $f$ is locally H\"older continuous in $x$, uniformly with respect to $t$, then the function
\begin{equation}
\label{integral representation}
u(x,t)=\int_{\R}\Gamma(x,t,\xi,0)u_0(\xi)d\xi+\int_0^t\int_{\R}\Gamma(x,t,\xi,\tau)f(\xi,\tau)d\xi d\tau
\end{equation}
is the unique solution of the Cauchy problem \eqref{eq0.11} in  $C^{2,1}(\R\times(0,T])\cap C(\R\times[0,T])$ bounded by an exponential growth with respect to $x$, as in \eqref{exp growth}.
\end{theorem}

\smallskip

For a proof, see e.g. \cite[p. 25]{Fri} and \cite[p.182]{Pro}.

\smallskip

\begin{remark}
\label{ob0.1}
In the particular case where $c(x,t)\equiv 0$, we have
$\int_{\R}\Gamma(x,t,\xi,\tau)d\xi=1$ and \break $\int_{\R} D_t^rD_x^s \Gamma(x,t,\xi,\tau)d\xi=0$
for all $r,s\in\mathbb{Z}_{+}$ such that $0<2r+s\leq 2$;
\end{remark}
cf. \eqref{eq0.7.2}. 

\noindent
Indeed, the second claim comes from the first, by the
Lebesgue's convergence dominated theorem, and if $c(x,t)=0$, $u(x,t)=t$ is the unique solution (in the $C^{2,1}$ class with an exponential growth for large $x$) of the problem 
\begin{equation}
\label{eq0.12}
\left\{\begin{array}{l}
\mathcal{L}u(x,t)=1,\quad {\rm in} \quad \R\times(0,T],\\
u(x,0)=0,\quad {\rm in} \quad \R,
\end{array}\right.
\end{equation}
thus, by the Theorem \ref{teo0.1}, it follows that
$t=\int_0^t\int_{\R}\Gamma(x,t,\xi,\tau)d\xi d\tau$,
so $h(\tau)\equiv\int \Gamma(x,t,\xi,\tau)d\xi=1$,
since if $h(\tau_0)\not=1$ for some $\tau_0\in[0,T]$ then, assuming, without
loss of generality, that $h(\tau_0)>1$, by continuity of $h(\tau)$, there would
exist an interval $[a,b]\subset[0,T]$ such that $h(\tau)>1$ for any $\tau\in[a,b]$,
and thus we would get the contradiction \ 
$b-a=\int_0^b\int \Gamma(x,t,\xi,\tau)d\xi d\tau 
- \int_0^a\int \Gamma(x,t,\xi,\tau)d\xi d\tau$\\
$=\int_a^b\int \Gamma(x,t,\xi,\tau)d\xi d\tau>b-a$.

\subsection{Continuous dependence on the coefficients}
\label{dependence_parameters}

We begin this section by setting a notation for {\lq\lq}bounded{\rq\rq} sets of coefficients 
$a,b,c$ of parabolic equations \eqref{Lu}. Given positive numbers $T$, $R$, $\lambda$ and $\alpha$, with $0< \alpha\leq 1$ and $\lambda<R$, let $B(R,\lambda,\alpha)$ be the set of vector valued functions $v=(a(x,t),b(x,t),c(x,t))$ in $C^{\alpha,\frac{\alpha}{2}}(\Omega_T)$ such that $a\geq\lambda$ and $\|a\|_{\alpha,\frac{\alpha}{2}},\|b\|_{\alpha,\frac{\alpha}{2}},\|c\|_{\alpha,\frac{\alpha}{2}}<R$. For a $v\in B(R,\lambda,\alpha)$, we define the norm 
${\|v\|}_{\alpha,\frac{\alpha}{2}}=\max\{{\|a\|}_{\alpha,\frac{\alpha}{2}},{\|b\|}_{\alpha,\frac{\alpha}{2}},{\|c\|}_{\alpha,\frac{\alpha}{2}}\}$. Any $v=(a,b,c)\in B(R,\lambda,\alpha)$ defines
a parabolic equation of the form \eqref{Lu} (with \eqref{parabolic} satisfied with 
$\lambda_0=\lambda$ and $\lambda_1=R$) and, reciprocally, any (uniformly) parabolic
equation of the form \eqref{Lu} (satisfying \eqref{parabolic}) yields a
$v=(a,b,c)\in B(R,\lambda,\alpha)$, for any $\lambda\in(0,\lambda_0)$ and
$R>{\|v\|}_{\alpha,\frac{\alpha}{2}}$.
To highlight the dependence of the operator $\mathcal{L}$ given in \eqref{Lu} on the
coefficients $a,b,c\equiv v$, we shall write $\mathcal{L}=\mathcal{L}_{[v]}$, i.e.
$$
\mathcal{L}_{[v]}u\equiv \mathcal{L}u=\dfrac{\partial u}{\partial t}-a(x,t)\frac{\partial^2u}{\partial x^2}+ b(x,t)\frac{\partial u}{\partial x}+c(x,t)u\,
$$
and for the fundamental solution of $\mathcal{L}_{[v]}u=0$ we shall write 
$\Gamma=\Gamma_{[v]}$, i.e.
\begin{equation}
\label{eq1.1}
\Gamma_{[v]}(x,t,\xi,\tau)=Z_{[v]}(x,t,\xi,\tau)+\int_\tau^t\int_{\R}Z_{[v]}(x,y,t,\sigma)\phi_{[v]}(y,\sigma,\xi,\tau)dyd\sigma\,,
\end{equation}
where \ $Z_{[v]}=Z_{[(a,0,0)]}\equiv Z$ \ and \ $\phi_{[v]}\equiv\phi$ \ are given in \eqref{z} 
and \eqref{phi}.

In the next lemmas we establish some estimates for the fundamental solution \eqref{eq1.1} and its
derivatives which takes into account the dependence on its coefficients $a,b,c\equiv v$. We first
establish these estimates for the functions $Z_{[v]}$ and $\phi_{[v]}$ and then, using 
\eqref{eq1.1} and the series \eqref{phi} for $\phi_{[v]}$, we extend them for $\Gamma_{[v]}$. 
These estimates are the key point to obtain the local solution stated in Theorem \ref{local}.

We shall write $C$ and $K$ to denote positive constants that might depend on
the parameters $R,\lambda,\alpha,T$, but not on the coefficients $v$ neither
on the solutions $u$ or the data $f$, $u_0$, unless otherwise stated. Besides,
$K$ depends continuously on $T$.

\begin{lemma}
\label{lem1.1}
Given $v,\overline{v}\in B(R,\lambda,\alpha)$, we have that
$$
|(D_x^sZ_{[v]}-D_x^sZ_{[\overline{v}]})(x,t,\xi,\tau)|\leq K{\|a-\overline{a}\|}_\infty\frac{1}{(t-\tau)^{\frac{s+1}{2}}}\mbox{e}^{-C\frac{(x-\xi)^2}{(t-\tau)}},
$$
for $s=0,1,2$, where $C<1/(4R)$ and $K=K(R,\lambda)$.
\end{lemma}

\begin{proof} Since
$Z_{[v]}(x,t,\xi,\tau)=\frac{1}{(4\pi a(\xi,\tau)(t-\tau))^{\frac{1}{2}}}\mbox{e}^{-\frac{(x-\xi)^2}{4a(\xi,\tau)(t-\tau)}}$ and its derivatives on $x$ depends on the coefficient $a$ of $\mathcal{L}_{[v]}$, but not depends on the other coefficients $b$ and $c$, computing the derivative of $D_x^sZ_{[v]}$ with respect to $a$, we find
$|D_aD_x^sZ_{[v]}(x,t,\xi,\tau)|\leq \frac{K}{(t-\tau)^{\frac{s+1}{2}}}\mbox{e}^{-C\frac{(x-\xi)^2}{(t-\tau)}}$, for $s=0,1,2$, and constants $K$ and $C$ as in statement of the lemma. Then the desired inequality
follows by the Mean Value Theorem.
 \end{proof}

\begin{lemma}
\label{lem1.2} 
Let $A$ and $\alpha$ be strictily positive numbers being $\alpha\leq 1$, and let $g$ denote the 
{\em gamma function} $g(x):=\int_0^\infty t^{x-1}\mbox{e}^{-t}dt$. Then  
$\sum_{m=1}^\infty{mA^m}/{g(\frac{m\alpha}{2})}$ is a convergent series.
\end{lemma}

\begin{proof}
We begin by recalling the relation $\frac{g(x)g(y)}{g(x+y)}=B(x,y)$
between the {\em gamma function} $g$
and the {\em beta function}, $B(x,y)=\int_0^1t^{x-1}(1-t)^{y-1}dt$
(see \cite[p.41]{Fra}).
Denoting the general term of the given series by $b_m$,
and using the above relation, we obtain
$$
\begin{array}{rl}
\lim_{m\rightarrow \infty}\frac{b_{m+1}}{b_m}&=
A\lim_{m\to\infty}\frac{B(\frac{m\alpha}{2},\frac{\alpha}{2})}{g(\frac{\alpha}{2})}\\

&=A\lim_{m\to\infty}\frac{1}{g(\frac{\alpha}{2})}\int_0^1t^{\frac{m\alpha}{2}-1}(1-t)^{
\frac{\alpha}{2}-1}dt=0.
\end{array}
$$
Therefore, the result follows.
\end{proof}

\smallskip

\begin{lemma}
\label{lem1.3}
Let $\beta\in [0,1]$ and $\gamma\in (0,\alpha)$. If  $v,\overline{v}\in B(R,\lambda,\alpha)$ then
\begin{equation}
\label{eq1.2}
|(\phi_{[v]}-\phi_{[\overline{v}]})(x,t,\xi,\tau)|\leq \frac{K\|v-\overline{v}\|_{\alpha,\frac{\alpha}{2}}}{(t-\tau)^{\frac{3-\alpha}{2}}}\mbox{e}^{-C\frac{(x-\xi)^2}{t-\tau}}
\end{equation}
and
\begin{equation}
\label{eq1.3}
|(\phi_{[v]}(x,t,\xi,\tau)-\phi_{[\overline{v}]}(x,t,\xi,\tau))-(\phi_{[v]}(y,t,\xi,\tau)-\phi_{[\overline{v}]}(y,t,\xi,\tau))|
\end{equation}
$$\leq \frac{K\|v-\overline{v}\|^\beta_{\alpha,\frac{\alpha}{2}}|x-y|^{\gamma(1-\beta)}}{(t-\tau)^{\frac{3-(\alpha-\gamma(1-\beta))}{2}}}\left(\mbox{e}^{- C\frac{(x-\xi)^2}{t-\tau}}+\mbox{e}^{- C\frac{(y-\xi)^2}{t-\tau}}\right),$$
where $C<\frac{1}{4R}$ and $K=K(R,\lambda,\alpha,T)$.
\end{lemma}

\begin{proof}
The proof of \eqref{eq1.2} follows from the following inequality:
\begin{equation}
\label{eq1.4}
\begin{array}{rl}
 & |({(\mathcal{L}Z_{[v]})}_m-{(\mathcal{L}Z_{[\overline{v}]})}_m)(x,t,\xi,\tau)|\\
\leq & mK^m\left(\frac{\pi}{C}\right)^{\frac{m-1}{2}}\|v-\overline{v}\|_{\alpha,\frac{\alpha}{2}}\frac{g(\frac{\alpha}{2})^m}{g(\frac{m\alpha}{2})}\frac{1}{{(t-\tau)}^{\frac{3-m\alpha}{2}}}\mbox{e}^{-C\frac{(x-\xi)^2}{t-\tau}},
\end{array}
\end{equation}
where, for simplicity, we set  $\mathcal{L}\equiv \mathcal{L}_{[v]}$, and $g$ denotes the gamma function; see Lemma \ref{lem1.2}. We show this inequality by induction on $m$. For $m=1$, we have
\begin{align*}
|((&\mathcal{L}_{[v]}Z_{[v]})_1-(\mathcal{L}_{[\overline{v}]}Z_{[\overline{v}]})_1)(x,t,\xi,\tau)|\leq |((a(\xi,\tau)-a(x,t))\partial_{xx}Z_{[v]} +b(x,t)\partial_xZ_{[v]} \\
&+c(x,t)Z_{[v]})-((\overline{a}(\xi,\tau)-\overline{a}(x,t))\partial_{xx}Z_{[\overline{v}]}+\overline{b}(x,t)\partial_xZ_{[\overline{v}]}+
\overline{c}(x,t)Z_{[\overline{v}]})|\\
&\leq |((a(\xi,\tau)-a(x,t))-((\overline{a}(\xi,\tau)-\overline{a}(x,t))))||\partial_{xx}Z_{[v]} |\\
&+ |(\overline{a}(\xi,\tau)-\overline{a}(x,t))||\partial_{xx}Z_{[v]}-\partial_{xx}Z_{[\overline{v}]}|+
|b(x,t)-\overline{b}(x,t)||\partial_xZ_{[v]}|\\
&+|\overline{b}(x,t)||\partial_xZ_{[v]}-
\partial_xZ_{[\overline{v}]}|+|c(x,t)-\overline{c}(x,t)||Z_{[v]}|+
|\overline{c}(x,t)||Z_{[v]}-Z_{[\overline{v}]}|\equiv I.
\end{align*}
Then from Lemma \ref{lem1.1}, \eqref{eq0.7}, and the estimate
$$
\begin{array}{rl}
|x-\xi|^\alpha \mbox{e}^{-C\frac{(x-\xi)^2}{t-\tau}}&=\left(\frac{(x-\xi)^2}{t-\tau}\right)^\alpha(t-\tau)^{\alpha/2}\mbox{e}^{-\frac{1}{8R}\frac{(x-\xi)^2}{t-\tau}} \mbox{e}^{-\left(C-\frac{1}{8R}\right)\frac{(x-\xi)^2}{t-\tau}}\\
&\leq K (t-\tau)^{\alpha/2}\mbox{e}^{-C'\frac{(x-\xi)^2}{t-\tau}},
\end{array}
$$
where $C'$ is a new constant which we shall continue denoting by $C$, we obtain
\begin{align*}
I&\leq K\|v-\overline{v}\|_{\alpha,\frac{\alpha}{2}}(\frac{2}{(t-\tau)^{\frac{3}{2}-\frac{\alpha}{2}}}\mbox{e}^{-C\frac{(x-\xi)^2}{t-\tau}}+\frac{2}{(t-\tau)}\mbox{e}^{-C\frac{(x-\xi)^2}{t-\tau}}+\frac{2}{(t-\tau)^{\frac{1}{2}}}\mbox{e}^{-C\frac{(x-\xi)^2}{t-\tau}})\\
&\leq\frac{K\|v-\overline{v}\|_{\alpha,\frac{\alpha}{2}}}{(t-\tau)^{\frac{3-\alpha}{2}}}\mbox{e}^{-C\frac{(x-\xi)^2}{t-\tau}},
\end{align*}
where $C<1/(4R)$ and $K=K(R,\lambda,\alpha,T)$. 
 Now, assuming that \eqref{eq1.4} is true for an $m\ge1$, we obtain:
\begin{align*}
&|(( \mathcal{L}_{[v]}Z_{[v]})_{m+1}-(\mathcal{L}_{[\overline{v}]}Z_{[\overline{v}]})_{m+1})(x,t,\xi,\tau)|
\\
\leq& |\int_{\tau}^t\int_{\R}[\mathcal{L}_{[v]}Z_{[v]}(x,y,t,\sigma)](\mathcal{L}_{[v]}Z_{[v]})_m(y,\xi,\sigma,\tau)\\
&-[\mathcal{L}_{[\overline v]}Z_{[\overline v]}(x,y,t,\sigma)]
(\mathcal{L}_{[\overline v]}Z_{[\overline v]})_m(y,\xi,\sigma,\tau)dyd\sigma|
\\
\leq& \int_{\tau}^t\int_{\R}|(\mathcal{L}_{[v]}Z_{[v]}-\mathcal{L}_{[\overline v]}Z_{[\overline v]})(x,y,t,\sigma)|
|(\mathcal{L}_{[v]}Z_{[v]})_m (y,\xi,\sigma,\tau)|dyd\sigma
\\
&+\int_{\tau}^t\int_{\R}|\mathcal{L}_{[\overline v]}Z_{[\overline v]}(x,t,\xi,\tau)|
|(\mathcal{L}_{[v]}Z_{[v]})_m(y,\xi,\sigma,\tau)-(\mathcal{L}_{[\overline v]}Z_{[\overline v]})_m(y,\xi,\sigma,\tau)|dyd\sigma
\\
\leq&\int_{\tau}^t\int_{\R}K\frac{\|v-\overline{v}\|_{\alpha,
\frac{\alpha}{2}}}{(t-\sigma)^{\frac{3-\alpha}{2}}}
\mbox{e}^{-C\frac{(x-y)^2}{t-\sigma}}K^m\left(\frac{\pi}{C}\right)^{\frac{m-1}{2}}
\frac{g(\frac{\alpha}{2})^m}{g(\frac{m\alpha}{2})}
\frac{1}{(\sigma-\tau)^{\frac{3-m\alpha}{2}}}\mbox{e}^{-C\frac{(y-\xi)^2}{\sigma-\tau}}dyd\sigma
\\
&+\int_{\tau}^t\int_{\R}K\frac{1}{(t-\sigma)^{\frac{3-\alpha}{2}}}
\mbox{e}^{-C\frac{(x-y)^2}{t-\sigma}}mK^m\left(\frac{\pi}{C}\right)^{\frac{m-1}{2}}\|v-
\overline{v}\|_{\alpha,\frac{\alpha}{2}}\\
&\hspace{7.0cm} \times \frac{g(\frac{\alpha}{2})^m}
{g(\frac{m\alpha}{2})}\frac{1}{(\sigma-\tau)^{\frac{3-m\alpha}{2}}}
\mbox{e}^{-C\frac{(y-\xi)^2}{\sigma-\tau}}dyd\sigma
\\
=&(m+1)K^{m+1}\left(\frac{\pi}{C}\right)^{\frac{m-1}{2}}\|v-\overline{v}\|_{\alpha,\frac{\alpha}{2}}
\frac{g(\frac{\alpha}{2})^m}{g(\frac{m\alpha}{2})}\\
&\hspace{4.0cm}
\times \int_{\tau}^t\int_{\R}\frac{1}{(t-\sigma)^{\frac{3-\alpha}{2}}}
\mbox{e}^{-C\frac{(x-y)^2}{t-\sigma}}\frac{1}{(\sigma-\tau)^{\frac{3-m\alpha}{2}}}
\mbox{e}^{-C\frac{(y-\xi)^2}{\sigma-\tau}}dyd\sigma \,,
\end{align*}
where we used that
\begin{equation*}
\int_{\R}\mbox{e}^{-C\frac{(x-y)^2}{t-\sigma}}\mbox{e}^{-C\frac{(y-\xi)^2}{\sigma-\tau}}dy=
\left(\frac{\pi}{C}\right)^{\frac{1}{2}}\left(\frac{(t-\sigma)(\sigma-\tau)}{t-\tau}\right)^{\frac{1}{2}}
\mbox{e}^{-C\frac{(x-\xi)^2}{t-\tau}}
\end{equation*}
and
\begin{equation*}
\int_\tau^t\frac{1}{(t-\sigma)^{\frac{2-\alpha}{2}}}\frac{1}{(\sigma-\tau)^{\frac{2-m\alpha}{2}}}d\sigma=
\frac{1}{(t-\tau)^{\frac{2-(m+1)\alpha}{2}}}\frac{g(\frac{\alpha}{2})
g(\frac{m\alpha}{2})}{g\left(\frac{(m+1)\alpha}{2}\right)}
\end{equation*}
(see \cite[p. 362]{Lad}). So,
\begin{align*}
|((& L_{[v]}Z_{[v]})_{m+1}-(L_{[\overline{v}]}Z_{[\overline{v}]})_{m+1})(x,t,\xi,\tau)| \\
&\leq (m+1)K^{m+1}\left(\frac{\pi}{C}\right)^{\frac{m}{2}}
\frac{g(\frac{\alpha}{2})^{m+1}}{g\left(\frac{(m+1)\alpha}{2}\right)}
\frac{\|v-\overline{v}\|_{\alpha,\frac{\alpha}{2}}}{(t-\tau)^{\frac{3-(m+1)\alpha}{2}}}\mbox{e}^{-C\frac{(x-\xi)^2}{t-\tau}}.
\end{align*}
This proves the inequality \eqref{eq1.4}.
The inequality \eqref{eq1.4}, \eqref{phi} and Lemma \ref{lem1.2} imply that
\begin{align*}
&|(\phi_{[v]}-\phi_{[\overline{v}]})(x,t,\xi,\tau)
|\leq\sum^{\infty}_{m=1}|((\mathcal{L}_{[v]}Z_{[v]})_m-(\mathcal{L}_{[\overline{v}]}Z_{[\overline{v}]})_m)(x,t,\xi,\tau)|
\\
&\leq\sum^{\infty}_{m=1}mK^m\left(\frac{\pi}{C}\right)^{\frac{m-1}{2}}\frac{g(\frac{\alpha}{2})^m}
{g(\frac{m\alpha}{2})}\frac{\|v-\overline{v}\|_{\alpha,\frac{\alpha}{2}}}{(t-\tau)^{\frac{3-m\alpha}{2}}}\mbox{e}^{-C\frac{(x-\xi)^2}{t-\tau}}
\\
&\leq K\frac{\|v-\overline{v}\|_{\alpha,\frac{\alpha}{2}}}{{(t-\tau)}^{\frac{3-\alpha}{2}}}\mbox{e}^{-C\frac{(x-\xi)^2}{t-\tau}},
\end{align*}
where $C<\frac{1}{4R}$  and $K=\sum^{\infty}_{m=1}mK^m{(\frac{\pi}{C})}^{\frac{m-1}{2}}\frac{{g(\frac{\alpha}{2})}^m}{g(\frac{m\alpha}{2})}T^{\frac{(m-1)\alpha}{2}}
$ are positive constants and  $t-\tau$ was estimated by $T$. This ends the proof of \eqref{eq1.2}.

To prove \eqref{eq1.3}, we write
$$
|(\phi_{[v]}(x,t,\xi,\tau)-\phi_{[\overline{v}]}(x,t,\xi,\tau))-(\phi_{[v]}(y,t,\xi,\tau)-\phi_{[\overline{v}]}(y,t,\xi,\tau))|\equiv\, J=J^{\beta}.J^{1-\beta}
$$
and then use \eqref{eq1.2} to estimate $J^\beta$ and \eqref{eq0.9} to estimate $J^{1-\beta}$, noticing that we can estimate $J$ by
$|(\Phi_{[v]}-\Phi_{[\overline{v}]})(x,t,\xi,\tau)|+|(\Phi_{[v]}-\Phi_{[\overline{v}]})(y,t,\xi,\tau)|$ and also by
$|\Phi_{[v]}(x,t,\xi,\tau)-\Phi_{[v]}(y,t,\xi,\tau)|+|\Phi_{[\overline{v}]}(x,t,\xi,\tau)-\Phi_{[\overline{v}]}(y,t,\xi,\tau)|$.
\ \ \end{proof}

\smallskip

\begin{lemma}
\label{lem1.4}
Let $v,\overline{v}\in B(R,\lambda,\alpha)$, $\beta\in (0,1)$, $\gamma\in (0,\alpha)$, and $\Gamma_{[v]}$, 
$\Gamma_{[\overline{v}]}$, the fundamental solutions of the equations $\mathcal{L}_{[v]}u=0$, 
$\mathcal{L}_{[\overline{v}]}u=0$, as defined in \eqref{eq1.1} and \eqref{z}-\eqref{defLZm}.
Then we have the following estimates:
\begin{equation}
\label{eq1.5}
|(D_x^s\Gamma_{[v]}-D_x^s\Gamma_{[\overline{v}]})(x,t,\xi,\tau)|\leq \frac{K\|v-\overline{v}\|_{\alpha,\frac{\alpha}{2}}}{(t-\tau)^{\frac{s+1}{2}}}\mbox{e}^{-C\frac{(x-\xi)^2}{t-\tau}},\, s=0,1;
\end{equation}
\begin{align}
\label{eq1.6}
&|(\partial_{xx}\Gamma_{[v]}-\partial_{xx}\Gamma_{[\overline{v}]})(x,t,\xi,\tau)|\\
\nonumber
\leq & K(\|v-\overline{v}\|_{\alpha,\frac{\alpha}{2}}+\|v-\overline{v}\|^\beta_{\alpha,\frac{\alpha}{2}})(\frac{1}{|x-\xi|^{1-(\alpha-\gamma(1-\beta))}(t-\tau)^{1-\frac{\gamma(1-\beta)}{2}}}\\
\nonumber
&\hspace{7cm} +\frac{1}{(t-\tau)^{\frac{3}{2}}})
\mbox{e}^{-C\frac{(x-\xi)^2}{t-\tau}};
\end{align}
and
\begin{align}
\label{eq1.7}
&|(\partial_{t}\Gamma_{[v]}-\partial_{t}\Gamma_{[\overline{v}]})(x,t,\xi,\tau)|\\
\nonumber
\leq & K(\|v-\overline{v}\|_{\alpha,\frac{\alpha}{2}}+\|v-\overline{v}\|^\beta_{\alpha,\frac{\alpha}{2}})(\frac{1}{|x-\xi|^{1-(\alpha-\gamma(1-\beta))}(t-\tau)^{1-\frac{\gamma(1-\beta)}{2}}}\\
\nonumber
&\hspace{7cm} +\frac{1}{(t-\tau)^{\frac{3}{2}}})
\mbox{e}^{-C\frac{(x-\xi)^2}{t-\tau}};
\end{align}
where $C<\frac{1}{4R}$ and $K=K(R,\lambda,\alpha,T)$.
\end{lemma}

\medskip

\begin{proof}
For $s=0$ we have that

$|(\Gamma_{[v]}-\Gamma_{[\overline{v}]})(x,t,\xi,\tau)|\leq |(Z_{[v]}-Z_{[\overline{v}]})(x,t,\xi,\tau)|+$

$+\int_{\tau}^t\int_{\R}|Z_{[v]}(x,y,t,\sigma)\phi_{[v]}(y,\sigma,\xi,\tau)-Z_{[\overline{v}]}(x,y,t,\sigma)\phi_{[\overline{v}]}(y,\sigma,\xi,\tau)|dyd\sigma\leq$

$\leq|Z_{[v]}-Z_{[\overline{v}]}|+\int_{\tau}^t\int_{\R}|Z_{[v]}-Z_{[\overline{v}]}||\phi_{[v]}|+|Z_{[\overline{v}]}||\phi_{[v]}-\phi_{[\overline{v}]}|dyd\sigma.$\\

\noindent
From estimates \eqref{eq0.7} and \eqref{eq0.8} and Lemmas \ref{lem1.1} and \ref{lem1.3} it follows that

 \begin{align*}
 &|(\Gamma_{[v]}-\Gamma_{[\overline{v}]})(x,t,\xi,\tau)|\\
\leq& \frac{K\|a-\overline{a}\|_\infty}{(t-\tau)^{\frac{1}{2}}}\mbox{e}^{-C\frac{(x-\xi)^2}{t-\tau}}\\
\nonumber
 &\ \ +\int_{\tau}^t\int_{\R}\frac{K\|a-\overline{a}\|_\infty}{(t-\sigma)^{\frac{1}{2}}}\mbox{e}^{-C\frac{(x-y)^2}{t-\sigma}}\frac{K}{(\sigma-\tau)^{\frac{3-\alpha}{2}}}\mbox{e}^{-C\frac{(y-\xi)^2}{\sigma-\tau}}dyd\sigma\\
\nonumber
 &\ \ +\int_{\tau}^t\int_{\R}\frac{K}{(t-\sigma)^{\frac{1}{2}}}\mbox{e}^{-C\frac{(x-y)^2}{t-\sigma}}\frac{K\|v-\overline{v}\|_{\alpha},\frac{\alpha}{2}}{(\sigma-\tau)^{\frac{3-\alpha}{2}}}\mbox{e}^{-C\frac{(y-\xi)^2}{\sigma-\tau}}dyd\sigma.
\end{align*}

\noindent
Since,
\begin{equation}
\label{eq1.7.1}
\int_{\R}\mbox{e}^{-C\frac{(x-y)^2}{t-\sigma}}\mbox{e}^{-C\frac{(y-\xi)^2}{\sigma-\tau}}dy
=(\frac{\pi}{C})^{\frac{1}{2}}(\frac{(t-\sigma)(\sigma-\tau)}{t-\tau})^{\frac{1}{2}}\mbox{e}^{-C\frac{(x-\xi)^2}{t-\tau}},
\end{equation}
we obtain
\begin{align}
&|(\Gamma_{[v]}-\Gamma_{[\overline{v}]})(x,t,\xi,\tau)|\\
\nonumber
\leq& (1+(\frac{\pi}{C})^{\frac{1}{2}}\int_\tau^t(\sigma-\tau)^{-1+\frac{\alpha}{2}}d\sigma+(\frac{\pi}{C})^{\frac{1}{2}}\int_\tau^t(t-\sigma)^{-1+\frac{\alpha}{2}}d\sigma)\frac{K\|v-\overline{v}\|_{\alpha,\frac{\alpha}{2}}}{(t-\tau)^{\frac{1}{2}}}\mbox{e}^{-C\frac{(x-\xi)^2}{t-\tau}}\\
\nonumber
\leq& \frac{K\|v-\overline{v}\|_{\alpha,\frac{\alpha}{2}}}{(t-\tau)^{\frac{1}{2}}}\mbox{e}^{-C\frac{(x-\xi)^2}{t-\tau}}
\end{align}
where $K=K(R,\lambda,\alpha,T)$.

For the case $s=1$, we have
\begin{align*}
 &(\partial_x\Gamma_{[v]}-\partial_x\Gamma_{[\overline{v}]})(x,t,\xi,\tau)\\
\nonumber
=&(\partial_xZ_{[v]}-\partial_xZ_{[\overline{v}]})(x,t,\xi,\tau)\\
\nonumber
&+\int_\tau^t\int \partial_xZ_{[v]}(x,y,t,\sigma)\phi_{[v]}(y,\sigma,\xi,\tau)-\partial_xZ_{[\overline{v}]}(x,y,t,\sigma)\phi_{[\overline{v}]}(y,\sigma,\xi,\tau)dyd\sigma\\
\nonumber
=&(\partial_xZ_{[v]}-\partial_xZ_{[\overline{v}]})+\int_\tau^t\int(\partial_xZ_{[v]}-\partial_xZ_{[\overline{v}]})\phi_{[v]}dyd\sigma+\int_\tau^t\int\partial_xZ_{[\overline{v}]}(\phi_{[v]}-\phi_{[\overline{v}]})dyd\sigma\\
\nonumber
&\equiv J_1+J_2+J_3.
\end{align*}

\noindent
From Lemma \ref{lem1.1}, we have
\begin{equation}
\label{eq1.8}
|J_1|=|(\partial_xZ_{[v]}-\partial_xZ_{[\overline{v}]})(x,t,\xi,\tau)|\leq \frac{K\|a-\overline{a}\|_\infty}{t-\tau}\mbox{e}^{-C\frac{(x-\xi)^2}{t-\tau}}.
\end{equation}
Using Lemma \ref{lem1.1}, the estimate \eqref{eq0.8} and the identity \eqref{eq1.7.1},  we have
\begin{align}
\label{eq1.9}
|J_2|&\leq\int_\tau^t\int|(\partial_xZ_{[v]}-\partial_xZ_{[\overline{v}]})(x,y,t,\sigma)||\phi_{[v]}(y,\sigma,\xi,\tau)|dyd\sigma\\
\nonumber
&\leq\int_\tau^t\int\frac{K\|a-\overline{a}\|_\infty}{t-\sigma}\mbox{e}^{-C\frac{(x-y)^2}{t-\sigma}}\frac{K}{(\sigma-\tau)^{\frac{3-\alpha}{2}}}\mbox{e}^{-C\frac{(y-\xi)^2}{\sigma-\tau}}dyd\sigma\\
\nonumber
&\leq \frac{K\|a-\overline{a}\|_\infty}{(t-\sigma)^{\frac{1}{2}}}\mbox{e}^{-C\frac{(x-y)^2}{t-\sigma}}\int_\tau^t(t-\sigma)^{-\frac{1}{2}}(\sigma-\tau)^{-1+\frac{\alpha}{2}}d\sigma\\
\nonumber
&\leq \frac{K\|a-\overline{a}\|_\infty}{(t-\tau)^{\frac{2-\alpha}{2}}}\mbox{e}^{-C\frac{(x-\xi)^2}{t-\tau}}\leq \frac{K\|a-\overline{a}\|_\infty}{t-\tau}\mbox{e}^{-C\frac{(x-\xi)^2}{t-\tau}}.
\end{align}

Finally, using Lemma \ref{lem1.3}, \eqref{eq0.7} and \eqref{eq1.7.1}, we obtain
\begin{align}
\label{eq1.10}
|J_3|&\leq\int_\tau^t\int|\partial_xZ_{[\overline{v}]}(x,y,t,\sigma)||(\phi_{[v]}-\phi_{[\overline{v}]})(y,\sigma,\xi,\tau)|dyd\sigma\\
\nonumber
&\leq\int_\tau^t\int\frac{K}{t-\sigma}\mbox{e}^{-C\frac{(x-y)^2}{t-\tau}}\frac{K\|v-\overline{v}\|_{\alpha,\frac{\alpha}{2}}}{(\sigma-\tau)^{\frac{3-\alpha}{2}}}\mbox{e}^{-C\frac{(y-\xi)^2}{\sigma-\tau}}dyd\sigma\\
\nonumber
&\leq\frac{K\|v-\overline{v}\|_{\alpha,\frac{\alpha}{2}}}{(t-\tau)^{\frac{1}{2}}}\mbox{e}^{-C\frac{(x-\xi)^2}{t-\tau}}\int_\tau^t(t-\sigma)^{-\frac{1}{2}}(\sigma-\tau)^{-1+\frac{\alpha}{2}}d\sigma\\
\nonumber
&\leq \frac{K\|v-\overline{v}\|_{\alpha,\frac{\alpha}{2}}}{t-\tau}\mbox{e}^{-C\frac{(x-\xi)^2}{t-\tau}},
\end{align}
where 
 $K=K(R,\lambda,\alpha,T)$.
From \eqref{eq1.8}, \eqref{eq1.9} and \eqref{eq1.10}, we get
$$
|(\partial_x\Gamma_{[v]}-\partial_x\Gamma_{[\overline{v}]})(x,t,\xi,\tau)|\leq \frac{K\|v-\overline{v}\|_{\alpha,\frac{\alpha}{2}}}{t-\tau}\mbox{e}^{-C\frac{(x-\xi)^2}{t-\tau}},
$$
for a $K$ as above.

Regarding the second derivative with respect to $x$, we have
\begin{align*}
&(\partial_{xx}\Gamma_{[v]}-\partial_{xx}\Gamma_{[\overline{v}]})(x,t,\xi,\tau)
=(\partial_{xx}Z_{[v]}-\partial_{xx}Z_{[\overline{v}]})(x,t,\xi,\tau)\\
\nonumber
+&\int_\tau^t\int \partial_{xx}Z_{[v]}(x,y,t,\sigma)\phi_{[v]}(y,\sigma,\xi,\tau)-\partial_{xx}Z_{[\overline{v}]}(x,y,t,\sigma)\phi_{[\overline{v}]}(y,\sigma,\xi,\tau)dyd\sigma\\
\nonumber
=&(\partial_{xx}Z_{[v]}-\partial_{xx}Z_{[\overline{v}]})+\int_\tau^t\int(\partial_{xx}Z_{[v]}-\partial_{xx}Z_{[\overline{v}]})\phi_{[v]}dyd\sigma\\
\nonumber
&+\int_\tau^t\int\partial_{xx}Z_{[\overline{v}]}(\phi_{[v]}-\phi_{[\overline{v}]})dyd\sigma\\
\nonumber
&\equiv I_1+I_2+I_3.
\end{align*}

\noindent
From Lemma \ref{lem1.1},
\begin{equation}
\label{eq1.11}
|I_1|=|(\partial_{xx}Z_{[v]}-\partial_{xx}Z_{[\overline{v}]})(x,t,\xi,\tau)|\leq \frac{K\|a-\overline{a}\|_\infty}{(t-\tau)^{\frac{3}{2}}}\mbox{e}^{-C\frac{(x-\xi)^2}{t-\tau}}.
\end{equation}
To estimate $I_2$, we write
\begin{align*}
&I_2=\int_\tau^t\int(\partial_{xx}Z_{[v]}-\partial_{xx}Z_{[\overline{v}]})(x,y,t,\sigma)\phi_{[v]}(y,\sigma,\xi,\tau)dyd\sigma\\
\nonumber
=&\int_\tau^{\frac{t+\tau}{2}}\int(\partial_{xx}Z_{[v]}-\partial_{xx}Z_{[\overline{v}]})\phi_{[v]}dyd\sigma+\int_{\frac{t+\tau}{2}}^t\int(\partial_{xx}Z_{[v]}-\partial_{xx}Z_{[\overline{v}]})\phi_{[v]}dyd\sigma\\
\nonumber
&\equiv I_2'+I_2''.
\end{align*}
Applying Lemma \ref{lem1.1}, \eqref{eq0.8} and \eqref{eq1.7.1} 
 we get
\begin{align*}
|I_2'|& \leq
\int_\tau^{\frac{t+\tau}{2}}\int \frac{K\|a-\overline{a}\|_\infty}{(t-
\sigma)^{\frac{3}{2}}}\mbox{e}^{-C\frac{(x-y)^2}{t-\sigma}}\frac{K}{(\sigma-\tau)^{\frac{3-\alpha}{2}}}
\mbox{e}^{-C\frac{(y-\xi)^2}{\sigma-\tau}}dyd\sigma \\
\nonumber
&\leq \frac{K\|a-\overline{a}\|_\infty}{(t-\tau)^{\frac{1}{2}}}\mbox{e}^{-C\frac{(x-\xi)^2}{t-\tau}}\int_\tau^{\frac{t+\tau}{2}}(t-\sigma)^{-1}(\sigma-\tau)^{-1+\frac{\alpha}{2}}d\sigma\\
\nonumber
& \leq \frac{K\|a-\overline{a}\|_\infty}{(t-\tau)^{\frac{3-\alpha}{2}}}
\mbox{e}^{-C\frac{(x-\xi)^2}{t-\tau}}
\end{align*}
Now,
\begin{align*}
I_2''&=\int_{\frac{t+\tau}{2}}^t\int(\partial_{xx}Z_{[v]}-\partial_{xx}Z_{[\overline{v}]})(x,y,t,\sigma)\phi_{[v]}(y,\sigma,\xi,\tau)dyd\sigma\\
\nonumber
&=\int_{\frac{t+\tau}{2}}^t\int(\partial_{xx}Z_{[v]}-\partial_{xx}Z_{[\overline{v}]})(x,y,t,\sigma)(\phi_{[v]}(y,\sigma,\xi,\tau)-\phi_{[v]}(x,\xi,\sigma,\tau)+\phi_{[v]}(x,\xi,\sigma,\tau))dyd\sigma\\
\nonumber
&=\int_{\frac{t+\tau}{2}}^t\int(\partial_{xx}Z_{[v]}-\partial_{xx}Z_{[\overline{v}]})(x,y,t,\sigma)(\phi_{[v]}(y,\sigma,\xi,\tau)-\phi_{[v]}(x,\xi,\sigma,\tau))dyd\sigma\\
\nonumber
&+\int_{\frac{t+\tau}{2}}^t\int\left((\partial_{xx}Z_{[v]}-\partial_{xx}Z_{[\overline{v}]})(x,y,t,\sigma)-(\partial_{xx}Z_{[v]}-\partial_{xx}Z_{[\overline{v}]})(x,x,t,\sigma)\right)\phi_{[v]}(x,\xi,\sigma,\tau)dyd\sigma,
\end{align*}
since $\int(\partial_{xx}Z_{[v]}(x,x,t,\sigma)-\partial_{xx}Z_{[\overline{v}]}(x,x,t,\sigma))dy=0$ (see \eqref{eq0.7.2}). Then, applying Lemma \ref{lem1.1}, \eqref{eq0.7.3}, \eqref{eq0.8}  and \eqref{eq0.9}, we obtain
\begin{align*}
|I_2''|&\leq\int_{\frac{t+\tau}{2}}^t\int\frac{K\|a-\overline{a}\|_\infty}{(t-\sigma)^{\frac{3}{2}}}
\mbox{e}^{-C\frac{(x-y)^2}{t-\sigma}}\frac{K|x-y|^{\gamma}}{(\sigma-\tau)^{\frac{3-(\alpha-\gamma)}{2}}}
(\mbox{e}^{-C\frac{(y-\xi)^2}{\sigma-\tau}}+\mbox{e}^{-C\frac{(x-\xi)^2}{\sigma-\tau}})dyd\sigma\\
\nonumber
&+\int_{\frac{t+\tau}{2}}^t\int K\|a-\overline{a}\|^\beta_\infty{(\frac{\mbox{e}^{-C\frac{(x-y)^2}{t-\sigma}}}{(t-\sigma)^{\frac{3}{2}}})}^\beta{|x-y|}^{\alpha(1-\beta)}{(\frac{\mbox{e}^{-C\frac{(x-y)^2}{t-\sigma}}}{(t-\sigma)^{\frac{3}{2}}})}^{1-\beta}{(\frac{\mbox{e}^{-C\frac{(x-\xi)^2}{\sigma-\tau}}}{(\sigma-\tau)^{\frac{3-\alpha}{2}}})}dyd\sigma\\\
\nonumber
&\leq \int_{\frac{t+\tau}{2}}^t\int\frac{K\|a-\overline{a}\|_\infty}{(t-\sigma)^{\frac{3-\gamma}{2}}}
\mbox{e}^{-C\frac{(x-y)^2}{t-\sigma}}\frac{1}{(\sigma-\tau)^{\frac{3-(\alpha-\gamma)}{2}}}
(\mbox{e}^{-C\frac{(y-\xi)^2}{\sigma-\tau}}+\mbox{e}^{-C\frac{(x-\xi)^2}{\sigma-\tau}})dyd\sigma\\
\nonumber
&+\int_{\frac{t+\tau}{2}}^t\int K\|a-\overline{a}\|^\beta_\infty\frac{\mbox{e}^{-C\frac{(x-y)^2}{t-\sigma}}}{(t-\sigma)^{\frac{3-\alpha(1-\beta)}{2}}}\frac{\mbox{e}^{-C\frac{(x-\xi)^2}{\sigma-\tau}}}{(\sigma-\tau)^{\frac{3-\alpha}{2}}}dyd\sigma\\
\nonumber
&\leq K(\|a-\overline{a}\|_\infty+\|a-\overline{a}\|^\beta_\infty)\frac{\mbox{e}^{-C\frac{(x-\xi)^2}{t-\tau}}}{(t-\tau)^{\frac{3-\alpha}{2}}}.
\end{align*}
Thus,
\begin{equation}
\label{eq1.12}
|I_2|\leq|I_2'|+|I_2''|\leq K(\|a-\overline{a}\|_\infty+\|a-\overline{a}\|^\beta_\infty)\frac{\mbox{e}^{-C\frac{(x-\xi)^2}{t-\tau}}}{(t-\tau)^{\frac{3-\alpha}{2}}}.
\end{equation}

To estimate $I_3$, we write
\begin{align*}
I_3=&\int_\tau^t\int\partial_{xx}Z_{[\overline{v}]}(x,y,t,\sigma)(\phi_{[v]}-
\phi_{[\overline{v}]})(y,\sigma,\xi,\tau)dyd\sigma\\
\nonumber
=&\int_\tau^t\int\partial_{xx}Z_{[\overline{v}]}(x,y,t,\sigma)
[(\phi_{[v]}-\phi_{[\overline{v}]})(y,\sigma,\xi,\tau)-(\phi_{[v]}-
\phi_{[\overline{v}]})(x,\xi,\sigma,\tau)]dyd\sigma\\
\nonumber
&+\int_\tau^t\int(\partial_{xx}Z_{[\overline{v}]}(x,y,t,\sigma)-\partial_{xx}Z_{[\overline{v}]}(x,x,t,\sigma))(\phi_{[v]}-
\phi_{[\overline{v}]})(x,\xi,\sigma,\tau)dyd\sigma\\
\nonumber
&\equiv I_3'+I_3'',
\end{align*}
where we have used \eqref{eq0.7.2}. Applying Lemma \ref{lem1.3} and \eqref{eq0.7}, we get
\begin{align*}
&|I_3'|\\
\leq&\int_{\tau}^t\int\frac{K}{(t-\sigma)^{\frac{3}{2}}}
\mbox{e}^{-C\frac{(x-y)^2}{t-\sigma}}\frac{K\|v-\overline{v}\|_{\alpha,\frac{\alpha}{2}}^\beta|x-y|^{\gamma(1-
\beta)}}{(\sigma-\tau)^{\frac{3-(\alpha-\gamma(1-\beta))}{2}}}(\mbox{e}^{-C\frac{(x-\xi)^2}{\sigma-\tau}}+ \mbox{e}^{-C\frac{(y-\xi)^2}{\sigma-\tau}})dyd\sigma\\
\nonumber
\leq&\int_{\tau}^t\int\frac{K}{(t-\sigma)^{\frac{3-\gamma(1-\beta)}{2}}}
\mbox{e}^{-C\frac{(x-y)^2}{t-\sigma}}\frac{K\|v-\overline{v}\|_{\alpha,\frac{\alpha}{2}}^\beta}
{(\sigma-\tau)^{\frac{3-(\alpha-\gamma(1-\beta))}{2}}}(\mbox{e}^{-C\frac{(x-\xi)^2}{\sigma-\tau}}+ \mbox{e}^{-C\frac{(y-\xi)^2}{\sigma-\tau}})dyd\sigma\\
\nonumber
\leq& K\|v-\overline{v}\|_{\alpha,\frac{\alpha}{2}}^\beta\int_{\tau}^t\frac{1}{(t-\sigma)^{\frac{3-\gamma(1-\beta)}{2}}}
\frac{1}
{(\sigma-\tau)^{\frac{3-(\alpha-\gamma(1-\beta))}{2}}}\mbox{e}^{-C\frac{(x-\xi)^2}{\sigma-\tau}}(\int \mbox{e}^{-C\frac{(x-y)^2}{t-\sigma}}dy)d\sigma\\
\nonumber
+&K\|v-\overline{v}\|_{\alpha,\frac{\alpha}{2}}^\beta\int_{\tau}^t\frac{1}{(t-\sigma)^{\frac{3-\gamma(1-\beta)}{2}}}
\frac{1}
{(\sigma-\tau)^{\frac{3-(\alpha-\gamma(1-\beta))}{2}}}(\int \mbox{e}^{-C\frac{(x-y)^2}{t-\sigma}} \mbox{e}^{-C\frac{(y-\xi)^2}{\sigma-\tau}}dy)d\sigma\\
\nonumber
\leq& K\|v-\overline{v}\|_{\alpha,\frac{\alpha}{2}}^\beta\int_{\tau}^t\frac{1}{(t-\sigma)^{\frac{2-\gamma(1-\beta)}{2}}}
\frac{1}
{(\sigma-\tau)^{\frac{3-(\alpha-\gamma(1-\beta))}{2}}}\mbox{e}^{-C\frac{(x-\xi)^2}{\sigma-\tau}}d\sigma\\
\nonumber
&+\frac{K\|v-\overline{v}\|_{\alpha,\frac{\alpha}{2}}^\beta}{(t-\tau)^{\frac{1}{2}}}\mbox{e}^{-C\frac{(x-\xi)^2}{t-\tau}}\int_{\tau}^t\frac{1}{(t-\sigma)^{\frac{2-\gamma(1-\beta)}{2}}}
\frac{1}
{(\sigma-\tau)^{\frac{2-(\alpha-\gamma(1-\beta))}{2}}}d\sigma\\
\nonumber
\leq& K\|v-\overline{v}\|_{\alpha,\frac{\alpha}{2}}^\beta(\frac{1}{(t-\tau)^{\frac{3-\alpha}{2}}}+\frac{1}{|x-\xi|^{1-(\alpha-\gamma(1-\beta))}(t-\tau)^{1-\frac{\alpha(1-\beta)}{2}}})\mbox{e}^{-C\frac{(x-\xi)^2}{t-\tau}}\\
\nonumber
&+\frac{K\|v-\overline{v}\|_{\alpha,\frac{\alpha}{2}}^\beta}{(t-\tau)^{\frac{3-\alpha}{2}}}\mbox{e}^{-C\frac{(x-\xi)^2}{t-\tau}}\\
\nonumber
\leq& K\|v-\overline{v}\|^\beta_{\alpha,\frac{\alpha}{2}}(\frac{1}{|x-\xi|^{1-(\alpha-\gamma(1-\beta))}(t-\tau)^{1-\frac{\gamma(1-\beta)}{2}}}+\frac{1}{(t-\tau)^{\frac{3-\alpha}{2}}})
\mbox{e}^{-C\frac{(x-\xi)^2}{t-\tau}}.
\end{align*}

\noindent
In order to estimate $I_3''$, we use Lemma \ref{lem1.3} and \eqref{eq0.7.3} as follows:
\begin{align*}
|I_3''|&\leq\int_{\tau}^t\int\frac{K|x-y|^\alpha}{(t-\sigma)^{\frac{3}{2}}}
\mbox{e}^{-C\frac{(x-y)^2}{t-\sigma}}\frac{\|v-\overline{v}\|_{\alpha,\frac{\alpha}{2}}}{(\sigma-\tau)^{\frac{3-\alpha}{2}}}\mbox{e}^{-C\frac{(x-\xi)^2}{\sigma-\tau}}dyd\sigma\\
\nonumber
&\leq K\|v-\overline{v}\|_{\alpha,\frac{\alpha}{2}}\int_{\tau}^t\frac{1}{(t-\sigma)^{\frac{3-\alpha}{2}}}
\frac{1}{(\sigma-\tau)^{\frac{3-\alpha}{2}}}\mbox{e}^{-C\frac{(x-\xi)^2}{\sigma-\tau}}(\int \mbox{e}^{-C\frac{(x-y)^2}{t-\sigma}} dy)d\sigma\\
\nonumber
&\leq K\|v-\overline{v}\|_{\alpha,\frac{\alpha}{2}}\int_{\tau}^t\frac{1}{(t-\sigma)^{\frac{2-\alpha}{2}}}
\frac{1}{(\sigma-\tau)^{\frac{3-\alpha}{2}}}\mbox{e}^{-C\frac{(x-\xi)^2}{\sigma-\tau}}d\sigma
\end{align*}
\begin{align*}
\nonumber
&\leq K\|v-\overline{v}\|_{\alpha,\frac{\alpha}{2}}\int_{\tau}^t\frac{T^{\frac{\alpha-\gamma(1-\beta)\beta}{2}}}{(t-\sigma)^{\frac{2-\gamma(1-\beta)}{2}}}
\frac{T^{\frac{\gamma(1-\beta)}{2}}}{(\sigma-\tau)^{\frac{3-(\alpha-\gamma(1-\beta))}{2}}}\mbox{e}^{-C\frac{(x-\xi)^2}{\sigma-\tau}}d\sigma\\
\nonumber
&\leq K\|v-\overline{v}\|_{\alpha,\frac{\alpha}{2}}(\frac{1}{|x-\xi|^{1-(\alpha-\gamma(1-\beta))}(t-\tau)^{1-\frac{\gamma(1-\beta)}{2}}}+\frac{1}{(t-\tau)^{\frac{3-\alpha}{2}}})
\mbox{e}^{-C\frac{(x-\xi)^2}{t-\tau}}.
\end{align*}
Then,
\begin{align}
\label{eq1.13}
|I_3|&\leq |I_3'|+|I_3''| \\
\nonumber
&\leq K(\|v-\overline{v}\|_{\alpha,\frac{\alpha}{2}}+\|v-\overline{v}\|^\beta_{\alpha,\frac{\alpha}{2}})(\frac{1}{|x-\xi|^{1-(\alpha-\gamma(1-\beta))}(t-\tau)^{1-\frac{\gamma(1-\beta)}{2}}}\\
\nonumber
&\hspace{8cm} +\frac{1}{(t-\tau)^{\frac{3-\alpha}{2}}})
\mbox{e}^{-C\frac{(x-\xi)^2}{t-\tau}}.
\end{align}

From the above estimates, \eqref{eq1.11}, \eqref{eq1.12} and \eqref{eq1.13}, we obtain
\begin{align*}
&|(\partial_{xx}\Gamma_{[v]}-\partial_{xx}\Gamma_{[\overline{v}]})(x,t,\xi,\tau)|\\
\nonumber
\leq &K(\|v-\overline{v}\|_{\alpha,\frac{\alpha}{2}}+\|v-\overline{v}\|^\beta_{\alpha,\frac{\alpha}{2}})\\
& \hspace{1cm} \times (\frac{1}{|x-\xi|^{1-(\alpha-\gamma(1-\beta))}(t-\tau)^{1-\frac{\gamma(1-\beta)}{2}}}+\frac{1}{(t-\tau)^{\frac{3}{2}}})
\mbox{e}^{-C\frac{(x-\xi)^2}{t-\tau}}.
\end{align*}

Finally, the proof of \eqref{eq1.7} follows from \eqref{eq1.5}, \eqref{eq1.6}
and the equations\break\hfill $L_{[v]}\Gamma_{[v]}=L_{[\overline{v}]}\Gamma_{[\overline{v}]}=0$.
\ \ \end{proof}

\smallskip

\begin{corollary} 
\label{uniform L}
For $v\in B(R,\lambda,\alpha)$ we have the following uniform estimate:
\begin{equation}
\label{eq uniform L}
|D_x^s\Gamma_{[v]}(x,t,\xi,\tau)|\leq \frac{K}{(t-\tau)^{\frac{s+1}{2}}}\mbox{e}^{-C\frac{(x-\xi)^2}{t-\tau}},\, s=0,1
\end{equation}
where $K=K(R,\lambda,\alpha,T)$.
\end{corollary}

\begin{proof} Take $\overline{v}=(1,0,0)$ in \eqref{eq1.5}. \ \ \end{proof}
\smallskip

We also have the following lemma.
\begin{lemma}
\label{lem1.5}
Let $v_n,v\in B(R,\lambda,\alpha)$, $n=1,2,\cdots$. If $v_n(x,t)$ converges to $v(x,t)$ pointwise in $\R\times[0,T]$, as $n$ goes to infinity, then $\Gamma_{[v_n]}(x,t,\xi,\tau)$ converges to $\Gamma_{[v]}(x,t,\xi,\tau)$, for any $(x,t),\,(\xi,\tau)\in\R\times[0,T]$, with $t>\tau$.
\end{lemma}
\begin{proof}
First we show the pointwise convergence of $Z_{[v_n]}$ and $\phi_{[v_n]}$.
From \eqref{z} it is easy to see that
\begin{equation}
\label{eq1.27}
D_t^rD_x^sZ_{[v_n]}\rightarrow D_t^rD_x^sZ_{[v]}
\end{equation}
pointwise, where $r$ and $s$ are nonnegative integers.
To proof that $\phi_{[v_n]}$ converges pointwise to $\phi_{[v]}$, we notice that
$$
\mathcal{L}_{[v_n]}(Z_{[v_n]})=(a_n(\xi,\tau)-a_n(x,t))\partial_{xx}Z_{[v_n]}+b_n(x,t)\partial_x Z_{[v_n]}+c_n(x,t)Z_{[v_n]},
$$
so, it follows from \eqref{eq1.27} that
\begin{equation}
\label{eq1.24}
\mathcal{L}_{[v_n]}(Z_{[v_n]})\rightarrow\mathcal{L}_{[v]}(Z_{[v]}),
\end{equation}
pointwise. Besides, we have
\begin{equation}
\label{eq1.25}
|\mathcal{L}_{[v_n]}(Z_{[v_n]}(x,t,\xi,\tau))|\leq \frac{K}{(t-\tau)^{\frac{3-\alpha}{2}}}\mbox{e}^{-C\frac{(x-\xi)^2}{t-\tau}},
\end{equation}
where $K$ and $C$ are positive constants which do not depends on $n$.
Now, recalling \eqref{defLZm}, one can show by induction on $m$ that
$(\mathcal{L}_{[v_n]})_m$ converges to $(\mathcal{L}_{[v]})_m$, pointwise, as $m$ goes to infinity.
Indeed, following the construction of the fundamental solution in \cite[p. 362]{Lad}, we have
\begin{equation}
\label{eq1.26}
|(\mathcal{L}_{[v_n]})_m(Z_{[v_n]}(x,t,\xi,\tau))|\leq K^m\left(\frac{\pi}{C}\right)^{\frac{m-1}{2}}\frac{g(\frac{\alpha}{2})^m}{g(\frac{m\alpha}{2})}\frac{1}{(t-\tau)^{\frac{3-m\alpha}{2}}}\mbox{e}^{-C\frac{(x-\xi)^2}{t-\tau}},
\end{equation}
where $g$ is the {\em gamma function}.
So,

$|(\mathcal{L}_{[v_n]})(Z_{[v_n]}(x,y,t,\sigma))(\mathcal{L}_{[v_n]})_m(Z_{[v_n]}(y,\xi,\sigma,\tau))|$

\centerline{$\leq \frac{K\mbox{e}^{-C\frac{(x-y)^2}{t-\sigma}}}{(t-\sigma)^{\frac{3-\alpha}{2}}}K^m(\frac{\pi}{C})^{\frac{m-1}{2}}\frac{g(\frac{\alpha}{2})^m}{g(\frac{m\alpha}{2})}\frac{\mbox{e}^{-C\frac{(y-\xi)^2}{\sigma-\tau}}}{(\sigma-\tau)^{\frac{3-m\alpha}{2}}}$}

\noindent
and thus, by the induction hypothesis, we obtain $(\mathcal{L}_{[v_n]})(Z_{[v_n]})(L_{[v_n]})_m(Z_{[v_n]})\rightarrow (\mathcal{L}_{[v]})(Z_{[v]})(\mathcal{L}_{[v]})_m(Z_{[v]})$, pointwise. Then, by the Lebesgue's Dominated Convergence Theorem,

\noindent
$(\mathcal{L}_{[v_n]})_{m+1}(Z_{[v_n]}(x,t,\xi,\tau))$

\centerline{$=\int_\tau^t\int(\mathcal{L}_{[v_n]})(Z_{[v_n]}(x,y,t,\sigma))(\mathcal{L}_{[v_n]})_m(Z_{[v_n]}(y,\xi,\sigma,\tau))dyd\sigma$}

\noindent
converges to  $(\mathcal{L}_{[v]})_{m+1}(Z_{[v]}(x,t,\xi,\tau))$.
The estimate \eqref{eq1.26} ensures the uniform convergence of \break  $\sum_{m=1}^\infty(-1)^m(\mathcal{L}_{[v_n]}(Z_{[v_n]})_m(x,t,\xi,\tau)$ with respect to $(x,\xi)$ and $t-\tau>\delta$, for each fixed $\delta>0$, and so, $\phi_{[v_n]}\rightarrow \phi_{[v]}$, pointwise.
To end the proof of the Lemma, notice that

$|Z_{[v_n]}(x,y,t,\sigma)\phi_{[v_n]}(y,\sigma,\xi,\tau)|\leq \frac{K}{(t-\sigma)^{\frac{1}{2}}}\mbox{e}^{-C\frac{(x-y)^2}{t-\sigma}}\frac{1}{(\sigma-\tau)^{\frac{3-\alpha}{2}}}\mbox{e}^{-C\frac{(y-\xi)^2}{\sigma-\tau}}$\\
and  $Z_{[v_n]}\phi_{[v_n]}$ converges pointwise to $Z_{[v]}\phi_{[v]}$, so, again from the Lebesgue's Dominated Convergence Theorem it follows that
$\int_\tau^t\int Z_{[v_n]}\phi_{[v_n]}dyd\sigma\rightarrow \int_\tau^t\int Z_{[v]}\phi_{[v]}dyd\sigma$. Thus, we conclude that $\Gamma_{[v_n]}\rightarrow \Gamma_{[v]}$, pointwise.
\ \ \end{proof}

\smallskip

\begin{theorem}
\label{teo1.1}
Let $T>0$, $\beta\in (0,1)$, \ $v=(a,b,0), \ \overline{v}=(\overline{a},\overline{b},0)\in B(R,\lambda,1)$, $f,\overline{f}\in C^{1,\frac{1}{2}}(\Omega_T)$ and $u_0,\overline{u}_0$ be Lipschitz continuous and bounded functions in $\R$. If  $u$ and $\overline{u}$ are the  solutions of the problems
\begin{equation}
 \label{eq1.14}
 \mathcal{L}_{[v]}u=f,\quad\mbox{in}\quad \R\times(0,T],\quad  u(x,0)=u_0,\quad x\in\R,
  \end{equation}
  \begin{equation}
  \mathcal{L}_{[\overline{v}]}\overline{u}=\overline{f},\quad\mbox{in}\quad \R\times(0,T],\quad u(x,0)=\overline{u}_0,\quad x\in\R,
  \end{equation}
then
  \begin{align}
  \|u-\overline{u}\|_{1,\frac{1}{2}}\leq K[&\|v-\overline{v}\|_{1,\frac{1}{2}}+\|v-\overline{v}\|_{1,\frac{1}{2}}^\beta+\|u_0-\overline{u}_0\|_1\\
  \nonumber
  &+T^{\frac{1}{2}}(\|f\|_{1,\frac{1}{2}}+1)(\|f-\overline{f}\|_{1,\frac{1}{2}}+\|v-\overline{v}\|_{1,\frac{1}{2}}+\|v-\overline{v}\|_{1,\frac{1}{2}}^\beta)],
\end{align}
where $K=K(R,\lambda,T,\|u_0\|_1)$.
 \end{theorem}

\begin{proof}
From Theorem  \ref{teo0.1} we have
\begin{align*}
(u-\overline{u})(x,t)=&\int_\R\Gamma_{[v]}(x,t,\xi,0)u_0(\xi)-\Gamma_{[\overline{v}]}(x,t,\xi,0)\overline{u}_0(\xi)d\xi\\
\nonumber
&+\int_0^t\int_\R\Gamma_{[v]}(x,t,\xi,\tau)f(\xi,\tau)-\Gamma_{[\overline{v}]}(x,t,\xi,\tau)\overline{f}(\xi,\tau)d\xi d\tau\\
\nonumber
&\equiv V(x,t)+W(x,t).
\end{align*}
By Lema \ref{lem1.4} and \eqref{eq0.10}, we get
\begin{align}
\label{eq1.15}
|V(x,t)|&\leq\int_\R|\Gamma_{[v]}(x,t,\xi,0)u_0(\xi)-\Gamma_{[\overline{v}]}(x,t,\xi,0)\overline{u}_0(\xi)|d\xi\\
\nonumber
&\leq\int_\R|(\Gamma_{[v]}-\Gamma_{[\overline{v}]})(x,t,\xi,0)u_0(\xi)|+|\Gamma_{[\overline{v}]}(x,t,\xi,0)(u_0(\xi)-\overline{u}_0(\xi))|d\xi\\
\nonumber
&\leq\int_\R\frac{K\|v-\overline{v}\|_{1,\frac{1}{2}}}{t^{\frac{1}{2}}}\mbox{e}^{-C\frac{(x-\xi)^2}{t}}\|u_0\|_\infty+\frac{K}{t^{\frac{1}{2}}}\mbox{e}^{-C\frac{(x-\xi)^2}{t}}\|u_0-\overline{u}_0\|_\infty d\xi\\
\nonumber
&\leq K(\|v-\overline{v}\|_{1,\frac{1}{2}}+ \|u_0-\overline{u}_0\|_\infty),
\end{align}
where $K=K(R,\lambda,T,\|u_0\|_\infty)$. 
 In view of Remark \ref{ob0.1}, we can write
\begin{align*}
\partial_xV(x,t)&=\int_\R\partial_x\Gamma_{[v]}(x,t,\xi,0)u_0(\xi)-\partial_x\Gamma_{[\overline{v}]}(x,t,\xi,0)\overline{u}_0(\xi)d\xi\\
\nonumber
&=\int_\R(\partial_x\Gamma_{[v]}-\partial_x\Gamma_{[\overline{v}]})(x,t,\xi,0)(u_0(\xi)-u_0(x))d\xi\\
\nonumber
&+\int\partial_x\Gamma_{[\overline{v}]}(x,t,\xi,0)[(u_0(\xi)-\overline{u}_0(\xi))-(u_0(x)-\overline{u}_0(x))]d\xi,\\
\end{align*}
so, by Lemma \ref{lem1.4} and estimate \eqref{eq0.10} and using that
$$
\frac{|x-\xi|}{t} \mbox{e}^{-C\frac{(x-\xi)^2}{t}}
=\frac{1}{t^{1/2}} (\frac{|x-\xi|}{t^{1/2}} \mbox{e}^{-(C/2)\frac{(x-\xi)^2}{t}})\mbox{e}^{-(C/2)\frac{(x-\xi)^2}{t}}
\leq \mbox{const.}\frac{1}{t^{1/2}}\mbox{e}^{-(C/2)\frac{(x-\xi)^2}{t}},
$$
we get
\begin{align}
\label{eq1.16}
|\partial_xV(x,t)|&\leq \int_\R(\frac{K\|v-\overline{v}\|_{1,\frac{1}{2}} \|u_0\|_1|x-\xi|}{t}+\frac{K\|u_0-\overline{u}_0\|_1|x-\xi|}{t})\mbox{e}^{-C\frac{(x-\xi)^2}{t}}d\xi\\
\nonumber
&\leq \int_\R\frac{K\|u_0\|_1\|v-\overline{v}\|_{1,\frac{1}{2}}}{t^{\frac{1}{2}}}+\frac{K\|u_0-\overline{u}_0\|_1}{t^{\frac{1}{2}}})\mbox{e}^{-C\frac{(x-\xi)^2}{t}}d\xi\\
\nonumber
&\leq K(\|v-\overline{v}\|_{1,\frac{1}{2}}+\|u_0-\overline{u}_0\|_1),
\end{align}
with $K=K(R,\lambda,T,\|u_0\|_1)$.

In order to get the H\"older continuity with respect to $t$, using again Remark \ref{ob0.1}, we write
\begin{align*}
&V(x,t)-V(x,t')\\
\nonumber
&=\int_\R(\Gamma_{[v]}(x,t,\xi,0)-\Gamma_{[v]}(x,t',\xi,0))u_0(\xi)-(\Gamma_{[\overline{v}]}(x,t,\xi,0)-\Gamma_{[\overline{v}]}(x,t',\xi,0))\overline{u}_0(\xi)d\xi\\
\nonumber
&=\int_\R\int_{t'}^t\partial_t\Gamma_{[v]}(x,s,\xi,0)u_0(\xi)-\partial_t\Gamma_{[\overline{v}]}(x,s,\xi,0)\overline{u}_0(\xi)dsd\xi \\
\nonumber
&=\int_{t'}^t\int_\R(\partial_t\Gamma_{[v]}-\partial_t\Gamma_{[\overline{v}]})(x,s,\xi,0)u_0(\xi)+\partial_t\Gamma_{[\overline{v}]}(x,s,\xi,0)(u_0-\overline{u}_0)(\xi)d\xi ds\\
\nonumber
&=\int_{t'}^t\int_\R(\partial_t\Gamma_{[v]}-\partial_t\Gamma_{[\overline{v}]})(x,s,\xi,0)(u_0(\xi)-u_0(x))d\xi ds\\
\nonumber
&+\int_{t'}^t\int_\R\partial_t\Gamma_{[\overline{v}]}(x,s,\xi,0)[(u_0-\overline{u}_0)(\xi)-(u_0-\overline{u}_0)(x)]d\xi ds.
\end{align*}
Thence, from Lemma \ref{lem1.4} and estimate \eqref{eq0.10}, we obtain
\begin{align}
\label{eq1.17}
&|V(x,t)-V(x,t')|\\
\nonumber
\leq&\int_{t'}^t\int_\R|\partial_t\Gamma_{[v]}-\partial_t\Gamma_{[\overline{v}]})(x,s,\xi,0)||u_0(\xi)-u_0(x)|d\xi ds\\
\nonumber
&+\int_{t'}^t\int_\R|\partial_t\Gamma_{[\overline{v}]}(x,s,\xi,0)||(u_0-\overline{u}_0)(\xi)-(u_0-\overline{u}_0)(x)|d\xi ds\\
\nonumber
\leq &\int_{t'}^t\int_\R K(\|v-\overline{v}\|_{1,\frac{1}{2}}+\|v-\overline{v}\|^\beta_{1,\frac{1}{2}})\|u_0\|_1|x-\xi|(\frac{1}{|x-\xi|^{\gamma(1-\beta)}s^{\frac{2-\gamma(1-\beta)}{2}}}\\
\nonumber
&\hspace{7.7cm} +\frac{1}{s^{\frac{3}{2}}})\mbox{e}^{-C\frac{(x-\xi)^2}{s}}d\xi ds\\
\nonumber
&+\int_{t'}^t\int_\R\frac{K\|u_0-\overline{u}_0\|_1|x-\xi|}{s^{\frac{3}{2}}}\mbox{e}^{-C\frac{(x-\xi)^2}{s}}d\xi ds\\
\nonumber
\leq& \int_{t'}^t\int_\R K(\|v-\overline{v}\|_{1,\frac{1}{2}}+\|v-\overline{v}\|^\beta_{1,\frac{1}{2}})\|u_0\|_1(\frac{T^{\frac{1}{2}}}{s}+\frac{1}{s})\mbox{e}^{-C\frac{(x-\xi)^2}{s}}d\xi ds\\
\nonumber
&+\int_{t'}^t\int_\R\frac{K\|u_0-\overline{u}_0\|_1}{s}\mbox{e}^{-C\frac{(x-\xi)^2}{s}}d\xi ds\\
\nonumber
\leq& K(\|v-\overline{v}\|_{1,\frac{1}{2}}+\|v-\overline{v}\|^\beta_{1,\frac{1}{2}}+\|u_0-\overline{u}_0\|_1)\int_{t'}^t\int_\R\frac{1}{s}\mbox{e}^{-C\frac{(x-\xi)^2}{s}}d\xi ds\\
\nonumber
\leq& K(\|v-\overline{v}\|_{1,\frac{1}{2}}+\|v-\overline{v}\|^\beta_{1,\frac{1}{2}}+\|u_0-\overline{u}_0\|_1)\int_{t'}^t\frac{1}{s^{\frac{1}{2}}}ds\\
\nonumber
\leq& K(\|v-\overline{v}\|_{1,\frac{1}{2}}+\|v-\overline{v}\|^\beta_{1,\frac{1}{2}}+\|u_0-\overline{u}_0\|_1)(t-t')^{\frac{1}{2}},
\end{align}
where $K=K(R,\lambda,T,\|u_0\|_1)$.

From estimates \eqref{eq1.15}, \eqref{eq1.16} and \eqref{eq1.17}, we have
\begin{equation}
\label{eq1.18}
\|V\|_{1,\frac{1}{2}}\leq  K(R,\lambda,T,\|u_0\|_1)(\|v-\overline{v}\|_{1,\frac{1}{2}}+\|v-\overline{v}\|^\beta_{1,\frac{1}{2}}+\|u_0-\overline{u}_0\|_1),
\end{equation}
with a new $K$.

Similarly, we can estimate $W$:
\begin{align*}
W(x,t)&=\int_0^t\int_\R\Gamma_{[v]}(x,t,\xi,\tau)f(\xi,\tau)-\Gamma_{[\overline{v}]}(x,t,\xi,\tau)\overline{f}(\xi,\tau)d\xi d\tau\\
\nonumber
&=\int_0^t\int_\R(\Gamma_{[v]}-\Gamma_{[\overline{v}]})(x,t,\xi,\tau)f(\xi,\tau)+\Gamma_{[\overline{v}]}(x,t,\xi,\tau)(f-\overline{f})(\xi,\tau)d\xi d\tau.
\end{align*}
Hence, using Lemma \ref{lem1.4} and \eqref{eq0.10}, we have
\begin{align}
\label{eq1.19}
|W(x,t)|&\leq\int_0^t\int_\R|(\Gamma_{[v]}-\Gamma_{[\overline{v}]})(x,t,\xi,\tau)f(\xi,\tau)|+|\Gamma_{[\overline{v}]}(x,t,\xi,\tau)(f-\overline{f})(\xi,\tau)|d\xi d\tau\\
\nonumber
&\leq\int_0^t\int_\R K(\|v-\overline{v}\|_{1,\frac{1}{2}}\|f\|_\infty+\|f-\overline{f}\|_\infty)\frac{1}{(t-\tau)^{\frac{1}{2}}}\mbox{e}^{-C\frac{(x-\xi)^2}{t-\tau}}d\xi d\tau\\
\nonumber
&\leq K(R,\lambda,T)(\|f\|_\infty+1)T(\|v-\overline{v}\|_{1,\frac{1}{2}}+\|f-\overline{f}\|_\infty).
\end{align}
Besides,
\begin{align}
\label{eq1.20}
|\partial_xW(x,t)|&\leq\int_0^t\int_\R|(\partial_x\Gamma_{[v]}-\partial_x\Gamma_{[\overline{v}]})(x,t,\xi,\tau)f(\xi,\tau)|+|\partial_x\Gamma_{[\overline{v}]}(x,t,\xi,\tau)(f-\overline{f})(\xi,\tau)|d\xi d\tau\\
\nonumber
&\leq\int_0^t\int_\R K(\|v-\overline{v}\|_{1,\frac{1}{2}}\|f\|_\infty+\|f-\overline{f}\|_\infty)\frac{1}{(t-\tau)}\mbox{e}^{-C\frac{(x-\xi)^2}{t-\tau}}d\xi d\tau\\
\nonumber
&\leq K(R,\lambda,T)(\|f\|_\infty+1)T^{\frac{1}{2}}(\|v-\overline{v}\|_{1,\frac{1}{2}}+\|f-\overline{f}\|_\infty)
\end{align}

To prove the H\"older continuity with respect to $t$, we write
\begin{align*}
&W(x,t)-W(x,t')\\
\nonumber
=&\int_{0}^t\int_\R[(\Gamma_{[v]}(x,t,\xi,\tau)]f(\xi,\tau)-\Gamma_{[\overline{v}]}(x,t,\xi,\tau)\overline{f}(\xi,\tau)d\xi d\tau\\
\nonumber
&-\int^{t'}_0\int_\R[(\Gamma_{[v]}(x,t',\xi,\tau)]f(\xi,\tau)-\Gamma_{[\overline{v}]}(x,t,\xi',\tau)\overline{f}(\xi,\tau)d\xi d\tau\\
\nonumber
=&\int_{t'}^t\int_\R(\Gamma_{[v]}(x,t,\xi,\tau)]f(\xi,\tau)-\Gamma_{[\overline{v}]}(x,t,\xi,\tau)\overline{f}(\xi,\tau))d\xi d\tau\\
\nonumber
+&\int_0^{t'}\int_\R[(\Gamma_{[v]}(x,t,\xi,\tau)-\Gamma_{[v]}(x,t',\xi,\tau))f(\xi,\tau)-(\Gamma_{[\overline{v}]}(x,t,\xi,\tau)-\Gamma_{[\overline{v}]}(x,t',\xi,\tau))\overline{f}(\xi,\tau)]d\xi d\tau\\
\nonumber
=&\int_{t'}^t\int_\R((\Gamma_{[v]}-\Gamma_{[\overline{v}]})(x,t,\xi,\tau)]f(\xi,\tau)+\Gamma_{[\overline{v}]}(x,t,\xi,\tau)(f-\overline{f})(\xi,\tau))d\xi d\tau\\
\nonumber
+&\int_{t'-\epsilon}^{t'}\int_\R[(\Gamma_{[v]}(x,t,\xi,\tau)-\Gamma_{[v]}(x,t',\xi,\tau))f(\xi,\tau)-(\Gamma_{[\overline{v}]}(x,t,\xi,\tau)-\Gamma_{[\overline{v}]}(x,t',\xi,\tau))\overline{f}(\xi,\tau)]d\xi d\tau\\
\nonumber
&+\int_0^{t'-\epsilon}\int_\R\int_{t'}^t[\partial_t\Gamma_{[v]}(x,\xi,s,\tau)f(\xi,\tau)-\partial_t\Gamma_{[\overline{v}]}(x,\xi,s,\tau)\overline{f}(\xi,\tau)]dsd\xi d\tau,\\
\nonumber
\equiv& W_1+W_2+W_3
\end{align*}
where $0<\epsilon<t'$ is arbitrary.
Using Lemma \ref{lem1.4} and \eqref{eq1.10}, we estimate
\begin{align}
\label{eqw1}
|W_1|&\leq\int_{t'}^t\int_\R (K\|v-\overline{v}\|_{1,\frac{1}{2}}\|f\|_\infty+K\|f-\overline{f}\|_\infty)\frac{\mbox{e}^{-C\frac{(x-\xi)^2}{t-\tau}}}{(t-\tau)^{\frac{1}{2}}}d\xi d\tau\\
\nonumber
&\leq K(R,\lambda,T)(\|f\|_\infty+1)T^{\frac{1}{2}}(\|v-\overline{v}\|_{1,\frac{1}{2}}+\|f-\overline{f}\|_\infty)(t-t')^{\frac{1}{2}}
\end{align}
Regarding $W_2$, we apply \eqref{eq1.10} to get
\begin{align}
 \label{eqw2}
 |W_2|&\leq\int_{t'-\epsilon}^{t'}\int_\R\left(\frac{K}{(t-\tau)^{\frac{1}{2}}}\mbox{e}^{-C\frac{(x-\xi)^2}{t-\tau}}+\frac{K}{(t'-\tau)^{\frac{1}{2}}}\mbox{e}^{-C\frac{(x-\xi)^2}{t'-\tau}} \right)(\|f\|_\infty+\|\overline{f}\|_\infty) d\xi d\tau\\
 \nonumber
 &\leq K(\|f\|_\infty+\|\overline{f}\|_\infty)\epsilon.
\end{align}
The term $W_3$ can be estimated using Remark \ref{ob0.1} as follows:
\begin{align*}
W_3&=\int_0^{t'-\epsilon}\int_\R\int_{t'}^t[\partial_t\Gamma_{[v]}(x,\xi,s,\tau)f(\xi,\tau)-\partial_t\Gamma_{[\overline{v}]}(x,\xi,s,\tau)\overline{f}(\xi,\tau)]dsd\xi d\tau\\
 \nonumber
 &=\int_0^{t'-\epsilon}\int_{t'}^t\int_\R[(\partial_t\Gamma_{[v]}-\partial_t\Gamma_{[\overline{v}]})(x,\xi,s,\tau)(f(\xi,\tau)-f(x,\tau))\\
 \nonumber
 &\hspace{2.5cm}+\partial_t\Gamma_{[\overline{v}]}(x,\xi,s,\tau)((f-\overline{f})(\xi,\tau)-(f-\overline{f})(x,\tau))]d\xi ds d\tau
\end{align*}
Now, applying Lemma \ref{lem1.4} and \eqref{eq1.10}, and writing $K_1= K(\|v-\overline{v}\|_{1,\frac{1}{2}}+\|v-\overline{v}\|_{1,\frac{1}{2}}^\beta)\|f\|_{1,\frac{1}{2}}$, it follows that
\begin{align}
\label{eqw3}
|W_3|&\leq \int_0^{t'-\epsilon}\int_{t'}^t\int_\R K_1( \frac{1}{|x-\xi|^{\gamma(1-\beta)}(s-\tau)^{\frac{2-\gamma(1-\beta)}{2}}}+\frac{1}{(s-\tau)^{\frac{3}{2}}})\mbox{e}^{-C\frac{(x-\xi)^2}{s-\tau}}|x-\xi|\\
\nonumber
&\hspace{2.5cm}+\frac{K}{(s-\tau)^{\frac{3}{2}}}\mbox{e}^{-C\frac{(x-\xi)^2}{s-\tau}}\|f-\overline{f}\|_{1,\frac{1}{2}}|x-\xi|d\xi ds d\tau\\
\nonumber
&\leq K(1+\|f\|_{1,\frac{1}{2}})[\|v-\overline{v}\|_{1,\frac{1}{2}}+\|v-\overline{v}\|_{1,\frac{1}{2}}^\beta\\
\nonumber
&+\|f-\overline{f}\|_{1,\frac{1}{2}}]\int_0^{t'-\epsilon}\int_{t'}^t\int_\R ( \frac{|x-\xi|^{1-\gamma(1-\beta)}}{(s-\tau)^{\frac{2-\gamma(1-\beta)}{2}}}+\frac{|x-\xi|}{(s-\tau)^{\frac{3}{2}}})\mbox{e}^{-C\frac{(x-\xi)^2}{s-\tau}}d\xi ds d\tau\\
\nonumber
&\leq K(1+\|f\|_{1,\frac{1}{2}})[\|v-\overline{v}\|_{1,\frac{1}{2}}+\|v-\overline{v}\|_{1,\frac{1}{2}}^\beta\\
\nonumber
&+\|f-\overline{f}\|_{1,\frac{1}{2}}]\int_0^{t'-\epsilon}\int_{t'}^t\int_\R (\frac{1}{(s-\tau)^{\frac{1}{2}}}+\frac{1}{s-\tau})\mbox{e}^{-C\frac{(x-\xi)^2}{s-\tau}}d\xi ds d\tau\\
\nonumber
&\leq K(1+\|f\|_{1,\frac{1}{2}})[\|v-\overline{v}\|_{1,\frac{1}{2}}+\|v-\overline{v}\|_{1,\frac{1}{2}}^\beta\\
\nonumber
&+\|f-\overline{f}\|_{1,\frac{1}{2}}](T^{\frac{1}{2}}+1)\int_0^{t'-\epsilon}\int_{t'}^t\int_\R \frac{1}{s-\tau}\mbox{e}^{-C\frac{(x-\xi)^2}{s-\tau}}d\xi ds d\tau\\
\nonumber
&\leq K(1+\|f\|_{1,\frac{1}{2}})[\|v-\overline{v}\|_{1,\frac{1}{2}}+\|v-\overline{v}\|_{1,\frac{1}{2}}^\beta+\|f-\overline{f}\|_{1,\frac{1}{2}}]\int_0^{t'-\epsilon}\int_{t'}^t \frac{1}{(s-\tau)^{\frac{1}{2}}} ds d\tau\\
\nonumber
&\leq K(1+\|f\|_{1,\frac{1}{2}})[\|v-\overline{v}\|_{1,\frac{1}{2}}+\|v-\overline{v}\|_{1,\frac{1}{2}}^\beta+\|f-\overline{f}\|_{1,\frac{1}{2}}]T(t-t')^{\frac{1}{2}},
\end{align}
where for the last inequality we used that \eqref{eqw2} is true for all $\epsilon\in (0,t')$. \
From \eqref{eqw1}, \eqref{eqw2} and \eqref{eqw3}, we conclude that
\begin{equation}
\label{eq1.21}
\begin{array}{rl}
     &|W(x,t)-W(x,t')|\\
\leq & K(\|f\|_{1,\frac{1}{2}}+1)(\|v-\overline{v}\|_{1,\frac{1}{2}}+\|v-\overline{v}\|_{1,\frac{1}{2}}^\beta +\|f-\overline{f}\|_{1,\frac{1}{2}})T^{\frac{1}{2}}(t-t')^{\frac{1}{2}},
\end{array}
\end{equation}
where $K=K(R,\lambda,T)$. 
 It follows from \eqref{eq1.19}, \eqref{eq1.20} and \eqref{eq1.21} that
\begin{equation}
\label{eq1.22}
\|W\|_{1,\frac{1}{2}}\leq K(R,\lambda,T)T^{\frac{1}{2}}(\|f\|_{1,\frac{1}{2}}+1)(\|v-\overline{v}\|_{1,\frac{1}{2}}+\|v-\overline{v}\|_{1,\frac{1}{2}}^\beta +\|f-\overline{f}\|_{1,\frac{1}{2}}).
\end{equation}
Finally, from \eqref{eq1.18} and \eqref{eq1.22}, we have
\begin{equation}
\label{eq1.23}
\begin{array}{rl}
     \|u-\overline{u}\|_{1,\frac{1}{2}}&\leq \  \|V\|_{1,\frac{1}{2}}+\|W\|_{1,\frac{1}{2}}\\
&\leq \
K(\|v-\overline{v}\|_{1,\frac{1}{2}}+\|v-\overline{v}\|_{1,\frac{1}{2}}^\beta+\|u_0-\overline{u}_0\|_1\\  & \ \ \ \ \ + \ T^{\frac{1}{2}}(\|f\|_{1,\frac{1}{2}}+1)(\|f-\overline{f}\|_{1,\frac{1}{2}}+\|v-\overline{v}\|_{1,\frac{1}{2}}+\|v-\overline{v}\|_{1,\frac{1}{2}}^\beta)),
\end{array}
\end{equation}
where $K=K(R,\lambda,T,\|u_0\|_1)$. 
\ \ \end{proof}

In particular we have the following estimate for a solution of \eqref{eq1.14}
\begin{corollary}
\label{cor1.1}
In the same conditions of Theorem \ref{teo1.1}, if $u$ is a solution of \eqref{eq1.14}
then
\begin{equation}
\|u\|_{1,\frac{1}{2}}\leq
 K(R,\lambda,T,\|u_0\|_1)(\|u_0\|_1+T^{\frac{1}{2}}(\|f\|_{1,\frac{1}{2}}+1)\|f\|_{1,\frac{1}{2}}),
  \end{equation}
where $K=K(R,\lambda,T,\|u_0\|_1)$. 
\end{corollary}

\begin{proof}The proof follows from Theorem \ref{teo1.1} by taking
 $\overline{v}=v$, $\overline{f}=2f$ and $\overline{u}_0=2u_0$.~\end{proof}

\section{Local solution}
\label{local solution}

In this section we prove Theorem \ref{local}. For simplicity we shall write $f$ and $f_i$ instead of $\tilde{f}$ and $\tilde{f}_i$, respectively. 

Consider the operator $\mathcal A$ given in \eqref{eq4}. 
In the lemma below we construct an invariant set for $\mathcal A$.

\begin{lemma}
\label{lem2.1}
Let $R_i=\frac{\lambda_i+c_i}{a_i}\left(1+\frac{2b_i}{a_i}\|y_{i,0}\|_1\right)$, $K(R_i,\frac{\lambda_i}{a_i+b_i\|y_{i,0}\|_\infty},T,\|u_{i,0}\|_1)$ be the constant given in the Corollary \ref{cor1.1}, $K_i=\sup_{0\leq T\leq 1}K(R_i,\frac{\lambda_i}{a_i+b_i\|y_{i,0}\|_\infty},T,\|u_{i,0}\|_1) $,  $M_i>K_i\|u_{i,0}\|_1$ and $\Sigma=\{(u_1,u_2)\in \left(C^{1,\frac{1}{2}}(\R\times[0,T])\right)^2;\, \|u_i\|_{1,\frac{1}{2}}\leq M_i\}$. 
Then ${\mathcal A}(\Sigma)\subset\Sigma$, if $T>0$ is sufficiently small.
\end{lemma}

\begin{proof}
Given $(u_1,u_2)\in\Sigma$ we get $y_i\equiv y_i(u_i)$ explicitly solving
${(y_i)_t}=-A_iy_if(u_i)$, i.e.
\begin{equation}
\label{eq2.6}
  y_i(x,t)=y_{i,0}(x)\mbox{e}^{-A_i \int_0^tf(u_i(x,s))ds}
\end{equation}
and, so, $0\leq y_i\leq\|y_{i,0}\|_\infty$, and else,
\begin{align}
\label{eq2.7}
\|y_i\|_{1,\frac{1}{2}}&\leq \| y_{i,0}\|_1\left\|\mbox{e}^{-A_i \int_0^tf(u_i(x,s))ds}\right\|_{1,\frac{1}{2}}\\
\nonumber
&={\|y_{i,0}\|}_1(\sup_{(x,t)\in\Omega_T}(\mbox{e}^{-A_i\int_0^tf(u_i(x,s))ds})\\
\nonumber
&\hspace{2.3cm}+\sup_{(x,t),(\overline{x},\overline{t})\in\Omega_T}(\frac{|\mbox{e}^{-A_i\int_0^tf(u_i(x,s))ds}-\mbox{e}^{-A_i\int_0^{\overline{t}}f(u_i(\overline{x},s))ds}|}{|x-\overline{x}|+|t-\overline{t}|^{\frac{1}{2}}}))\\
\nonumber
&\leq{\|y_{i,0}\|}_1(1+\sup_{(x,t),(\overline{x},\overline{t})\in\Omega_T}(\frac{|\mbox{e}^{-A_i\int_0^tf(u_i(x,s))ds}-\mbox{e}^{-A_i\int_0^{t}f(u_i(\overline{x},s))ds}|}{|x-\overline{x}|})\\
\nonumber
&\hspace{2.3cm}+\sup_{(x,t),(\overline{x},\overline{t})\in\Omega_T}(\frac{|\mbox{e}^{-A_i\int_0^tf(u_i(\overline{x},s))ds}-\mbox{e}^{-A_i\int_0^{\overline{t}}f(u_i(\overline{x},s))ds}|}{|t-\overline{t}|^{\frac{1}{2}}}))\\
\nonumber
&\leq{\|y_{i,0}\|}_1(1+\sup_{(x,t),(\overline{x},\overline{t})\in\Omega_T}(A_i\int_0^t\frac{|f(u_i(x,s)-f(u_i(\overline{x},s)|}{|x-\overline{x}|}ds)\\
\nonumber
&\hspace{2.3cm}+\sup_{(x,t),(\overline{x},\overline{t})\in\Omega_T}(\frac{A_i|\int_{\overline{t}}^tf(u_i(\overline{x},s))ds|}{|t-\overline{t}|^{\frac{1}{2}}}))\\
\nonumber
&\leq{\|y_{i,0}\|}_1(1+\sup_{(x,t),(\overline{x},\overline{t})\in\Omega_T}(A_i\int_0^t\|f'\|_\infty\|u_i\|_{1,\frac{1}{2}}ds)\\
&\hspace{3.3cm}+\sup_{(x,t),(\overline{x},\overline{t})\in\Omega_T}(A_i\|f\|_\infty|t-\overline{t}|^{\frac{1}{2}}))\\
\nonumber
&\leq  \|y_{i,0}\|_1(1+TA_i\|u_i\|_{1,\frac{1}{2}}\|f'\|_\infty+T^{\frac{1}{2}}A_i)\\
\nonumber
&\leq  \| y_{i,0}\|_1(1+TA_iM_i\|f'\|_\infty+T^{\frac{1}{2}}A_i)\leq  2\| y_{i,0}\|_1,
\end{align}
where, for the last inequality, we took $T>0$ sufficiently small such that\break\hfill 
$T^{\frac{1}{2}}+TA_iM_i\|f'\|_\infty\leq 1$. Furthermore, 
$v_i\equiv v_i(u_i)=\left(\frac{\lambda_i}{a_i+b_iy_i(u_i)},\frac{c_i}{a_i+b_iy_i(u_i)},0\right)
\in B(R_i,\frac{\lambda_i}{a_i+b_i\|y_{i,0}\|_\infty},1)$;
see the definition of the set $B(R,\lambda,\alpha)$ at the beginning of section \ref{dependence_parameters}. Indeed,
\begin{align*}
&\|\frac{\lambda_i}{a_i+b_iy_i(u_i)}\|_{1,\frac{1}{2}}\\
=&\sup_{(x,t)\in\Omega_T}|\frac{\lambda_i}{a_i+b_iy_i(u_i)}|+\sup_{(x,t),(\overline{x},\overline{t})\in\Omega_T}\frac{|\frac{\lambda_i}{a_i+b_iy_i(u_i(x,t))}-\frac{\lambda_i}{a_i+b_iy_i(u_i(\overline{x},\overline{t}))}|}{|x-\overline{x}|+|t-\overline{t}|^{\frac{1}{2}}}\\
 \nonumber
 \leq& \frac{\lambda_i}{a_i}+\frac{b_i\lambda_i}{a_i^2}(\sup_{(x,t),(\overline{x},\overline{t})\in\Omega_T}\frac{|y_i(u_i(x,t))-y_i(u_i(\overline{x},\overline{t}))|}{|x-\overline{x}|+|t-\overline{t}|^{\frac{1}{2}}})\\
 \nonumber
 \leq& \frac{\lambda_i}{a_i}+\frac{b_i\lambda_i}{a_i^2}\|y_i\|_{1,\frac{1}{2}}\leq \frac{\lambda_i}{a_i}+\frac{2b_i\lambda_i\|y_{i,0}\|_1}{a_i^2},
 \end{align*}
where we used \eqref{eq2.7}.
Analogously, we can verify that $\|\frac{c_i}{a_i+b_iy_i(u_i)}\|_{1,\frac{1}{2}}\leq \frac{c_i}{a_i}+\frac{2b_ic_i\|y_{i,0}\|_1}{a_i^2}$ and, so,
$$
\|\frac{\lambda_i}{a_i+b_iy_i(u_i)}\|_{1,\frac{1}{2}}+\|\frac{c_i}{a_i+b_iy_i(u_i)}\|_{1,\frac{1}{2}}\leq \frac{\lambda_i+c_i}{a_i}(1+\frac{2b_i\|y_{i,0}\|_1}{a_i})=R_i
$$
Adding to the fact that $0<\frac{\lambda_i}{a_i+b_i\|y_{i,0}\|_\infty}\leq\frac{\lambda_i}{a_i+b_iy_i(u_i)}$, we conclude that \\ $v_i(u_i)\in B\left(R_i,\frac{\lambda_i}{a_i+b_i\|y_{i,0}\|_\infty},1\right)$.

From the above, the hypotesis of Theorem \ref{teo0.1} are satisfied. Therefore, the problem
\begin{equation}
\left\{\begin{array}{l}
\mathcal{L}_{[v_i(u_i)]}(w_i)=f_i(y_i,u_1,u_2),\quad\mbox{in}\quad \R\times(0,T],\\
w_i(x,0)=u_{i,0}(x),\quad x\in\R,
\end{array}\right.
\end{equation}
has a unique solution with an exponetial growth in the space $C^{2,1}(\R\times(0,T])\cap C(\R\times[0,T])$.\\

From Corollary \ref{cor1.1}, we have
\begin{align*}
&\|w_i\|_{1,\frac{1}{2}}\\
\leq &K(R_i,\frac{\lambda_i}{a_i+b_i\|y_{i,0}\|_\infty},T,\|u_{i,0}\|_1)[\|u_{i,0}\|_1\\
&+T^{\frac{1}{2}}(\|f_i(y_i,u_1,u_2)\|_{1,\frac{1}{2}}+1)\|F_i(y_i,u_1,u_2)\|_{1,\frac{1}{2}}]\\
\nonumber
\leq &K_i\left[\|u_{i,0}\|_1+T^{\frac{1}{2}}(\|f_i(y_i,u_1,u_2)\|_{1,\frac{1}{2}}+1)\|F_i(y_i,u_1,u_2)\|_{1,\frac{1}{2}}\right]\\
\nonumber
\leq &K_i[\|u_{i,0}\|_1+T^{\frac{1}{2}}\overline{K}(M_1,M_2,\|y_{i,0}\|_1)]\leq M_i,
\end{align*}
provided that $T$ is sufficiently small. \ \ \end{proof}

\subsection{Proof of Theorem \ref{local}}
\label{local proof}

Let $T$, $\mathcal A$ and $\Sigma$ be as in Lemma \ref{lem2.1}, and for any fixed 
$(u_1^{0},u_2^{0})\in\Sigma$, let $(u_1^{n},u_2^{n})$, $n=1,2,\cdots$, be the sequence defined by 
$(u_1^{n},u_2^{n})={\mathcal A}(u_1^{n-1},u_2^{n-1})$. 
From Lemma \ref{lem2.1} we have that this sequence is bounded in $C^{1,\frac{1}{2}}(\Omega_T)$ ($\Omega_T=\R\times (0,T)$). Then, by Arzel\`a-Ascoli's theorem (see \cite[p. 635]{Eva}), there exists a 
$(u_1,u_2)\in \Sigma$ and a subsequence of $(u_1^{n},u_2^{n})$,
which we shall still denote by $(u_1^{n},u_2^{n})$, such that it converges to
$(u_1,u_2)$, uniformly in compacts sets in $\R\times[0,T]$. By Theorem \ref{teo0.1} we can write
\begin{equation}
\label{eq2.8}
\begin{array}{rl}
u_i^{n+1}(x,t)=&\int\Gamma_{[v_i(u_i^{n})]}(x,t,\xi,0)u_{i,0}(\xi)d\xi\\
&+\int_0^t\int\Gamma_{[v_i(u_i^{n})]}(x,t,\xi,\tau)f_i(y_i(u_i^{n}),u_1^{n},u_2^{n})(\xi,\tau)d\xi d\tau,
\end{array}
\end{equation}
with
\begin{equation}
\label{eq2.9}
v_i(u_i^{n})=\left(\frac{\lambda_i}{a_i+y_i(u_i^{n})},\frac{c_i}{a_i+y_i(u_i^{n})},0\right)
\end{equation}
and
\begin{equation}
y_i(u_i^{n})(x,t)=y_{i,0}(x)\mbox{e}^{-A_i\int_0^tf(u_i^{n}(x,s))ds}.
\end{equation}
As $u_i^{n}$ converges to $u_i$, we have that
$y_i(u_i^{n})$, $v_i(u_i^{n})$ and $f_i(y_i(u_i^{n}),u_1^{n},u_2^{n})$ converge to $y_i(u_i)$, $v_i(u_i)$ and $f_i(y_i(u_i),u_1,u_2)$, respectively.
Such convergences are uniform on compacts sets in $\R\times [0,T]$, because $u_i^{n}$ so
converges, $f'$ is bounded on $\R$, and $\nabla f_i$ is bounded on $[0,M_1]\times[0,M_2]\times[0,\|y_{i,0}\|_\infty]$.
Moreover, as $\|u_i^{n}\|_{1,\frac{1}{2}}\leq M_i$, for all $n\in\mathbb{N}$, 
 we have that $\|u_i\|_{1,\frac{1}{2}}\leq M_i$ and so, $v_i(u_i^{n}),v_i(u_i)\in B\left(R_i, \frac{\lambda_i}{a_i+b_i\|y_{i,0}\|_\infty},1\right)$.
 From Lemma \ref{lem1.5} we have that $\Gamma_{[v_i(u_i^{n})]}$
converges to $\Gamma_{[v_i(u_i)]}$ pointwise.\break\hfill 
As \ $|\Gamma_{[v_i(u_i^{n})]}(x,t,\xi,0)u_{i,0}(\xi)|\leq Kt^{-\frac{1}{2}}\mbox{e}^{-C\frac{(x-\xi)^2}{t}}$
\ and  
$$
|\Gamma_{[v_i(u_i^{n})]}(x,t,\xi,\tau)f_i(y_i(u_i^{n}),u_1^{n},u_2^{n})(\xi,\tau)|
\leq K(t-\tau)^{-\frac{1}{2}}\mbox{e}^{-C\frac{(x-\xi)^2}{t-\tau}},
$$
where $K$ and $C$ are constants that do depend on $n$ (see Corollary \ref{uniform L}), it follows, by the Lebesgue's dominated convergence theorem, that
\begin{equation}
\label{integral representation for ui}
u_i(x,t)=\int\Gamma_i(x,t,\xi,0)u_{i,0}(\xi)d\xi+\int_0^t\int\Gamma_i(x,t,\xi,\tau)
f_i(y_i,u_1,u_2)(\xi,\tau)d\xi d\tau.
\end{equation}
where $\Gamma_i$ is the fundamental solution to the equation 
$(w_i)_t-\alpha_i(y_i)(w_i)_{xx}+\beta_i(y_i)(w_i)_x=0$, whith 
$y_i\equiv y_i(x,t)=y_{i,0}(x)\mbox{e}^{-A_i \int_0^tf(u_i(x,\tau))d\tau}$.
Then, by Theorem \ref{teo0.1}, $u=(u_1,u_2)$ is a solution of the
system \eqref{eq1-general-2}--\eqref{in cond yi}, with  $u_i\in C^{2,1}(\R\times(0,T])\cap C^{1,\frac{1}{2}}(\R\times[0,T])$. 

To obtain that $u$ is in the sector $\langle 0,\varphi\rangle_T$, by what we discussed in the
Introduction (see p. \pageref{Cauchy}) we need to show the continuous dependence of the solution of
the Cauchy problem \eqref{Cauchy} with respect to reaction functions $\tilde{f}_i$ (here, denoted
simply by $f_i$) i.e. (more precisely) that the solution $u^{\delta}=(u_1^{\delta},u_2^{\delta})$, 
$\delta>0$, of 
\begin{equation}
\label{Cauchy w delta}
\left\{\begin{array}{ll}
(w_i)_t-\alpha_i(y_i)(w_i)_{xx}+\beta_i(y_i)(w_i)_x
=f_i^\delta(y_i,w_1,w_2), &
x\in\R, \ t>0\\
w_i(x,0)=u_{i,0}(x), & x\in\R,
\end{array}
\right.
\end{equation}
where $f_i^\delta(y_i,w_1,w_2):=f_i(y_i,w_1,w_2)\pm\delta$, 
$y_i\equiv y_i(x,t)=y_{i,0}(x)\mbox{e}^{-A_i \int_0^tf(u_i(x,\tau))d\tau}$, converges pointwise to 
$u$ when $\delta\to0+$, and, that all hypotheses of Corollary \ref{comparison cor} are fulfilled. 

Let us first observe that $\hat{u}=(0,0)$ and $\tilde{u}=(\varphi,\varphi)$, where 
(see p. \pageref{upper}) $\varphi(t)=(M+\beta)\mbox{e}^{\alpha t}-\beta$ \ 
(being $M=\max_{i=1,2}\|u_{i,0}\|_\infty$, 
$\alpha=\max_{i=1,2}\{\frac{A_ib_i\|y_{i,0}\|_\infty}{a_i}\}$ and  
$\beta=\max_{i=1,2}\{\frac{d_i}{A_ib_i}\}$) are a pair of lower and upper solutions to
the system\break\hfill $\mathcal{L}_i(w_i)\equiv(w_i)_t-\alpha_i(y_i)(w_i)_{xx}+\beta_i(y_i)(w_i)_x
=f_i(y_i,w_1,w_2)\equiv f_i(x,t,w_1,w_2)$\break\hfill occurring in \eqref{Cauchy} (i.e. the system in \eqref{Cauchy w delta} without delta). \label{pair}
(See Lemma 2 in \cite{js}.) Indeed, it is obvious that $\hat{u}=(0,0)$ is a lower solution 
(in fact, a\break\hfill solution) to this system, since $f_i(0,0)=0$. Regarding $\tilde{u}=(\varphi,\varphi)$, notice that 
$f_i=\dfrac{b_iA_iw_i+d_i}{a_i+b_iy_i}y_i\tilde{f}(w_i)
   \le\dfrac{b_iA_iw_i+d_i}{a_i}\|y_{i,0}\|_\infty$ 
when $w_1=w_2$ and $w_i\ge0$ (recall that $\tilde{f}$ is the function that coincides with the 
{\lq\lq}Arrhenius function{\rq\rq} $\mbox{e}^{-\frac{E}{s}}$ for $s>0$ and vanishes for $s\le0$) 
and $\mathcal{L}_i(\varphi)=\varphi'(t)=\alpha(M+\beta)\mbox{e}^{\alpha t}$, so,
\begin{align*}
\mathcal{L}_i(\varphi)(x,t)-f_i(x,t,\varphi,\varphi)
&\ge \alpha(M+\beta)\mbox{e}^{\alpha t}-\dfrac{A_ib_i\varphi(t)+d_i}{a_i}\|y_{i,0}\|_\infty\\
&  = (M+\beta)(\alpha-\dfrac{A_ib_i\|y_{i,0}\|_\infty}{a_i})\mbox{e}^{\alpha t} 
      +\dfrac{A_ib_i}{a_i}(\beta-\dfrac{d_i}{A_ib_i})\|y_{i,0}\|_\infty\\
&\ge0.			
\end{align*}

Next, as we noticed in the Introduction, we observe that $f_i$ is increasing with respect to $w_j$ 
($i,j=1,2$; $j\not=i$) for $(a_i+b_iy_i)\partial f_i/\partial w_j=q>0$. On the other hand,
$(a_i+b_iy_i)|\partial f_i/\partial w_i|
=   |b_iA_iy_i\tilde{f}(w_i)+(b_iA_iw_i+d_i)y_i\tilde{f}'(w_i)-q|
\le b_iA_iy_i+k_i(b_iA_i+d_i)y_i+q
\le b_iA_i\|y_{i,0}\|_\infty+k_i(b_iA_i+d_i)\|y_{i,0}\|_\infty+q$, where $k_i$ is some positive
constant, so $|\partial f_i/\partial w_i|$ is bounded by a constant, i.e. $f_i$ is uniformly 
lipschitz continuous in the variable $w_i$, and thus the 
{\lq\lq}semi-lipschitz{\rq\rq} condition \eqref{semi lip} is satisfied with $c_i$ being a constant,
for an arbitrary $\varepsilon_0$ (in the notation of Theorem \eqref{comparison}).  Concerning the
condition \eqref{cond 2}, we have $f_i\big|_{w_j=u_j}^{w_j=s+u_j}=sq$, so, it is satisfied with 
any $\varepsilon_0<\delta/q$ and $\delta'=\varepsilon_0q$. 

Now, we notice that both the lower solution $\hat{u}=(0,0)$ and the upper solution 
$\tilde{u}=(\varphi,\varphi)$ satisfy trivially the condition \eqref{growth}, since their
components are non negative functions. As for $u$, using the integral representation 
\eqref{integral representation for ui}, we also see easily that it satisfies \eqref{growth}, since
the first part\break $\int\Gamma_i(x,t,\xi,0)u_{i,0}(\xi)d\xi$
is non negative ($\Gamma_i,u_{i,0}\ge0$) and the modulus of the second part \break
$\int_0^t\int\Gamma_i(x,t,\xi,\tau)f_i(y_i,u_1,u_2)(\xi,\tau)d\xi d\tau$ can be estimated by
a constant times $t$, because $u$ is bounded and $\int\Gamma_i d\xi=1$ (see Remark \ref{ob0.1}). It remains to show the continuous dependence, i.e. that $u^\delta$ 
converges pointwise to $u$, but up to here, we can conclude, by Theorem \ref{comparison}, that $u^\delta\in \langle 0,\varphi\rangle_T$. In particular,
$u^\delta$ is bounded, uniformly with respect to $\delta$. 

To show the continuous dependence, using the integral representation 
\eqref{integral representation}, with $\Gamma_i$ being the fundamental solution to the equation $(w_i)_t-\alpha_i(y_i)(w_i)_{xx}+\beta_i(y_i)(w_i)_x=0$, and again that\break
$\int_0^t\int_{\R}\Gamma_i(x,t,\xi,\tau)\xi d\xi d\tau=1$ (see Remark \ref{ob0.1}), we have
$$
(u_i-u_i^\delta)(x,t)
=\int_0^t\int_{\R}\Gamma_i(x,t,\xi,\tau)[f(y_i,u_1,u_2)-f(y_i,u_1^\delta,u_2^\delta)](\xi,\tau)d\xi d\tau
\pm \delta t
$$ 
thus, using the lipschitz continuity of $f_i$
in bounded sets (recall that $u$ is bounded and $u^\delta$ is in the sector $\langle 0,\varphi\rangle_T$;
the latter being a consequence of Theorem \ref{comparison}) we obtain
$\sup_x|(u-u^\delta)(x,t)| \le K\int_0^t\sup_x|(u-u^\delta)(x,\tau)| d\tau + \delta T$, so, by Gronwall's lemma,
$\sup_x|(u-u^\delta)(x,t)|\le \delta Te^{KT}$, for some constant $K$. 
This shows that\break\hfill
$\lim_{\delta\to0+}u^\delta=u$ pointwise (in fact, uniformly) in $\Omega_T=\R\times (0,T)$. 

Now it remains to show the $L^p$ assertion (the last assertion) in Theorem \ref{local}. 
 This is essentially a consequence of the {\lq\lq}generalized Young's inequality{\rq\rq} \cite[p. 9]{Fol1} and the fact that the fundamental solution $\Gamma_{[v_i(u_i^{n})]}$ is a {\lq\lq}regular kernel{\rq\rq}, uniformly with respect to $n$. More precisely, we shall show in the next paragraph 
that there exist positive numbers $\overline{T}\leq T$ and $S$ such that, if $\|u_i^n(.,t)\|_{L^p}\leq S$ for all $t\in[0,\overline{T}]$ then $\|u_i^{n+1}(.,t)\|_{L^p}\leq S$ for all $t\in[0,\overline{T}]$ as well. Then the assertion follows by Banach-Alaoglu's theorem. 

From \eqref{eq2.8}, the {\lq\lq}generalized Young's inequality{\rq\rq} \cite[p. 9]{Fol1} and 
the Minkowski's ineguality for integrals \cite[p. 194]{Fol}), we have
\begin{align}
\nonumber
&\|u_i^{n+1}(.,t)\|_{L^p}\\
\nonumber
\leq &\,\|\int\Gamma_{[v_i(u_i^{n})]}(\cdot,t,\xi,0)u_{i,0}(\xi)d\xi\|_{L^p}\\
\nonumber
&+\|\int_0^t\int\Gamma_{[v_i(u_i^{n})]}(.,t,\xi,\tau)f_i(u_1^n,u_2^n,y_i(u_i^n))(\xi,\tau)d\xi d\tau\|_{L^p}\\
\nonumber
\leq&\,(\sup_\xi\int |\Gamma_{[v_i(u_i^{n})]}(x,t,\xi,0)|dx)\,\|u_{i,0}\|_{L^p}\\
\nonumber
&+\int_0^t(\sup_\xi\int |\Gamma_{[v_i(u_i^{n})]}(x,t,\xi,\tau)|dx)\,\|f_i(u_1^n,u_2^n,y_i(u_i^n))(\cdot,\tau)\|_{L^p} d\tau\\
\nonumber
\leq&\,(\sup_\xi\int \frac{K}{t^{\frac{1}{2}}}\mbox{e}^{-C\frac{(x-\xi)^2}{t}}dx)\,\|u_{i,0}\|_{L^p}\\
\nonumber
&+\int_0^t(\sup_\xi\int \frac{K}{(t-\tau)^{\frac{1}{2}}}\mbox{e}^{-C\frac{(x-\xi)^2}{t-\tau}}dx)\,\|f_i(u_1^n,u_2^n,y_i(u_i^n))(\cdot,\tau)\|_{L^p} d\tau\\
\nonumber
=&\,\overline{K}\|u_{i,0}\|_{L^p}
+\overline{K}\int_0^t\|f_i(u_1^n,u_2^n,y_i(u_i^n))(\cdot,\tau)\|_{L^p} d\tau\\
\nonumber
&\hspace{1.0cm} (\overline{K}\equiv K\int \mbox{e}^{-Cx^2}dx)
\end{align}
\begin{align}
\nonumber
\leq&\, \overline{K}\|u_{i,0}\|_{L^p}
+\overline{K}(\ \sup_{\mbox{\scriptsize in a compact set}}|\nabla f_i|)\int_0^t(\|u_1^n(\cdot,\tau)\|_{L^p}+ \|u_2^n(\cdot,\tau)\|_{L^p})d\tau\\
\nonumber
\leq&\, \frac{S}{2}
+\overline{K}(\ \sup_{\mbox{\scriptsize in a compact set}}|\nabla f_i|)2S\overline{T}\\
\nonumber
&\hspace{1.0cm} \mbox{ if } S\ge 2\overline{K}\|u_{i,0}\|_{L^p}
 \mbox{ and } \|u_i^n(\cdot,t)\|_{L^p}\le S \mbox{ for all }
t\in [0,\overline{T}]\\
\nonumber
\leq&\, S\,, \hspace{0.5cm} \mbox{ if } \overline{T}\le 1/4\overline{K}(\ \sup_{\mbox{\scriptsize in a compact set}}|\nabla f_i|).
\end{align}
This ends the proof of Theorem \ref{local}. 

\section{Proofs of theorems \ref{invariant} and \ref{comparison} and other results}
\label{invariant and comparison}

In this section we are concerned with general parabolic operators $\mathcal{L}_i$ given by \eqref{general rd op}. We prove theorems \ref{invariant} and 
\ref{comparison} and state and prove two corollaries which are version of these theorems in the case one has
continuous dependence of the solution of the system with respect to the 
reaction functions, and also make three remarks giving alternative conditions for the hypotheses of theorems \ref{invariant} and \ref{comparison}. 

We begin by giving the main idea to prove Theorem \ref{invariant}, cf. 
\cite[Theorem 4.1]{ccs} (\cite[Theorem 14.7]{smoller}). Under the hypotheses of Theorem 
\ref{invariant}, except for the condition \eqref{growth} for now, suppose for an 
arbitrary small positive number $\varepsilon$ 
($0<\varepsilon<\varepsilon_0$) there is a point $(x_0,t_0)\in\R^d\times (0,T)$ on which 
$u=(u_1,u_2)$ belongs to the boundary of the slightly enlarged quadrant  
$Q_\varepsilon:=\{u_1\ge -\varepsilon \mbox{ and } u_2\ge-\varepsilon\}$ and such that 
$u(x,t)$ belongs to its interior for all $(x,t)\in\R^d\times (0,t_0)$. 
If $u(x_0,t_0)$ belongs to the vertical part  
$\mathcal{V}_\varepsilon:=\{u_1=-\varepsilon \mbox{ and } u_2\ge -\varepsilon\}$ of the boundary 
$\partial Q_\varepsilon$, taking the equation \eqref{general rd s} at the the point 
$(x,t)=(x_0,t_0)$ we obtain $u_1=-\varepsilon$ and $\mathcal{L}_1(u)\le0$, so $f_1(x_0,t_0,-\varepsilon,u_2(x_0,t_0))+\delta=(\mathcal{L}_1(u_1)+c_1u_1)(x_0,t_0)\le -c_1\varepsilon$,
which contradicts the hypothesis $f_1(x,t,u_1,u_2)\ge0$, when $-\varepsilon_0<u_1<0$ and 
$u_2>-\varepsilon_0$, since we can take $\varepsilon\in (0,\varepsilon_0)$ sufficiently small such
that $-c_1\varepsilon<\delta$.
Analogously, we obtain a contradiction if $u(x_0,t_0)$ belongs to the horizontal  part 
$\mathcal{H}_\varepsilon:=\{u_1\ge-\varepsilon \mbox{ and } u_2=-\varepsilon\}$. Thus, the crux point of this argument is to show the existence of the point $(x_0,t_0)$ having the above properties. The idea is that if we assume that $u(x,t)$ does not belong to $Q_\varepsilon$ for all $(x,t)\in\Omega_T=\R^2\times (0,T)$  then, since at $t=0$, $u\in int.(Q_\varepsilon)$, there would exist this {\lq\lq}first point{\rq\rq} $(x_0,t_0)\in\R^d \times(0,T)$ (with $t_0>0$) on which $u$ belongs to the boundary of $Q_\varepsilon$, from thence we obtain the contradiction with
the assumption $f_i\ge0$ when $-\varepsilon_0<u_i<0$ and $u_j>-\varepsilon_0$ 
(being $j\not= i, \ i,j=1,2$). However, {\em a priori} it might occur that 
$u(x_n,t_n)\not\in Q_\varepsilon$ for a sequence of points $(x_n,t_n)\in\Omega_T$ with 
$t_n\searrow 0$ and $|x_n|\to\infty$, even though $u(x,0)\in int.(Q_\varepsilon)$ for all 
$x\in\R^d$, and in this case, this point $(x_0,t_0)$ would not exist. 
This situation is avoided with the condition \eqref{growth}.

\smallskip

\noindent
{\bf Proof of Theorem \ref{invariant}.} \ \label{proof of inv}
Let us assume there is a point $(x,t)\in \Omega_T$ such that
$u(x,t)\not\in Q_\varepsilon$ and we shall obtain a contradiction. If this is the case then, 
by the continuity of $u$, there exists another point on which $u$ belongs to 
$\mathcal{V}_\varepsilon$ or $\mathcal{H}_\varepsilon$ (defined above). 
Consider the case that $u\in\mathcal{V}_\varepsilon$ (the case $u\in\mathcal{H}_\varepsilon$ is similar). Then we define 
$t_0=\inf\{t\in (0,T); u(x,t)\in\mathcal{V}_\varepsilon \mbox{ for some } x\in\R^d\}$. 
We claim that $t_0>0$.
Let $(x_n,t_n)\in\Omega_T=\R\times(0,T)$ be a sequence with $t_n\searrow t_0$ and 
$u_1(x_n,t_n)=-\varepsilon$.
Now, from \eqref{growth} there are positive numbers $R$ and $\tau$ such $u_1(x,t)>-\varepsilon/2$
for $(x,t)\in\Omega_T$ with $|x|<R$ and $0<t<\tau$. Then if $t_0=0$, we would have 
$|x_n|\le R$ for all $n$ sufficiently large, so, by passing to some subsequence we can assume 
that $(x_n)$ converges to some $x_0\in\R$. By continuity again, we arrive at $u_1(x_0,0)=-\varepsilon$. This contradicts the hypothesis $u_1(x,0)\ge0$ for all $x\in\R^d$.
Thus, we conclude that $t_0>0$. Moreover, $u_1(x_0,t_0)=-\varepsilon$, and, as we show above
this contradicts the hypothesis $f_1\ge0$ when $-\varepsilon_0<u_1<0$ and $u_2>-\varepsilon_0$.  
\ \ 
\vbox{\hrule height0.6pt\hbox{%
\vrule height1.3ex width0.6pt\hskip0.8ex
\vrule width0.6pt}\hrule height0.6pt}

\smallskip
\noindent
{\bf Proof of Theorem \ref{comparison}.} \  
Theorem \ref{comparison} is obtained by comparison, via Theorem \ref{invariant}. 
 Indeed, in the case that $u$ is an upper solution to \eqref{general rd s 2}, defining 
$w_i=u_i-\hat{u}_i$, 
 we have 
$(\mathcal{L}_i-c_i)w_i=\mathcal{L}_i(u_i)-\mathcal{L}_i(\hat{u}_i)-c_iw_i
\ge f_i(x,t,u_1,u_2)-f_i(x,t,\hat{u}_1,\hat{u}_2)-c_iw_i+\delta
=f_i(x,t,w_1+\hat{u_1},w_2+\hat{u}_2)-f_i(x,t,\hat{u}_1,\hat{u}_2)-c_iw_i+\delta
\equiv g_i(x,t,w_1,w_2)+\delta-\delta'$, where 
$g_i(x,t,w_1,w_2)=f_i(x,t,w_1+\hat{u_1},w_2+\hat{u}_2)-f_i(x,t,\hat{u}_1,\hat{u}_2)-c_iw_i+\delta'$. 
Now, omitting the dependence on some arguments for simplicity, 
and considering the case $i=1$ (the case $i=2$ is similar), subtracting and adding the
term $f_1(\hat{u_1},w_2+\hat{u}_2)$, we have 
\begin{align*}
g_1 &= [f_1(w_1+\hat{u_1},w_2+\hat{u}_2) - f_1(\hat{u_1},w_2+\hat{u}_2)] - c_1w_1\\
& \hspace{1cm} + [f_1(\hat{u_1},w_2+\hat{u}_2) - f_1(\hat{u}_1,\hat{u}_2)]+\delta'\\
&\ge 0
\end{align*}
for all $w_1\in (-\varepsilon_0, 0)$ and $w_2\ge-\varepsilon_0$, by \eqref{semi lip}, the monotonicity of $f_1$ with respect to $u_2$, and \eqref{cond 2}. Thus we have
shown that $w_1$ satisfies all the hypotheses of Theorem \ref{invariant} with 
$w_1$, $-c_1$ and $\delta-\delta'$ in place of $u_1$, $c_1$ and $\delta$,
respectively, then, we conclude that $w_1(x,t)\ge0$, and, similarly, we can show that
$w_2(x,t)\ge0$, for all $(x,t)\in\Omega_T$.
This ends the proof of the first statement of Theorem \ref{comparison}.

Regarding the second statement, that is, the case that $u$ is a lower solution,
we observe that it reduces to the first statement by substituting $f_i$ by 
$f_i-\delta$ and taking $\tilde{u}$ in place of the $u$ in the first statement
and the $u$ in the second statement in place of $\hat{u}$.  
\ \ 
\vbox{\hrule height0.6pt\hbox{%
\vrule height1.3ex width0.6pt\hskip0.8ex
\vrule width0.6pt}\hrule height0.6pt}

\smallskip
 
\begin{corollary} {\em (Corollary of Theorem \ref{invariant}.)} \ 
\label{invariant cor}
Under the hypotheses and notations of Theorem \ref{invariant} but with $\delta=0$, suppose we have a continuous dependence of the solutions of the system
\begin{equation}
\label{general rd s wo delta}
\mathcal{L}_i(u_i)+c_iu_i=f_i(x,t,u_1,u_2)
\end{equation}
$(x,t)\in\Omega_T=\R^d\times (0,T)$, $0<T\le\infty$, with respect to the reaction functions 
$f_i$. Then the quadrant $Q=\{(u_1,u_2)\,;\, u_1\ge0, u_2\ge0\}$ is a positively invariant region to the system \eqref{general rd s wo delta}. 
More precisely, let $u=(u_1,u_2)\in C^{1,2}(\Omega_T)\cap C(\R^d\times [0,T))$ be a solution to the system \eqref{general rd s wo delta}
such that $u(x,0)\in Q$ for all $x\in\R^d$. If $u$ is the pointwise limit, when $\delta\to0+$, of
$u^\delta$, where $u^\delta\in C^{1,2}(\Omega_T)\cap C(\R^d\times [0,T))$ satisfying \eqref{growth}
is a solution of the Cauchy problem (assuming it has such a solution) 
\begin{equation}
\label{general rd Cauchy}
\left\{\begin{array}{l}
\mathcal{L}_i(u_i^\delta)+c_iu_i^\delta=f_i(x,t,u_1^\delta,u_2^\delta)+\delta,
\quad\quad (x,t)\in\Omega_T\\
u^\delta(x,0)=u(x,0), \quad\quad x\in\R^d
\end{array}\right. 
\end{equation}
then $u(x,t)\in Q$ for all $(x,t)\in\R^d\times [0,T)$.
\end{corollary}

\begin{proof}
By Theorem \ref{invariant} we have $u^\delta(x,t)\in Q$ for all $(x,t)\in\R^d\times [0,T)$, for any $\delta>0$. Since $u(x,t)=\lim_{\delta\to0+}u^\delta(x,t)$ for each
$(x,t)\in\R^d\times[0,T)$ and $Q$ is a closed set in $\R^d$, it follows that $u(x,t)\in Q$ for all $(x,t)\in\R^d\times [0,T)$
as well.
\end{proof}

\begin{remark}
\label{invariant rem}
{\em As we can see by the proofs of Theorem \ref{invariant} and Corollary \ref{invariant cor}, 
we can replace in these results the condition on the reaction functions 
$f_i\ge0$ when $-\varepsilon_0<u_i<0$ and $u_j>-\varepsilon_0$ ($i,j=1,2$, $j\not=i$)
by $f_i>0$ when $u_i=0$ and $u_j\ge0$, if we assume that $u(x,0)\in int.(Q)$ for
all $x\in\R^d$, or, if we assume a continuous dependence also on the initial data, i.e. 
$u(x,0)\in Q$ (for all $x\in\R^d$) and $u$ is the pointwise limit, when $\delta\to0+$, of the solution 
$u^\delta=(u_1^\delta,u_2^\delta)$ in the space $C^{1,2}(\Omega_T)\cap C(\R^d\times [0,T))$ and satisfying \eqref{growth} of the Cauchy problem (assuming it has such a solution)} 
\begin{equation}
\left\{\begin{array}{l}
\mathcal{L}_i(u_i^\delta)+c_iu_i^\delta=f_i(x,t,u_1^\delta,u_2^\delta)+\delta,
\quad\quad (x,t)\in\Omega_T\\
u_i^\delta(x,0)=u_i(x,0)+\delta, \quad\quad x\in\R^d.
\end{array}\right. 
\end{equation}
\end{remark} 

\begin{remark} {\em In Theorem \ref{invariant} and Corollary \ref{invariant cor},
and in Remark \ref{invariant rem} as well, we
notice that to obtain $u(x,t)\in Q$ for all $(x,t)\in\Omega_T$, it suffices 
to show that $u(x,t)\in Q_{\varepsilon,S}$ for all $(x,t)\in\Omega_T$, for
arbitrarily small $\varepsilon>0$ and large $S>0$, where $Q_{\varepsilon,S}=\{(u_1,u_2)\, ;\, -\varepsilon\le u_i\le S\}$.
Then we obtain the same results if we dispense the condition
$f_i\ge0$ or $f_i>0$, when $-\varepsilon_0<u_i<0$ and $u_j>-\varepsilon_0$, ($i,j=1,2$, $j\not=i$), and assume that $f_i>0$ when $u_i=0$ and $f_i$ is continuous at the point $u_i=0$, uniformly with respect to $(x,t)\in\Omega_T$ and 
$-\varepsilon_1\le u_j\le S$, for any $S>0$ and some $\varepsilon_1>0$. 
Indeed, in this case, given $S>0$, there exists some $\varepsilon_0>0$ such
that $f_i>0$ when $-\varepsilon_0<u_i<0$ and $-\varepsilon_0<u_j\le S$}.
\end{remark}

\smallskip

\begin{corollary} {\em (Corollary of Theorem \ref{comparison}.)} \ 
\label{comparison cor}
Let the hypotheses of Theorem \ref{comparison} 
 on the reactions functions $f_i$ be in force and suppose we have a continuous dependence of the solutions of
the system \eqref{general rd s wo ci and delta} with respect to the reaction functions $f_i$; more precisely, suppose $u=(u_1,u_2)\in C^{1,2}(\Omega_T)\cap C(\R^d\times [0,T))$ 
{\em ($0<T\le\infty$, $\Omega_T=\R^d\times (0,T)$)} is a solution of 
\eqref{general rd s wo ci and delta} which is the pointwise limit, in $\Omega_T$,
of $(u^{+\delta})$ and also of $(u^{-\delta})$, 
when $\delta\to0+$, where $u^{+\delta}=(u_1^{+\delta},u_2^{+\delta})$ 
(respect. $u^{-\delta}=(u_1^{-\delta},u_2^{-\delta})$), in the space 
$C^{1,2}(\Omega_T)\cap C(\R^d\times [0,T))$ and satisfying \eqref{growth}, is a solution of the Cauchy problem (assuming such a solution exists) 
\begin{equation}
\label{general rd Cauchy wo ci}
\left\{\begin{array}{l}
\mathcal{L}_i(u_i^{\pm\delta})=f_i(x,t,u_1^{\pm\delta},u_2^{\pm\delta})\pm\delta,
\quad\quad (x,t)\in\Omega_T\\
u^{\pm\delta}(x,0)=u(x,0), \quad\quad x\in\R^d,
\end{array}\right. 
\end{equation}
$\delta>0$.
Then if $\hat{u}=(\hat{u}_1,\hat{u}_2)$ (respect. $\tilde{u}=(\tilde{u}_1,\tilde{u}_2)$), in the space $C^{1,2}(\Omega_T)\cap C(\R^d\times [0,T))$ and satisfying \eqref{growth}, is a lower 
(respect. upper) solution to the system \eqref{general rd s wo ci and delta} 
(i.e. $\mathcal{L}_i(\hat{u}_i)(x,t)\le f_i(x,t,\hat{u}_1(x,t),\hat{u}_2(x,t))$
for all $(x,t)\in\Omega_T$; 
respect. $\mathcal{L}_i(\tilde{u}_i)(x,t)\ge f_i(x,t,\tilde{u}_1(x,t),\tilde{u}_2(x,t))$
for all $(x,t)\in\Omega_T$) such that $\hat{u}_i(x,0)\le u_i(x,0)$ (respect. 
$u_i(x,0)\le \tilde{u}_i(x,0)$) for all $x\in\R^d$, then $\hat{u}_i(x,t)\le u_i(x,t)$ 
(respect. $u_i(x,t)\le \tilde{u}_i(x,t)$) for all $(x,t)\in\R^d\times[0,T)$.  
\end{corollary}

\begin{proof}
Let us consider only the case regarding the lower solution $\hat{u}$, since the case
regarding the upper solution $\tilde{u}$ can be proven analogously. The proof consists
in applying Theorem \ref{comparison} with $f_i+\hat{\delta}$ ($f_i-\hat{\delta}$ if it were the
case regarding $\tilde{u}$) in place of $f_i$, where $\hat{\delta}$ is some number betweem
$\delta'$ and $\delta$, e.g. $(\delta'+\delta)/2$. Notice that $\hat{u}$ and $u^\delta$ are,
respectively, a lower and an upper solution to the system 
$\mathcal{L}_i(u_i)=f_i(x,t,u_1,u_2)+(\delta+\delta')/2$. Besides, $u^\delta-\hat{u}$
satisfies \eqref{growth}, since both $u^\delta$ and $\hat{u}$ do satisfy, and 
$u_i^\delta(x,0)=u_i(x,0)\ge\hat{u}_i(x,0)$ for all $x\in\R^d$. Then, by Theorem \ref{comparison},
we have that $u_i^\delta(x,t)\ge\hat{u}_i(x,t)$ for all $(x,t)\in\R^d\times [0,T)$. Since
this is true for any $\delta\to0+$ and $u(x,t)=\lim_{\delta\to0+}u_i^\delta(x,t)$ for all 
$(x,t)\in\Omega_T$, we obtain the result. \ \ \end{proof} 
 
\begin{remark}
{\em In the statement 1 (respect. statement 2) of Theorem \ref{comparison} we can replace the conditions
\eqref{semi lip} and \eqref{cond 2} by
\begin{equation}
\label{semi lip sector}
\begin{array}{c}
\ f_1(x,t,s+u_1(x,t),u_2(x,t))-f_1(x,t,u_1(x,t),u_2(x,t))\ge c_1(x,t)s,\\
f_2(x,t,u_1(x,t),s+u_2(x,t))-f_2(x,t,u_1(x,t),u_2(x,t))\ge c_2(x,t)s
\end{array}
\end{equation} 
and 
\begin{equation}
\label{cond 2 sector}
\begin{array}{c} 
\ f_1(x,t,u_1(x,t),s+u_2(x,t))-f_1(x,t,u_1(x,t),u_2(x,t))\ge-\delta',\\
f_2(x,t,s+u_1(x,t),u_2(x,t))-f_1(x,t,u_1(x,t),u_2(x,t))\ge-\delta'
\end{array}
\end{equation}
for all $(x,t)\in\Omega_T=\R^d\times (0,T)$, $s\in (-\varepsilon_0,0)$, and all 
$u=(u_1,u_2)\in C^{1,2}(\Omega_T)\cap C(\R^d\times [0,T))$ satisfying \eqref{growth} and such that
$u_i\ge \hat{u}_i$ (respect. $u_i\le \tilde{u}_i$). Cf. \cite[\S 8.2]{pao}.
Indeed, following the proof of Theorem \ref{invariant}, p. \pageref{proof of inv}, 
if there was a point $(x,t)\in\Omega_T$ such that $w(x,t):=(u-\hat{u})(x,t)\not\in Q_\varepsilon$
(respect. $w(x,t):=(\tilde{u}-u)(x,t)\not\in Q_\varepsilon$)
for some arbitrarily small $\varepsilon$, then we would get the contradiction
$(\mathcal{L}_i-c_i)w_i\le 0$ and (see the proof of Theorem \ref{comparison}) 
$(\mathcal{L}_i-c_i)w_i \ge f_i(x,t,u_1,u_2)-f_i(x,t,\hat{u}_1,\hat{u}_2)-c_iw_i+\delta>0$ at some point in 
$(x_0,t_0)\in\Omega_T$. Notice that $u_i\le\tilde{u}_i$ (respect. 
$u_i\le \tilde{u}_i$) for all $(x,t)\in\Omega_{t_0}$.}
\end{remark}

\section{Global solution}
\label{global solution}

In this section we prove Theorem \ref{global}. Let us denote in this section by $u=(u_1,u_2)$ a
maximal solution of \eqref{eq1-general-2}--\eqref{in cond ui}, defined in a maximal interval 
$[0,T^*)$, (see the Introduction, p. 5), in the space 
$X_{T^*}=C^{2,1}( \R\times(0,T^*))\cap C_{\mbox{\tiny{loc}}}^{1,\frac{1}{2}}(\R\times[0,T^*))
\cap L_{\!\!\!\mbox{ \tiny{loc} }}^\infty( (0,\infty) ; L^p(\R) )$, which was also presented in the
Introduction, intercepted with the sector $\langle 0,\varphi\rangle_{T^*}$.
Then we shall show that $T^*=\infty$. Throughout this section we assume all the hypotheses in Theorem \ref{global}, specially $u_{i,0}\in L^p(\mathbb{R})$, for some $p\in (1,\infty)$. 
We suppose that $T^*<\infty$ and we shall obtain a contradiction.
 
Let us recall that the convolution product of functions in conjugate Lebesgue spaces on 
$\R^n$ decay to zero at infinity, more precisely, if $f\in L^p(\R^n)$ and $g\in L^q(\R^n)$,
with $1<p<\infty$ and $1/p+1/q=1$, then $f*g\in C_0(\R^n)$, where $C_0(\R^n)$ denote the space
of continuous functions $h$ on $\R^n$ such that $\lim_{|x|\to\infty}|h(x)|=0$. Besides, 
$\sup_{x\in\R^n}|(f*g)(x)|\leq\|f\|_{L^p}\|g\|_{L^q}$. See \cite[p. 241]{Fol}.\footnote{We would like to thank Prof. Lucas C. F. Ferreira for bringing our attention to this fact and suggesting us to take the initial data 
$u_{i,0}$ in $L^p$.} Using this fact we can prove the following lemma.

\begin{lemma}
\label{lem2.3}
For any $t\in (0,T^*)$ and $s=0,1$, 
we have  $(\partial^s_xu) (.,t)\in C_0(\R)$. Furthermore, there exist the partial derivatives 
$(\partial^3_xu) (x,t)$ and $(\partial_t\partial_x)u (x,t)$, for any $(x,t)\in\R\times(0,T^*)$.
\end{lemma}
\begin{proof}
To prove the first part, for fixed $t \in (0, T^*)$ we use \eqref{integral representation for ui} to write
\begin{align*}
\partial_x u_i(x,t)=& \int\partial_x\Gamma(x,t,\xi,0)u_{i,0}(\xi)d\xi+
\int_0^t\int\partial_x\Gamma(x,t,\xi,\tau)F_i(\xi,\tau)d\xi d\tau\\
\equiv & \ V(x,t)+W(x,t)
\end{align*}
where $F_i(\xi,\tau)=f_i(y_i,u_1,u_2)(\xi,\tau)$.
Noting that 
$|V(x,t)|\le \int\frac{K}{t}\mbox{e}^{-C\frac{(x-\xi)^2}{t}}|u_{i,0}(\xi)|d\xi$ and using that $u_{i,0}\in L^p$ and 
$\mbox{e}^{-C\frac{(.)^2}{t}}\in L^q$, for $1<p<\infty$ and $q$ being the conjugate exponent of $p$,
it follows 
that $V(.,t)$ belongs to $C_0(\R)$. 
Now, for fixed $\tau$, $\tau<t$, the same argument proves also that  
$G(\cdot,t,\tau):=\int\partial_x\Gamma(\cdot,t,\xi,\tau)F_i(\xi,\tau)d\xi\in C_0(\R)$.
Since $|G(x,t,\tau)|\leq \frac{K}{(t-\tau)^\frac{1}{2}}$ and this latter function is integrable on $[0, t]$, it follows from the Lebesgue's dominated convergence theorem that $W(.,t)\in C_0(\R)$ as well.

To prove the second part, for fixed $T\in (0,T^*)$ we observe that  $u_i\in C^{1,\frac{1}{2}}(\R\times[0,T])$ and $w_i=u_i$ satisfies the parabolic equation 
$\mathcal{L}_{[v(u_i)]}(w_i)=F_i(y_i,u_1,u_2)$ in $\R\times(0,T]$.
Given that the coefficients of this equation and $F_i$ are H\"older continuous functions, it follows from Theorem \cite[p. 72]{Fri} that $\partial_xu_i$ is locally H\"older continuous, which implies that the coefficients of this equation has derivative with respect to $x$ locally H\"older
continuous and $(F_i)_x$ is also locally H\"older continuous. Then, again from Theorem 
\cite[p. 72]{Fri} we obtain the existence of the derivatives $\partial_x^3u_i$ and 
$\partial_t\partial_xu_i$. \hfill \end{proof}

\smallskip

Next we show that $\partial_xu_i$ is bounded, in $\R\times(0,T^*)$. It is enough to show 
this bound in $\R\times(T,T^*)$ for a $T\in (0,T^*)$, 
since $u_i\in C^{1,\frac{1}{2}}(\R\times[0,T])$, for any $T\in (0,T^*)$.
We start with the following lemma.

\smallskip

\begin{lemma}
\label{lem2.4} Let $\epsilon\in (0,T^*)$. There is a constant $K$ with the
following property: Let $a<b$ in $\mathbb{R}$ and $T\in(\epsilon,T^*)$. If
the maximum value $q_i$ of $|\partial_xu_i|$ in $[a,b]\times[\epsilon,T]$  
 is attained in a point in $(a,b)\times(\epsilon,T]$, or, 
$q_i$ is attained in a point in $(a,b)\times(\epsilon,T]$ and 
$\partial_xu_j$ is bounded in $\R\times(0,T^*)$, for $i,j=1,2$ and $i\not=j$, 
then 
$q_i\le K$. 
\end{lemma}

\begin{proof} 
Following \cite{Ole} (or \cite{Ole-v1,Ole-v2}; see \cite[p. 107]{Ole}), we define a new function $v_i$ by the equation
\begin{equation}
\label{vi}
u_i=h_i(v_i) :=
K^*(-2+3\mbox{e}\int_0^{v_i}\mbox{e}^{-s^{m_i}}ds),
\end{equation}
where $m_i$ is a sufficiently large constant 
and $K^*:=\varphi(T^*)$. ($K^*$ is a constant that bounds $u_i$. Recall that the maximal solution is in the sector $\langle 0,\varphi\rangle_{T^*}$
and we are assuming that $T^*<\infty$, in order to obtain a contradiction.)
We notice that
, for any positive numbers $m_i$, we have $v_i\ge 2/(3\mbox{e})$ and, $v_i\leq \overline{v}:=(\frac{\varphi(T)}{K^*}+2)/3$ for $t\le T\le T^*$.
Indeed, $-2+3e\int_0^{\frac{2}{3{e}}}{e}^{-s^{m_i}}ds
 \leq -2+3{e}\int_0^{\frac{2}{3{e}}}ds=0$, so it must be $v_i\ge 2/(3{e})$
in order that $u_i\ge 0$. On the other hand, if $t\le T\le T^*$ then 
$u_i\le\varphi(T)$, since $u=(u_1,u_2)$ belongs to the sector $\langle 0,\varphi\rangle_{T^*}$, thus, from the equation \eqref{vi} we have \ 
$e\int_0^{v_i}{e}^{-s^{m_i}}ds\le \overline{v}$. But
${e}\int_0^{\overline{v}}{e}^{-s^{m_i}}ds\ge {e}\int_0^{\overline{v}}ds=\overline{v}$, 
then $v_i\le\overline{v}$.

Now making the substitution $u_i=h_i(v_i)$ in the equation \eqref{eq1-general} (with the
constitutive functions \eqref{functions}), we have that $v_i$ satisfies the following equation:

$(v_i)_t-\left(\frac{\lambda_i}{a_i+b_iy_i}\right)(v_i)_{xx}+\left(\frac{c_i}{a_i+b_iy_i}\right)(v_i)_{x}-
(\frac{\lambda_i}{a_i+b_iy_i})\frac{h_i''}{h_i'}(v_i)^2_{x}=
\frac{f_i(y_i,h_1(v_1),h_2(v_2))}{h_i'(v_i)}$.

\noindent
Differentiating with respect to $x$, we have
\begin{equation}
\label{eq2.12}
\begin{array}{rl}
&(v_i)_{tx}+\frac{\lambda_ib_i{(y_i)}_x}{(a_i+b_iy_i)^2}(v_i)_{xx}-\frac{\lambda_i}{a_i+b_iy_i}(v_i)_{xxx}-
\frac{c_ib_i{(y_i)}_x}{(a_i+b_iy_i)^2}(v_i)_{x}+\frac{c_i}{a_i+b_iy_i}(v_i)_{xx}\\
&+\frac{\lambda_ib_i{(y_i)}_x}{(a_i+b_iy_i)^2}\frac{h_i''}{h_i'}(v_i)^2_x-
\frac{\lambda_i}{a_i+b_iy_i}\left(\frac{h_i''}{h_i'}\right)'(v_i)^3_x-
\frac{2\lambda_i}{a_i+b_iy_i}\frac{h_i''}{h_i'}(v_i)_{x}(v_i)_{xx}\\
&=(\frac{f_i(y_i,h_1(v_1),h_2(v_2))}{h_i'(v_i)})_x.
\end{array}
\end{equation}

Let $a,b,\epsilon, T$ as in the statement of the lemma. 
If the maximum value of $|(u_i)_x|$ in $[a,b]\times [\epsilon,T]$ is attained
in $(a,b)\times(\epsilon,T]$, then, the same holds for $\frac{1}{3eK^*}|(u_i)_x|$
and, as a consequence, this is also true for $|(v_i)_x|$, if $m_i$ is chosen
sufficiently large. Indeed,  $|(u_i)_x(x,t)|=3eK^*\mbox{e}^{-v_i(x,t)^{m_i}}|(v_i)_x(x,t)|$, so, for any
$(x,t)\in[a,b]\times[\epsilon,T]$, we have
\begin{align*}
|\frac{1}{3eK^*}|(u_i)_x(x,t)|-|(v_i)_x(x,t)||&\leq |\frac{1}{3eK^*}|(u_i)_x(x,t)|-
\frac{{e}^{v_i(x,t)^{m_i}}}{3eK^*}|(u_i)_x(x,t)||\\
&\leq|1-{e}^{v_i(x,t)^{m_i}}|\frac{1}{3eK^*}|(u_i)_x(x,t)|\\
&\leq|1-{e}^{\overline{v}^{m_i}}|\frac{1}{3eK^*}q_i,\\
\end{align*}
where, as in the statement of the lemma, $q_i$ is the maximum of $|(u_i)_x|$ in $[a,b]\times[\epsilon,T]$. Then, $|(v_i)_x|$
converges to $\frac{1}{3eK^*}|(u_i)_x(x,t)|$, uniformly in $(x,t)$, when $m_i$ tends to
infinity. The uniform convergence together with the fact that the maximum of $|(u_i)_x|$ occurs only
in $(a,b)\times(\epsilon,T]$ ensures that the maximum of $|(v_i)_x|$ also occurs only in
$(a,b)\times(\epsilon,T]$, if we take $m_i$ sufficiently large.


Let $p_i$ be the maximum value of $|(v_i)_x|$ in $[a,b] \times [\epsilon,T]$. By the hypothesis,
in the first alternative in the statement of the lemma, there is a point $(x_i,t_i)$ in 
$(a,b)\times(\epsilon,T]$ such that 
$(v_i)_x(x_i,t_i)=\pm p_i$. Since we want to estimate $|\partial_x v_i|$, we can assume $p_i>0$,
without loss of generality. 

Initially we consider the case where $(v_i)_x(x_i,t_i)=p_i$. 
Then, $(v_i)_{xt}(x_i,t_i)\geq 0$, $(v_i)_{xx}(x_i,t_i)=0$ and $(v_i)_{xxx}(x_i,t_i)\leq 0$,
so, computing \eqref{eq2.12} in $(x_i,t_i)$, we obtain
\begin{equation}
\label{eq2.13}
-\frac{c_ib_i{(y_i)}_x}{(a_i+b_iy_i)^2}p_i+\frac{\lambda_ib_i{(y_i)}_x}{(a_i+b_iy_i)^2}
\frac{h_i''}{h_i'}p_i^2-\frac{\lambda_i}{a_i+b_iy_i}\left(\frac{h_i''}{h_i'}\right)'p_i^3\leq
(\frac{f_i(y_i,h_1(v_1),h_2(v_2))}{h_i'(v_i)})_x.
\end{equation}
From the definition of $h_i$, equation \eqref{vi}, we have
\begin{equation}
\label{eq2.14}
\begin{array}{l}
h_i'=3eK^* \mbox{e}^{-v^{m_i}},\quad
h_i''=-3eK^* m_iv^{m_i-1}\mbox{e}^{-v^{m_i}},\\
\frac{h_i''}{h_i'}=-m_iv^{m_i-1},\quad
{\left(\frac{h_i''}{h_i'}\right)}'=-m_i(m_i-1)v^{m_i-2},
\end{array}\end{equation}
and from \eqref{eq2.6},
\begin{equation}
\label{eq2.15}
\begin{array}{rl}
(y_i)_x=&y_{i0}'(x)\mbox{e}^{-A_i\int_0^tf(h_i(v_i))ds}\\
&+y_{i0}(x)\mbox{e}^{-A_i\int_0^tf(h_i(v_i))ds}
(-A_i\int_0^tf'(h_i(v_i))h_i'(v_i)(v_i)_x(x,s)ds),
\end{array}
\end{equation}
So, at the point $(x_i,t_i)$, we obtain the estimate
\begin{equation}
\label{eq2.16}
|(y_i)_x|\leq K(1+3T^*eK^* p_i)\leq K_1(1+p_i),
\end{equation}
where $K_1$ is a constant independent of the interval $[a,b]$ and of $T$. 
%
On the other hand,
$$
\begin{array}{rl}
&f_i(y_i, h_1(v_1),h_2(v_2))_x\\
=&-[(b_iA_ih_i(v_i)+d_i)y_if(h_i(v_i))+
(-1)^iq(h_1(v_1)-h_2(v_2))](a_i+b_iy_i)^{-2}b_i(y_i)_x\\
&+\left\{b_iA_ih_i'(v_i)(v_i)_xy_if(h_i(v_i))\right.\\
& \ \ +\left. (b_iA_ih_i(v_i)+d_i)[(y_i)_xf(v_i)+y_if'(h_i(v_i))h_i'(v_i)(v_i)_x]\right.\\
& \ \ \ \ \ \left. + (-1)^iq(h_1'(v_1)(v_1)_x-h_2'(v_2)(v_2)_x)\right\}(a_i+b_iy_i)^{-1},
\end{array}
$$
therefore, at $(x,t)=(x_i,t_i)$, 
$|f_i(y_i,h_1(v_1),h_2(v_2))_x|\leq K(1+p_i+p_j)$, $i\not=j$,
with $K$ being a constant independent of $a,b$ and $T$.
As $(f_i/h_i'(v_i))_x=[(f_i)_xh_i'(v_i)-f_ih_i''(v_i)(v_i)_x](h_i'(v_i))^{-2}$, it follows that
\begin{equation}
\label{eq2.17}
|(\frac{f_i}{h_i'(v_i)})_x|\leq
\frac{|(f_i)_x||h_i'(v_i)|-|f_i||h_i''(v_i)||(v_i)_x|}{(h_i'(v_i))^2}\leq K(1+p_i+p_j),
\end{equation}
$i\not= j$. Substituting \eqref{eq2.14}, \eqref{eq2.16} and \eqref{eq2.17} in \eqref{eq2.13}, we arrive at
\begin{align*}
&-\frac{c_ib_iK_1(1+p_i)}{(a_i+b_iy_i)^2}p_i-
\frac{\lambda_ib_iK_1(1+ p_i)m_iv_i^{m_i-1}}{(a_i+b_iy_i)^2}p_i^2+
\frac{\lambda_im_i(m_i-1)v_i^{m_i-2}}{a_i+b_iy_i}p_i^3\\
\leq& K(1+p_i+p_j)
\end{align*}
i.e.
\begin{align*}
&-\frac{c_ib_iK_1(1+p_i)}{(a_i+b_iy_i)^2}p_i-
\frac{\lambda_ib_iK_1m_iv_i^{m_i-1}}{(a_i+b_iy_i)^2}p_i^2+((m_i-1)-
\frac{b_iv_iK_1}{(a_i+b_iy_i)} )\frac{\lambda_im_iv_i^{m_i-2}}{(a_i+b_iy_i)}p_i^3\\
\leq& K(1+p_i+p_j).
\end{align*}
As $0\leq y_i\leq \|y_{i,0}\|_\infty$ and $2/(3e)\leq v_i\leq 1$, we obtain
\begin{align*}
&-\frac{c_ib_iK_1(1+p_i)}{a_i^2}p_i-\frac{\lambda_ib_iK_1m_i}{a_i^2}p_i^2+((m_i-1)-\frac{b_iK_1}{a_i}
)\frac{\lambda_im_iv_i^{m_i-2}}{(a_i+b_iy_i)}p_i^3\\
\leq& K(1+p_i+p_j).
\end{align*}
Recall that $K_1$ in the inequality \eqref{eq2.16} does not depend on $m_i$. Thus, we can take $m_i$
large enough such that $p_1$ and $p_2$ satisfy
\begin{equation}
\label{eq2.18}
\left\{\begin{array}{l}
c_1 p_1^3-d_1 p_1^2-e_1 p_1-1\leq p_2\\
c_2p_2^3-d_2p_2^2-e_2p_2-1\leq p_1
\end{array}\right.
\end{equation}
where $c_i,d_i,e_i$ are positive constants and independent of $a,b$ and $T$. 
In the first quadrant, i.e. $p_1\geq0$ and
$p_2\geq 0$, the region defined by \eqref{eq2.18} is bounded. Therefore, there is a constant
$\overline{K}(T^*)$, such that  $0<p_i\leq \overline{K}(T^*)$.

Similarly, we can show that the same occurs when $\partial_xv_1(x_1,t_1)=p_1$ and\break\hfill 
$\partial_xv_2(x_2,t_2)=-p_2$ or $\partial_xv_1(x_1,t_1)=-p_1$  and $\partial_xv_2(x_2,t_2)=-p_2$.
Indeed, suppose now that $\partial_xv_1(x_1,t_1)=p_1$  and $\partial_xv_2(x_2,t_2)=-p_2$.
In this case we have $(v_1)_{xt}(x_1,t_1)\geq 0$, $(v_1)_{xx}(x_1,t_1)=0$ and $(v_1)_{xxx}(x_1,t_1)\leq0$.
While at the point $(x_2,t_2)$, we have $(v_2)_{xt}(x_2,t_2)\leq 0$, $(v_2)_{xx}(x_2,t_2)=0$ and
$(v_2)_{xxx}(x_2,t_2)\geq0$.
If we compute \eqref{eq2.12} at $(x_i,t_i)$, we obtain
\begin{equation}
\left\{\begin{array}{l}
-\frac{c_1b_1{(y_1)}_x}{(a_1+b_1y_1)^2}p_1
+\frac{\lambda_1b_1{(y_1)}_x}{(a_1+b_1y_1)^2}\frac{h_1''}{h_1'}p_1^2-
\frac{\lambda_1}{a_1+b_1y_1}(\frac{h_1''}{h_1'})'p_1^3\leq (\frac{f_1(y_1,h_1(v_1),h_2(v_2))}{h_1'(v_1)})_x\\
\frac{c_2b_2{(y_2)}_x}{(a_2+b_2y_2)^2}p_2
+\frac{\lambda_2b_2{(y_2)}_x}{(a_2+b_2y_2)^2}\frac{h_2''}{h_2'}p_2^2+
\frac{\lambda_2}{a_2+b_2y_2}(\frac{h_2''}{h_2'})'p_2^3\geq 
(\frac{f_2(y_2,h_1(v_1),h_2(v_2))}{h_2'(v_2)})_x.\end{array}\right.
\end{equation}
\begin{equation}
\left\{\begin{array}{l}
-\frac{c_1b_1K_1(1+p_1)}{(a_1+b_1y_1)^2}p_1-\frac{\lambda_1b_1K_1(1+p_1)m_1v_1^{m_1-1}}{(a_1+b_1y_1)^2}p_1^2+
\frac{\lambda_1m_1(m_1-1)v_1^{m_1-2}}{a_1+b_1y_1}p_1^3\leq K(1+p_1+p_2)\\
\frac{c_2b_2K_1(1+p_2)}{(a_2+b_2y_2)^2}p_2+\frac{\lambda_2b_2K_1(1+p_2)m_2v_2^{m_2-1}}{(a_2+b_2y_2)^2}p_2^2-
\frac{\lambda_2m_2(m_2-1)v_2^{m_2-2}}{a_2+b_2y_2}p_2^3\geq-K(1+p_1+ p_2)
\end{array}\right.
\end{equation}
Again, for $m_i$ large enough, we have that $p_1$ and  $p_2$ satisfy 
a system of the type \eqref{eq2.18}, and therefore $(p_1,p_2)$ is in a bounded region of the plane.

Finally, suppose that $(v_1)_x(x_1,t_1)=-p_1$ and $(v_2)_x(x_2,t_2)=-p_2$. In this case, at the point $(x_i,t_i)$ we have that $(v_i)_{xt}(x_i,t_i)\leq 0$, $(v_i)_{xx}(x_i,t_i)=0$ and
$(v_i)_{xxx}(x_i,t_i)\geq0$. 
If we compute \eqref{eq2.12} at the point $(x_i,t_i)$, we obtain
\begin{equation}
\frac{c_ib_i{(y_i)}_x}{(a_i+b_iy_i)^2}p_i+\frac{\lambda_ib_i{(y_i)}_x}{(a_i+b_iy_i)^2}
\frac{h_i''}{h_i'}p_i^2+\frac{\lambda_i}{a_i+b_iy_i}\left(\frac{h_i''}{h_i'}\right)'p_i^3\geq
(\frac{f_i(y_i, h_1(v_1),h_2(v_2))}{h_i'(v_i)})_x
\end{equation}

$
\frac{c_ib_iK_1(1+p_i)}{(a_i+b_iy_i)^2}p_2+\frac{\lambda_ib_iK_1(1+p_i)m_iv_i^{m_i-1}p_i^2}{(a_i+b_iy_i)^2}-
\frac{\lambda_im_i(m_i-1)v_i^{m_i-2}p_i^3}{a_i+b_iy_i}\geq-K(1+p_1+p_1),
$\\

\noindent
and, again, for $m_i$ large enough, 
analogously as in the previous cases, we infer that $(p_1, p_2)$ is bounded, independently of
$a,b$ and $T$.

To prove the lemma under the second hypothesis alternative, it is enough to notice that if 
$\partial_xv_2$ is bounded in $\R\times(0,T^*)$, the first inequality in
\eqref{eq2.18} is sufficient to assure that $\partial_xv_1$ is bounded, and similarly, for
the case that $\partial_xv_1$ is bounded. 
%
\  \ \end{proof}

\smallskip

In the next lemma, using that $u\in L_{\!\!\!\mbox{ \tiny{loc} }}^\infty( (0,\infty) ; L^p(\R))$, we show that $u(x,t)$ decay to zero when $|x|\to\infty$, uniformly
with respect to $t$ in any compact interval in $(0,T^*)$.

\begin{lemma}
\label{lem2.7}
Let $[\overline{t},\overline{\overline{t}}]\subset(0,T^*)$ \ 
($0<\overline{t}<\overline{\overline{t}}<T^*$). 
Then $\lim_{|x|\to\infty}|u(x,t)|=0$, uniformly with respect to 
$t\in [\overline{t},\overline{\overline{t}}]$. 
\end{lemma}
\begin{proof} Let $t\in[\overline{t},\overline{\overline{t}}]$. For $0<\epsilon<t$, from \eqref{integral representation for ui} we have that
\begin{align*}
\partial_x&u_i(x,t)=\\
\nonumber
&=\int\partial_x\Gamma_{[v_i(u_i)]}(x,t,\xi,0)u_{i,0}(\xi)d\xi+\int_0^t\int\partial_x\Gamma_{[v_i(u_i)]}(x,t,\xi,\tau)f_i(y_i(u_i),u_1,u_2)(\xi,\tau)d\xi d\tau\\
\nonumber
&=\int\partial_x\Gamma_{[v_i(u_i)]}(x,t,\xi,0)u_{i,0}(\xi)d\xi+\int_0^{t-\epsilon}\int\partial_x\Gamma_{[v_i(u_i)]}(x,t,\xi,\tau)f_i(y_i(u_i),u_1,u_2)(\xi,\tau)d\xi d\tau\\
\nonumber
&+\int_{t-\epsilon}^t\int\partial_x\Gamma_{[v_i(u_i)]}(x,t,\xi,\tau)f_i(y_i(u_i),u_1,u_2)(\xi,\tau)d\xi d\tau
\end{align*}
therefore, using the estimate \eqref{eq0.10}, we have
\begin{align*}
    &|\partial_xu_i(x,t)|\\
\leq&\int\frac{K}{t}\mbox{e}^{-C{\frac{(x-\xi)^2}{t}}}|u_{i,0}(\xi)|d\xi+\int_0^{t-\epsilon}\int\frac{K}{t-\tau}\mbox{e}^{-C{\frac{(x-\xi)^2}{t-\tau}}}|f_i(y_i(u_i),u_1,u_2)(\xi,\tau)|d\xi d\tau\\
\nonumber
&+\int_{t-\epsilon}^t\int\frac{K}{t-\tau}\mbox{e}^{-C{\frac{(x-\xi)^2}{t-\tau}}}|f_i(y_i(u_i),u_1,u_2)(\xi,\tau)| d\xi d\tau\\
\nonumber
\leq&\int\frac{K}{\overline{t}}\mbox{e}^{-C{\frac{(x-\xi)^2}{\overline{\overline{t}}}}}|u_{i,0}(\xi)|d\xi+\int_0^{t-\epsilon}\int\frac{K}{\epsilon}\mbox{e}^{-C{\frac{(x-\xi)^2}{\overline{\overline{t}}}}}|f_i(y_i(u_i),u_1,u_2)(\xi,\tau)|d\xi d\tau\\
\nonumber
&+\int_{t-\epsilon}^t\int\frac{K}{t-\tau}\mbox{e}^{-C{\frac{(x-\xi)^2}{t-\tau}}}|f_i(y_i(u_i),u_1,u_2)(\xi,\tau)|d\xi d\tau.\\
\end{align*}
Since $u=(u_1,u_2)\in \langle 0,\varphi\rangle_{T^*}$, $(y_i,u_1,u_2)$ belongs to a bounded region in $\R^3$. If $\|f_i\|_\infty$ is the sup of $|f_i|$ in this region, we have that
\begin{align*}
    &|\partial_xu_i(x,t)|\\
\leq&\int\frac{K}{\overline{t}}\mbox{e}^{-C{\frac{(x-\xi)^2}{\overline{\overline{t}}}}}|u_{i,0}(\xi)|d\xi+\int_0^{\overline{\overline{t}}}\int\frac{K}{\epsilon}\mbox{e}^{-C{\frac{(x-\xi)^2}{\overline{\overline{t}}}}}|f_i(y_i(u_i),u_1,u_2)(\xi,\tau)|d\xi d\tau\\
\nonumber
&+K\|f_i\|_\infty\int_{t-\epsilon}^t\frac{1}{(t-\tau)^\frac{1}{2}} d\tau\\
\nonumber
\leq&\int\frac{K}{\overline{t}}\mbox{e}^{-C{\frac{(x-\xi)^2}{\overline{\overline{t}}}}}|u_{i,0}(\xi)|d\xi+\int_0^{\overline{\overline{t}}}\int\frac{K}{\epsilon}\mbox{e}^{-C{\frac{(x-\xi)^2}{\overline{\overline{t}}}}}|f_i(y_i(u_i),u_1,u_2)(\xi,\tau)|d\xi d\tau +K\epsilon^\frac{1}{2}
\end{align*}
From \cite[p. 241]{Fol}, the last two integrals of the above inequalities belong to $C_0$
and do not depend on $t\in[\overline{t},\overline{\overline{t}}]$. Thus, taking an interval $[a,b]$ for which these
integrals are less than $\epsilon^\frac{1}{2}$ in $[a,b]^c$, we have $|\partial_xu_i(x,t)|\leq K\epsilon^\frac{1}{2}$
for all $(x,t)\in [a,b]^c\times[\overline{t},\overline{\overline{t}}]$. The constant $K$ is essentially the same that
appears in the estimate \eqref{eq0.10}.
\  \ \end{proof}

\smallskip

%
%
\begin{corollary}
\label{uix is bounded}
$\partial_xu_i$ is bounded in $\R\times(0,T^*)$.
\end{corollary}
\begin{proof}
Let $T_1\in(0,T^*)$ fixed.  We take a constant $\overline{K}>0$ strictly greater than
the constant obtained in Lemma \ref{lem2.4} and $sup_{\R\times[0,T_1]}|\partial_xu_i|$.
Initially, we assume that both $(u_1)_x$ and $(u_2)_x$ are unbounded in $\R\times(0,T^*)$.
Let $T\in(0,T^*)$ such that $|(u_1)_x|>\overline{K}$ at some point of $\R\times(0,T]$.
Now consider $\overline{t}=\inf\{t; |\partial_x u_1(x,t)|>\overline{K}$ for some  $x\in\R\}$ and
let $(x_n,t_n)_{n\in\mathbb{N}}$ be a sequence such that $|\partial_xu_1(x_n,t_n)|>\overline{K}$,
for all $n\in\mathbb{N}$, and $t_n\searrow \overline{t}$. It is clear that $\overline{t}>T_1$. 
Let $\epsilon\in (0,T^*-T_1)$. From Lemma \ref{lem2.7} there is an interval $[a,b]$ such that $|\partial_xu_1(x,t)|\leq \overline{K}$ for all $(x,t)\in [a,b]^c\times[\overline{t},\overline{t}+\epsilon]$, so, $x_n\in[a,b]$ for all $n$ sufficiently large.
Therefore, there are $\overline{x}\in[a,b]$ and a subsequence of $(x_n)$ such that $x_n\rightarrow \overline{x}$. Then, by the continuity of $|\partial_xu_1|$ in $\R\times(0,T]$,
$|\partial_xu_1(\overline{x},\overline{t})|= \overline{K}$ is the maximum of $|\partial_xu_1|$ in
$\R\times(0,\overline{t}]$.
Similarly, assuming that $\partial_xu_2$ is unbounded, we obtain a point
$(\overline{\overline{x}},\overline{\overline{t}})$ such that
$|\partial_xu_2(\overline{\overline{x}},\overline{\overline{t}})|=\overline{K}$ is the maximum
of $|\partial_xu_2|$ in $\R\times(0,\overline{\overline{t}}]$. It is clear that
$\overline{\overline{t}}> T_1$ and, without loss of generality, we can assume
$\overline{t}\leq\overline{\overline{t}}$ (otherwise, we exchange $u_1$ by $u_2$).
If $\overline{t}=\overline{\overline{t}}$ then taking an interval $(A,B)$ containing the points
$\overline{x}$  and $\overline{\overline{x}}$, we obtain a contradiction to Lemma \ref{lem2.4},
because both  maximum points of $|\partial_xu_1|$ and $|\partial_xu_2|$  occur in
$(A,B)\times(\frac{\overline{t}}{2},\overline{t}]$ and both are bounded by the constant given by
Lemma \ref{lem2.4}.
If $\overline{t}<\overline{\overline{t}}$ Lemma \ref{lem2.7} assures the existence of an interval
$(A,B)$ for which $|\partial_xu_1|<\overline{K}$ in
$[A,B]^c\times[\overline{t},\overline{\overline{t}}]$. Thus, the maximum points of $|\partial_xu_1|$
and $|\partial_xu_2|$ in $[A,B]\times[\frac{\overline{t}}{2},\overline{\overline{t}}]$ both
occur in $(A,B)\times(\frac{\overline{t}}{2},\overline{\overline{t}}]$ and they are not bounded by
the constant given by Lemma \ref{lem2.4}. 

Let us now assume that $\partial_xu_1$ is unbounded and $\partial_xu_2$ is bounded, in 
$\R\times(0,T^*)$. Repeating the initial argument for $\partial_xu_1$, we get that the maximum point
$(\overline{x},\overline{t})$ of $|\partial_xu_1|$ in $\R\times(0,\overline{t}]$ occurs
in $(A,B)\times(\frac{\overline{t}}{2},\overline{t}]$ and the maximum
$|\partial_xu_1(\overline{x},\overline{t})|=\overline{K}$ is greater than the constant
given by Lemma \ref{lem2.4}. Adding to this the fact that $\partial_xu_2$ is bounded in $\R\times(0,T^*)$, we again arrive at a contradiction with Lemma \ref{lem2.4}.
\  \ \end{proof}

\medskip

\begin{lemma}
\label{lem2.6}
The function $y_i=y_{i,0}(x)\mbox{e}^{-A_i \int_0^tf(u_i)ds}$ and the coefficients 
$\alpha_i(y_i), \beta_i(y_i)$ belong to $C^{1,\frac{1}{2}}(\R\times[0,T^*))$.
\end{lemma}
\begin{proof} It is clear that $|y_i(x,t)|\leq \|y_{i,0}\|_\infty$, for all $\R\times(0,T^*)$.
Besides,
$$
\partial_xy_i(x,t)=(y_{i,0}'(x)-A_iy_{i,0}(x)\int_0^tf'(u_i(x,s))
\partial_xu_i(x,s)ds)\mbox{e}^{-A_i\int_0^tf(u_i(x,s))ds}
$$ 
so, since  $y_{i,0}'$ is bounded, by hypothesis, and we have Corollary \ref{uix is bounded}, it follows that $(y_i)_x$ is bounded in $\R\times(0,T^*)$.
Moreover,
\begin{align*}
|y_i(x,t)-y_i(x,t')|
&\leq y_{i,0}(x)|\mbox{e}^{-A_i\int_0^tf(u(x,s))ds}-\mbox{e}^{-A_i\int_0^{t'}f(u_i(x,s))ds}|\\ 
&\leq K|\int_0^tf(u_i(x,s))ds-\int_0^{t'}f(u_i(x,s))ds|
\leq K|\int_{t'}^tf(u_i(x,s))ds|\\ 
&\leq K(t-t')\leq K(t-t')^\frac{1}{2}
\end{align*}
for all $(x,t), (x,t')\in \R\times(0,T^*)$, with $|t-t'|\le 1$, for some constant $K$. 
Finally, as the composition of a H\"older continuous function with
a function having a bounded derivative is also a H\"older continuous function, the result follows 
by using \eqref{functions}. \ \ \end{proof}

\subsection{Proof of Theorem \ref{global}}
\label{global proof}

Due to Lemma \eqref{lem2.6}, we can consider the parabolic equation 
$\partial_t-\alpha_i(y_i)\partial_{xx}+\beta_i(y_i)\partial_x=0$ in the domain $\R\times[0,T^*]$.
Let us denote its fundamental solution by $\Gamma$. By Theorem \ref{teo0.1} we can write
\begin{equation}
\label{eq2.25}
 u_i(x,t)=\int_\R\Gamma(x,t,\xi,0)u_{i,0}(\xi)d\xi+\int_0^t\int_\R
\Gamma(x,t,\xi,\tau)f_i(y_i,u_1,u_2)(\xi,\tau)d\xi d\tau 
\end{equation}
for all $(x,t)\in\R\times[0,T^*)$. Now, since $u_{i,0}$ is bounded, by hypothesis,
$u_i$ is bounded (recall that $u=(u_1,u_2)$ is in the sector $\langle 0,\varphi\rangle_{T^*}$), and
we have the estimate $\Gamma\le K(t-\tau)^{-1/2}e^{-C\frac{(x-\xi)^2}{t-\tau}}$ (see \eqref{eq0.10}),
for $x,\xi\in\R$ and $t,\tau\in [0,T^*]$, $\tau<t$, it follows that the functions 
$\Gamma(x,t,\xi,0)u_{i,0}(\xi), \Gamma(x,t,\xi,\tau)f_i(y_i,u_1,u_2)(\xi,\tau)$ are integrable with respect to $\xi$ in $\R$ when $t=T^*$, and $\int_\R\Gamma(x,T^*,\xi,\tau)f_i(y_i,u_1,u_2)(\xi,\tau)d\xi$ is integrable with respect to $\tau$ in $[0,T^*]$. Thus the right hand side of \eqref{eq2.25} is well defined for $t=T^*$ and we set $u_i(x,T^*)$ as being this value.   

Next, with this definition, we show that $u_i(x, t)$ converges to $u_i(x, T^*)$ when 
$t\nearrow T^*$ uniformly with respect to $x\in\R$. In fact, we have 
$\|u_i(\cdot,T^*)-u_i(\cdot,t)\|_{L^\infty(\R)}\le K(T^*-t)^{1/2}$, for some constant $K$.
Indeed,
\begin{align*}
 &u_i(x,T^*)-u_i(x,t)\\
=&\int_\R(\Gamma(x,\xi,T^*,0)-\Gamma(x,t,\xi,0))u_{i,0}(\xi)d\xi\\
&+\int_t^{T^*}\int_\R\Gamma(x,\xi,T^*,\tau)f_i(y_i,u_1,u_2)(\xi,\tau)d\xi d\tau\\
&+\int_0^{t}\int_\R(\Gamma(x,\xi,T^*,\tau)-
\Gamma(x,t,\xi,\tau))f_i(y_i,u_1,u_2)(\xi,\tau)d\xi d\tau\\
\equiv& I_1+I_2+I_3.
\end{align*}
If $\frac{T^*}{2}\leq t\leq T^*$, we have
\begin{align*}
|I_1|&=|\int(\Gamma_{[v_i(u_i)]}(x,\xi,T^*,0)-\Gamma_{[v_i(u_i)]}(x,t,\xi,0))u_{i,0}(\xi)d\xi|\\
&=|\int\partial_t\Gamma_{[v_i(u_i)]}(x,\xi,s,0)(T^*-t)u_{i,0}(\xi)d\xi|\\
&\leq\int \frac{K}{s^{\frac{3}{2}}}\mbox{e}^{-C\frac{(x-\xi)^2}{s}}|T^*-t|\|u_{i,0}\|_\infty d\xi=
\frac{K}{s}|T^*-t|\leq \frac{2K}{T^*}|T^*-t|
\end{align*}
\begin{align*}
|I_2|&= |\int_t^{T^*}\int\Gamma_{[v_i(u_i)]}(x,\xi,T^*,\tau)f_i(y_i(u_i),u_1,u_2)(\xi,\tau)d\xi d\tau|\\
&\leq\int_t^{T^*}\int\frac{K}{(T^*-\tau)^\frac{1}{2}}\mbox{e}^{-C\frac{(x-\xi)^2}{T^*-t}}\|f_i\|_\infty d\xi d\tau\\
&\leq\int_t^{T^*}K\|f_i\|_\infty d\tau=K\|f_i\|_\infty(T^*-t)
\end{align*}
\begin{align*}
|I_3|&=|\int_0^{t}\int\left(\Gamma_{[v_i(u_i)]}(x,\xi,T^*,\tau)-
\Gamma_{[v_i(u_i)]}(x,t,\xi,\tau)\right)f_i(y_i(u_i),u_1,u_2)(\xi,\tau)d\xi d\tau|\\
&=|\int_0^{t}\int\int_t^{T^*}\partial_t
\Gamma_{[v_i(u_i)]}(x,\xi,s,\tau)f_i(y_i(u_i),u_1,u_2)(\xi,\tau)dsd\xi d\tau|\\
&=|\int_0^{t}\int_t^{T^*}\int\partial_t
\Gamma_{[v_i(u_i)]}(x,\xi,s,\tau)f_i(y_i(u_i),u_1,u_2)(\xi,\tau)d\xi ds d\tau|\\
&=|\int_0^{t}\int_t^{T^*}\int\partial_t
\Gamma_{[v_i(u_i)]}(x,\xi,s,\tau)(f_i(y_i(u_i),u_1,u_2)(\xi,\tau)-f_i(y_i(u_i),u_1,u_2)(x,\tau))d\xi dsd\tau|\\
&\leq\int_0^{t}\int_t^{T^*}\int\frac{K}{(s-\tau)^\frac{3}{2}}\mbox{e}^{-C\frac{(x-\xi)^2}{s-\tau}}\|\nabla f_i\|_\infty(\|(u_1)_x\|_\infty+\|(u_2)_x\|_\infty+\|(y_i)_x\|_\infty)|x-\xi|d\xi dsd\tau\\
&=\int_0^{t}\int_t^{T^*}\int\frac{K}{s-\tau}\mbox{e}^{-C^*\frac{(x-\xi)^2}{s-\tau}}d\xi dsd\tau\leq
K\int_0^{t}\int_t^{T^*}\frac{1}{(s-\tau)^\frac{1}{2}}ds d\tau
\end{align*}
\begin{align*}
&=K\int_0^t[(T^*-\tau)^\frac{1}{2}-(t-\tau)^\frac{1}{2}]d\tau\leq K\int_0^t(T^*-t)^\frac{1}{2}d\tau\\
&\leq KT^*(T^*-t)^\frac{1}{2}.
\end{align*}

From the above convergence, we conclude that $u_i(\cdot,T^*)$ is bounded and nonnegative (since
$u=(u_1,u_2)\in\langle 0,\varphi\rangle_{T^*}$) and it is Lipschitz continuous as well, by Corollary 
\ref{uix is bounded}. To end the proof of Theorem \ref{global}, it remains to prove
that $u_i(\cdot,T^*)\in L^p$. Using again the {\lq\lq}generalized Young's inequality{\rq\rq} \cite[p. 9]{Fol1} and the Minkowski's ineguality for integrals \cite[p. 194]{Fol}) (see the proof of the $L^p$ assertion (the last assertion) in Theorem \ref{local}), we obtain
\begin{align*}
 &\|u_i(.,t)\|_{L^p}\\
\nonumber
 \leq&\|\int\Gamma_{[v_i(u_i)]}(.,t,\xi,0)u_{i,0}(\xi)d\xi\|_{L^p}\\
&+\|\int_0^t\int\Gamma_{[v_i(u_i)]}(x,t,\xi,\tau)f_i(y_i(u_i),u_1,u_2)(\xi,\tau)d\xi d\tau\|_{L^p}\\
\nonumber
\leq& C_1+\int_0^t\|\int\Gamma_{[v_i(u_i)]}(x,t,\xi,\tau)f_i(y_i(u_i),u_1,u_2)(\xi,\tau)d\xi\|_{L^p}d\tau\\
\nonumber
\leq& C_1+\int_0^t\|\int\Gamma_{[v_i(u_i)]}(x,t,\xi,\tau)(f_i(y_i(u_i),u_1,u_2)-
f_i(y_i(u_i),0,0))(\xi,\tau)d\xi\|_{L^p}d\tau \\
\nonumber
\leq& C_1+\int_0^t\|\int\Gamma_{[v_i(u_i)]}(x,t,\xi,\tau)
\nabla_u F_i(u_1(\xi,\tau),u_2(\xi,\tau))d\xi\|_{L^p}d\tau\\
\nonumber
\leq& C_1+C_2\int_0^t(\|u_1(.,\tau)\|_{L^p}+\|u_2(.,\tau)\|_{L^p})d\tau\, .
\end{align*}
Thus, $\|u_1(.,t)\|_{L^p}+\|u_2(.,t)\|_{L^p}\leq C_1+
C_2\int_0^t(\|u_1(.,\tau)\|_{L^p}+\|u_2(.,\tau)\|_{L^p})d\tau$.
Then, defining $\phi(t)=\|u_1(.,t)\|_{L^p}+\|u_2(.,t)\|_{L^p}$, we have
$\phi(t)\leq C_1+C_2\int_0^t\phi(\tau)d\tau$.
By the Gronwall's inequality for integrals (see \cite[p. 625]{Eva}), it follows that
$\phi(t)\leq C_1(1+C_2t\mbox{e}^{C_2t})$, for all $t\in[0,T^*)$.
Therefore,
$\|u_i(.,t)\|_{L^p}\leq C_1(1+C_2T^*\mbox{e}^{C_2T^*})$, for all $t\in[0,T^*)$, and from
the Fatou's lemma, we have $u_i(x,T^*)\in L^p$.
\ \ 
\vbox{\hrule height0.6pt\hbox{%
\vrule height1.3ex width0.6pt\hskip0.8ex
\vrule width0.6pt}\hrule height0.6pt}

\end{document}